\documentclass[12pt,longbibliography]{article}

\usepackage[utf8]{inputenc}

\usepackage[T1]{fontenc}

\usepackage[a4paper, margin=2.7cm]{geometry}

\usepackage{amsmath, amsthm, amsfonts, amssymb}
\usepackage{mathrsfs}           

\usepackage[colorlinks=true,linkcolor=blue,citecolor=blue,urlcolor=blue,breaklinks]{hyperref}

\usepackage{bbm} 
\usepackage{graphicx}

\usepackage{breakurl}
\usepackage{xcolor}
\usepackage{url}

\def\C{\mathbb{C}}
\def\F{\mathscr{F}}

\def\R{\mathbb{R}}

\DeclareMathOperator{\esssup}{ess\,sup}

\DeclareMathOperator*{\Res}{Res}

\renewcommand{\l}[1]{\widehat{#1}}

\def\d{\,\mathrm{d}}
\def\dv{\d v}

\def \ddk{\frac{\mathrm{d}}{\mathrm{d}k}}

\def\p{\partial}

\def\:{\colon}

\newtheorem{thm}{Theorem}[section]
\newtheorem{cor}[thm]{Corollary}
\newtheorem{lem}[thm]{Lemma}
\newtheorem{prp}[thm]{Proposition}

\theoremstyle{definition}
\newtheorem{dfn}[thm]{Definition}
\theoremstyle{remark}
\newtheorem{rem}[thm]{Remark}
\theoremstyle{example}

\numberwithin{equation}{section}

\def\thetitle{On the asymptotic behavior of the NNLIF neuron model for
  general connectivity strength} \def\theauthor{María J.~Cáceres \and
  José~A.~Cañizo \and Alejandro Ramos-Lora}

\hypersetup{pdftitle={\thetitle}}
\hypersetup{pdfauthor={\theauthor}}

\title{\thetitle}
\author{\theauthor}

\date{January 2024}

\AtEndDocument{\bigskip{\footnotesize
  \noindent
  \textsc{M.~J.~Cáceres. Departamento de Matemática Aplicada \& IMAG,
    Universidad de Granada.}
  \textit{E-mail address}: \texttt{\href{mailto:caceresg@ugr.es}{caceresg@ugr.es}}
  \par
  \addvspace{\medskipamount}
  \noindent
  \textsc{J.~A.~Cañizo. Departamento de Matemática Aplicada \& IMAG,
    Universidad de Granada.}
  \textit{E-mail address}: \texttt{\href{mailto:canizo@ugr.es}{canizo@ugr.es}}
  \par
  \addvspace{\medskipamount}
  \noindent
  \textsc{A.~Ramos-Lora. Departamento de Matemática Aplicada \& IMAG,
    Universidad de Granada.}
  \textit{E-mail address}: \texttt{\href{mailto:ramoslora@ugr.es}{ramoslora@ugr.es}}
}}

\begin{document}
	
\maketitle

\begin{abstract}
  We prove new results on the asymptotic behavior of the nonlinear
  integrate-and-fire neuron model. Among them, we give a criterion for
  the linearized stability or instability of equilibria, without
  restriction on the connectivity parameter, which provides a proof of
  stability or instability in some cases. In all cases, this criterion
  can be checked numerically, allowing us to give a full picture of
  the stable and unstable equilibria depending on the connectivity
  parameter $b$ and transmission delay $d$. We also give further
  spectral results on the associated linear equation, and use them to
  give improved results on the nonlinear stability of equilibria for
  weak connectivity, and on the link between linearized and nonlinear
  stability.
\end{abstract}

\tableofcontents

\section{Introduction and main results}
\label{sec:intro}

Simple mathematical models are often used to describe the activity of
populations of neurons, although the properties of their solutions are
frequently not rigorously known.  This is the case for nonlinear noisy
leaky integrate and fire (NNLIF) neuronal models, for which the
asymptotic behaviour of their solutions is not yet well understood,
even in their simplest version.  This article seeks to shed light in
that direction, especially on the stability or instability of their
equilibria.

NNLIF models describe the activity of a large number of neurons
(neuron networks), at the level of the membrane potential, which is
the potential difference across the neuronal membrane
\cite{lapicque,sirovich,omurtag,brunel2000dynamics,
  BrGe,brunel1999fast}.  Two different scales are usually considered
in models: microscopic, using stochastic differential equations (SDE)
\cite{delarue2015global,delarue2015particle,liu2022rigorous}, and
mesoscopic/macroscopic via mean-field Fokker-Planck-type equations
\cite{caceres2011analysis,carrillo2013classical,caceres2014beyond,
  caceres2017blow,
  caceres2018analysis,caceres2019global,hu2021structure,
  roux2021towards, ikeda2022theoretical}.  This paper is devoted to
one of the simplest PDE models in this family
\cite{brunel2000dynamics,delarue2015particle,delarue2015global,liu2022rigorous},
which arises as the mean-field limit of a large set of $\mathcal{N}$
identical neurons which are connected to each other in a network
\cite{brunel2000dynamics, liu2022rigorous}, as $\mathcal{N}\to\infty$.
It is given by the following nonlinear Fokker-Planck equation:
\begin{subequations}
  \label{eq:NNLIF}
  \begin{equation}
    \label{eq:NNLIF-PDE}
    \p_t p(v, t)
        +\p_v\left[\left(-v+bN_p(t-d)\right)p(v,t)\right]
    -\p^2_vp(v,t)=\delta_{V_R} N(t),
  \end{equation}
  where we always denote $\delta_{V_R}(v) \equiv \delta(v-V_R)$, the
  delta function at the point $v = V_R$. The unknown $p(v,t) \geq 0$
  is the probability density of finding neurons at voltage (or
  \emph{potential}) $v\in\left(-\infty,V_F\right]$ and time $t\geq 0$,
  so we are mainly interested in nonnegative solutions for this
  equation.  Each neuron spikes when its membrane voltage reaches the
  firing threshold value $V_F$, discharges immediately afterwards, and
  its membrane potential is restored to the reset value $V_R <
  V_F$. This resetting effect is described by the right hand side of
  \eqref{eq:NNLIF}, the boundary condition
  \begin{equation}
    \label{eq:NNLIF-bc}
    p(V_F,t) = 0, \qquad \text{for $t > 0$},
  \end{equation}
  and the definition of the firing rate $N(t)$ as the flux at $V_F$,
  \begin{equation}
    \label{eq:NNLIF-N}
    N_p(t):=-\p_v p(V_F,t)\ge 0.
  \end{equation}
  Moreover, there is a delay $d \geq 0$ in the synaptic transmission,
  which is included in the drift coefficient $-v+bN_p(t-d)$. Its effect
  is sometimes also included in a factor multiplying the diffusion
  term \cite{brunel2000dynamics,caceres2011analysis}, but we assume
  the diffusion coefficient to be $1$ for simplicity. Our
  techniques can easily be applied to a constant diffusion
  coefficient, but in this paper we always consider a coefficient $1$,
  as in \eqref{eq:NNLIF-PDE}. In order to have unique solutions, the
  system must include an initial condition of the form
  \begin{equation}
    \label{eq:NNLIF-u0}
    p(v,t) = p_0(v,t)
    \qquad \text{for $v \in (-\infty, V_F]$ and $t \in [-d,0]$,}
  \end{equation}
  where $p_0$ is a given function. We notice that in order to solve
  the PDE, the only initial conditions strictly needed are
  $p_0(\cdot,0)$ and $N_p(t)$ for $-d \leq t \leq 0$, but we write a
  full initial condition $p_0$ on $[-d,0]$ for ease of notation. The
  system \eqref{eq:NNLIF-PDE} is nonlinear since the firing rate
  $N_p(t)$ is computed by \eqref{eq:NNLIF-N}.
\end{subequations}
This system is the simplest of the family because all neurons are
assumed to be continuously active over time, and excitatory and
inhibitory neurons are not considered distinct populations. For more
realistic model with refractory states and separated populations of
excitatory and inhibitory neurons we refer to
\cite{brunel2000dynamics,caceres2018analysis}.  The key parameter of
system \eqref{eq:NNLIF} is the {\em connectivity parameter} $b$, which
gives us information on whether the network is on average excitatory
or inhibitory: if $b>0$ the network is average-excitatory; if $b<0$
the network is average-inhibitory; if $b=0$ the system is linear and
neurons are not connected to each other.

Given any nonnegative, integrable initial condition
$p(v,0)=p_0(v)\geq0$, the system \eqref{eq:NNLIF} satisfies the
conservation law
$\int_{-\infty}^{V_F}p(v,t) \dv = \int_{-\infty}^{V_F}p_0(v) \dv$ at
any time $t$ for which the solution is defined
\cite{carrillo2013classical,caceres2019global,roux2021towards} (we
always assume $\lim_{v \to -\infty}p(v, t)=0$, or other reasonable
conditions which ensure appropriate decay of solutions as
$v \to -\infty$).  For simplicity, and since this does not entail any
qualitative changes to the system, we will consider
$\int_{-\infty}^{V_F}p_0(v) \d v=1$ throughout this paper.

\medskip

The number of stationary solutions of equation \eqref{eq:NNLIF} and
their profiles are well understood \cite{caceres2011analysis}. The
\emph{probability steady states} (or simply steady states) of the
system \eqref{eq:NNLIF} are integrable, nonnegative solutions to the
following equation:
\begin{equation}
  \left\{
    \begin{array}{l}
      \p_v\left[\left(-v+bN_\infty\right)p_\infty(v)\right]-\p^2_v
      p_\infty(v)=\delta_{V_R} N_\infty,
      \\
      N_\infty=-\p_v p_\infty(V_F),
      \qquad 
      p_\infty(V_F) = 0,
      \\
      \text{and} \quad
      \int_{-\infty}^{V_F}p_\infty(v) \d v=1.
    \end{array}
  \right.
  \label{eq: large-delta-stationary}
\end{equation}
Interchangeably, they are also called \emph{equilibria} in the
literature, and we will use either of them as synonyms.
In this paper we always assume steady states to have
integral $1$.
These profiles are continuous in $(-\infty, V_F]$,
differentiable in $(-\infty, V_R) \cup (V_R, V_F]$, and they are given
by
\begin{equation}
  \label{eq:stationary-state}
  p_\infty(v) = N_\infty e^{-\frac{\left(v-bN_\infty\right)^2}{2}}
  \int_{\max(v, V_R)}^{V_F}e^{\frac{\left(w-bN_\infty\right)^2}{2}} \d w,
\end{equation}
where the \emph{stationary firing rate} $N_\infty$ is a solution of
the implicit equation $N_\infty I(N_\infty)=1$, with $ I(N) $ defined
as
\begin{equation}
  \label{eq:I}
  I(N):= \int_{-\infty}^{V_F}
  e^{-\frac{\left(v-bN\right)^2}{2}}
  \int_{\max(v, V_R)}^{V_F}e^{\frac{\left(w-bN\right)^2}{2}} \d w \d v.
\end{equation}
This implicit equation is obtained as a consequence of the
conservation of mass, i.e., the condition
$\int_{-\infty}^{V_F} p_\infty(v) \d v = 1$.  Hence the number of
equilibria of \eqref{eq:NNLIF} is the same as the number of solutions
to $N_\infty I(N_\infty)=1$, which depends on the connectivity parameter
$b$: for $b\le 0$ (inhibitory case) there is only one steady state;
for $b>0$ (excitatory case) there is only one if $b$ is small; there
are no steady states if $b>0$ is large; and there are at least two for
intermediate values of $b$. The previous results can be proved
rigorously \cite{caceres2011analysis}, and numerically it is clear
that there is a maximum of two stationary states for any $b$. Figure
\ref{fig:IN} shows the function $N \mapsto 1/I(N)$, for different
values of $b$. Each equilibrium of the nonlinear system
\eqref{eq:NNLIF} corresponds to a crossing of the function
$N \mapsto 1/I(N)$ with the diagonal. Figure \ref{fig:equilibria}
illustrates, for each $ b>0 $, the $ N_\infty $ values solving
$ N_\infty I(N_\infty)=1 $, that is, the values of $N_\infty$ at the
equilibria of \eqref{eq:NNLIF}.
\begin{figure}[h]
  \begin{center}
    \begin{minipage}[c]{0.49\linewidth}
      \begin{center}
        \includegraphics[width=\textwidth]{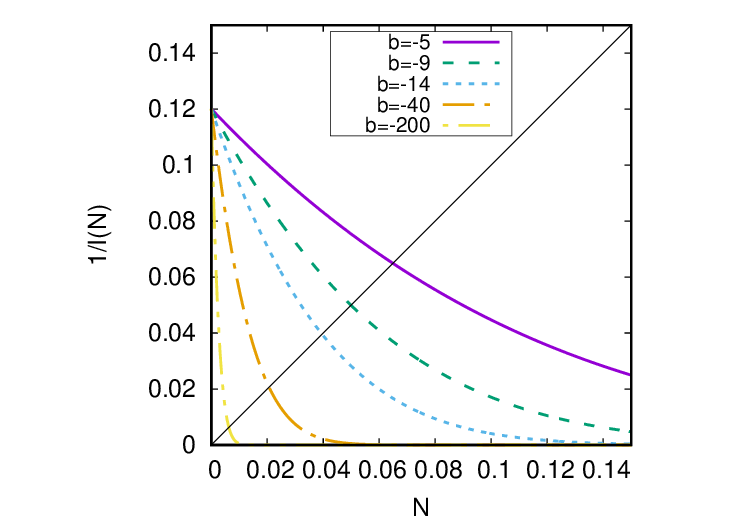}
      \end{center}
    \end{minipage}
    \begin{minipage}[c]{0.49\linewidth}
      \begin{center}
        \includegraphics[width=\textwidth]{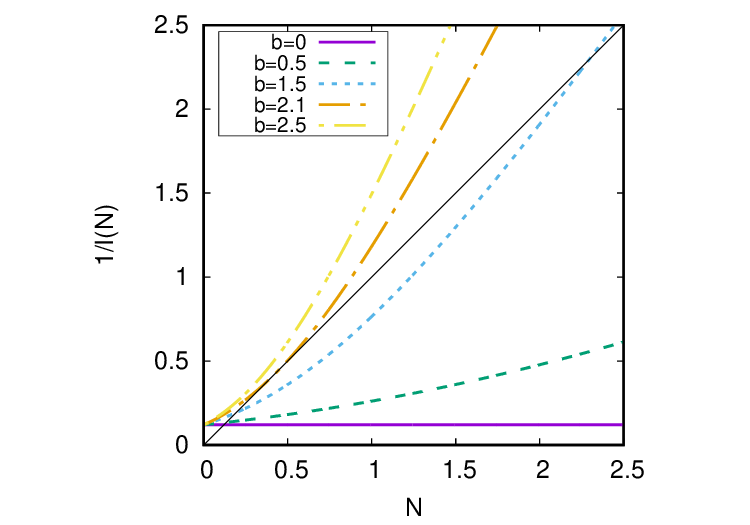}
      \end{center}
    \end{minipage}
  \end{center}
  \caption{{\bf Function $\boldsymbol{\frac{1}{I(N)}}$ (see
      \eqref{eq:I}) for different values of the connectivity parameter
      $\boldsymbol{b}$.}  Each crossing with the diagonal corresponds
    to an equilibrium of equation \eqref{eq:NNLIF}. Reset and firing potential are set to $ V_R=1 $ and $ V_F=2 $ respectively. }
  \label{fig:IN}
\end{figure}

\begin{figure}[h]
  \begin{center}
    \begin{minipage}[c]{0.49\linewidth}
      \begin{center}
        \includegraphics[width=\textwidth]{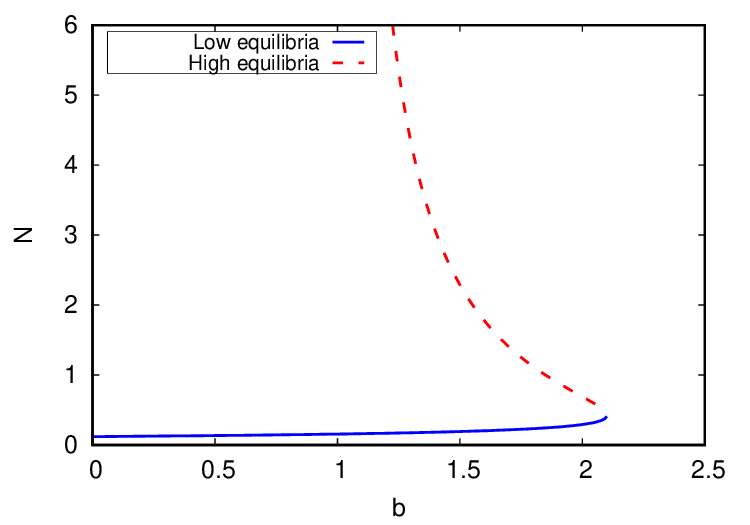}
      \end{center}
    \end{minipage}
  \end{center}
  \caption{{\bf Values of $ \boldsymbol{N} $ solving
      $ \boldsymbol{N(I(N))=1,} $ for each $ \boldsymbol{b} $.}  Point
    $ b=1 $ corresponds to $ b=V_F-V_R $ for the selected parameter
    values. Therefore, there is only one equilibrium for $ b<1 $, two
    equilibria for $1 < b < b_e$ with $b_e$ a certain critical value,
    and none for $b > b_e$. The dashed line has a vertical asymptote
    at $b=1$. Reset and firing potential are set to $ V_R=1 $ and $ V_F=2 $ respectively.}
  \label{fig:equilibria}
\end{figure}

\bigskip

Existence theory for NNLIF models has been developed in
\cite{carrillo2013classical,caceres2011analysis,
  carrillo2014qualitative,caceres2017blow,
  caceres2018analysis, caceres2019global,roux2021towards}.  In
\cite{carrillo2013classical} an existence criterion for the nonlinear
system \eqref{eq:NNLIF} without delay ($d=0$) was given: the solution
exists as long as the firing rate $N(t)$ remains finite. The authors
proved that for $d=0$ and $b<0$ (average-inhibitory case) solutions
are global in time, while for $b > 0$ a blow-up phenomenon (the firing rate $ N_p(t) $ diverges in finite time) may appear
if the initial condition is concentrated near the threshold potential
$V_F$ \cite{caceres2011analysis}, or if a large connectivity parameter
$b$ is considered \cite{roux2021towards}. In that case, solutions are
not globally defined for all times. The blow-up phenomenon disappears
if some transmission delay is taken into account, i.e.  if $d>0$,
leading to global existence as proved in
\cite{caceres2019global}. Blow-up may also be avoided by considering a
stochastic discharge potential \cite{caceres2014beyond}.

At the microscopic scale analogous criteria for existence and blow-up
phenomena were studied in \cite{delarue2015global}. For the
corresponding SDE the notion of solution was extended to the notion of
physical solutions in \cite{delarue2015particle}. Physical solutions
continue after system synchronization, this is after the blow-up
phenomenon, and therefore, they are global in time. In \cite{CR-L},
the particle system was analysed numerically to understand the meaning
of physical solutions for the Fokker-Planck equation
\eqref{eq:NNLIF}. However, the analytical study is still an open
problem.

\

This article is devoted to the study of the long-term behavior of
system \eqref{eq:NNLIF}. Here is an up-to-date summary of its expected
behavior, for which there is strong numerical evidence
\cite{brunel1999fast, caceres2011analysis, CR-L, caceres2018analysis,
  ikeda2022theoretical, pseudo-equilibria, hu2021structure}, but only
few proofs are available:
\begin{enumerate}
\item There is some $b_e > 0$ such that for $b > b_e$ no equilibria
  exist \cite{caceres2011analysis}. In this case, solutions without delay
  ($d=0$) are expected to always blow-up
  \cite{roux2021towards}. Solutions with positive delay $d > 0$ (which
  cannot blow-up) are expected to converge as $t \to +\infty$ to the
  \emph{plateau distribution} (the uniform distribution on
  $[V_R, V_F]$ \cite{CR-L}.
  
\item For $V_F-V_R < b < b_e$ two equilibria exist (in
  \cite{caceres2011analysis} it was proved that at least two
  exist). The higher equilibrium (the one with higher associated
  firing rate) is unstable for any delay $d \geq 0$ and the lower
  equilibrium (the one with lower firing rate) is asymptotically
  stable for any delay $d \geq 0$. Depending on the initial condition
  $p_0$, solutions either converge to the lower equilibrium, blow-up
  (only possible with $d=0$), or converge to the plateau distribution
  (only possible when $d > 0$).
  
\item For $0 < b < V_F-V_R$ the unique equilibrium is asymptotically
  stable regardless of delay. Solutions either converge to this unique
  equilibrium, or blow-up (only possible if $d=0$
  \cite{caceres2011analysis}).

\item There is a critical value $b_p< 0$ such that for any
  $b_p < b < 0$, the unique equilibrium is asymptotically stable for
  all values of the delay. All solutions are expected to converge
  to this unique equilibrium, independently of the delay.
  
\item For $b < b_p$, there is also a unique equilibrium.
  \begin{enumerate}
  \item For small values of the delay, this equilibrium is
    stable. All solutions are expected to converge towards it.
    
  \item For large values of the delay the equilibrium is
    unstable. All solutions are expected to approach a periodic
    solution.
  \end{enumerate}
\end{enumerate}
We emphasize that the above statements are on expected behavior, and
they are not proved in many cases. Let us also give a summary of rigorous
results on asymptotic behavior. In the linear case ($b=0$)
\cite{caceres2011analysis} proved that solutions converge
exponentially fast to the unique steady state. In the quasi-linear
case ($|b|$ small), the same is proved in
\cite{carrillo2014qualitative} if the initial condition is close
enough to the equilibrium.  These results were proved by means of the
entropy dissipation method, considering the relative entropy
functional between the solution $p$ and the stationary solution
$p_\infty$, given by
  \begin{equation*}
  H(p|p_\infty):=\int_{-\infty}^{V_F}
  \frac{\left(p(v,t)-p_\infty(v)\right)^2}{p_\infty(v)} \dv
  = \|p - p_\infty\|_{L^2(p_\infty^{-1})}.
\end{equation*}
It is also known rigorously that there are no periodic solutions if
$b$ is large enough, $V_F\leq 0$ and a transmission delay is
considered \cite{caceres2019global,roux2021towards}.

There is strong evidence for the existence of periodic solutions
(point 5(b)), according to the approximated model studied in
\cite{ikeda2022theoretical}, and these periodic solutions are clearly
seen numerically. In this paper we are able to map the region where
periodic solutions are expected; see Figure \ref{fig:map}. For the
complete model including different populations of excitatory and
inhibitory neurons, considering neurons with refractory periods, as well as the
stochastic discharge potential model, the numerical results in
\cite{caceres2018analysis,caceres2014beyond} also show periodic
solutions.

In a companion article to this one \cite{pseudo-equilibria}, we
numerically show the close relationship between a discrete sequence of
\emph{pseudo equilibria} and the long term behaviour of the nonlinear
system with large delays, with some further analytical results for
small $|b|$. This link gives further evidence for the expected
behavior described above in 1-5 (in the case of large delay $d$).

Applications of entropy methods have only given results for weakly
connected networks (small $b$) so far, i.e.  almost linear
systems. Building on these results, in this article we give further
properties of the linear operator and extend some of them to the
linearization of the PDE \eqref{eq:NNLIF}. We can thus obtain results
about its long term behavior using a different approach that does not
require the connectivity parameter $b$ to be small. Here is a summary
of the results we show in this paper for solutions to
\eqref{eq:NNLIF}:
\begin{enumerate}
\item We give a new proof that, for small $|b|$, the (unique)
  equilibrium is stable. Our proof works for a given small $|b|$, and
  any delay $d \geq 0$; this should be compared to
  \cite{caceres2019global}, where a condition on $d$ was needed. See Section
  \ref{sec:weakly-connected} for a precise statement.
  The proof we present here is also different from the one, mentioned above,
  that we proved in \cite{pseudo-equilibria} by pseudo equilibria sequence.
  
\item With a similar proof, in Section
  \ref{sec:linearized-to-nonlinear} we show that linearized stability
  of an equilibrium implies its nonlinear stability without delay,
  which is a new result as far as we know. We expect the same to be
  true also for any positive delay, but the proof runs into technical
  difficulties that we have not been able to overcome.
  
\item More importantly, we give explicit criteria to study whether a
  given equilibrium is linearly stable or not. These criteria are
  intimately related to the slope at which the curve
  $N \mapsto (I(N))^{-1}$ crosses the diagonal in Figure \ref{fig:IN}, which
  also determines the behaviour of the firing rate sequence,
  which provides the pseudo equilibria sequence (see
  \cite{pseudo-equilibria} for more details).
\end{enumerate}

Point 3 is in our opinion the most surprising one, since it gives a
rather complete picture of the asymptotic behavior of the PDE
\eqref{eq:NNLIF}. In order to state more precisely our theorem
regarding point 3 we need to introduce a few definitions. First, given
a fixed $N \geq 0$ and $b \in \R$ we define the linear operator
$L_{N}$, acting on functions $u = u(v)$, by
\begin{equation}
  \label{eq:linear-operator}
  L_{N} u := \p^2_v u + \p_v((v-bN) u)
  + \delta_{V_R} N_u.
\end{equation}
This operator is formally obtained by replacing $N_p$ on the nonlinear
term in the right hand side of \eqref{eq:NNLIF-PDE} by
$N$. The associated linear PDE,
\begin{equation}
  \label{eq:linear-intro}
  \p_t p = L_{N} p,
\end{equation}
with the same Dirichlet boundary condition as before and
$N_u(t) := -\p_v u(V_F,t)$, models a situation where neurons evolve
with a fixed background firing rate.  A central part of our strategy
is based on a careful study of the linear operator $L_{N}$. The PDE
\eqref{eq:linear-intro} has a unique (probability) stationary state
$p_\infty$, explicitly given by a very similar expression to
\eqref{eq:stationary-state}:
\begin{equation}
  \label{eq:stationary-state-linear}
  p_\infty^N(v) = \overline{N} e^{-\frac{\left(v-b N \right)^2}{2}}
  \int_{\max(v, V_R)}^{V_F}e^{\frac{\left(w-b N \right)^2}{2}} \d w,
\end{equation}
where $\overline{N}$ is a normalizing factor to ensure that $p_\infty^N$
is a probability distribution. It is also known that solutions to this
linear equation converge exponentially to equilibrium in the weighted
norm $L^2(1/p_\infty)$, which is a natural norm when considering a
quadratic entropy \cite{caceres2011analysis,
  carrillo2014qualitative}. By studying some well-posedness and
regularization bounds for the linear equation, we are able to show
that this exponential decay is also true in the smaller space $X$
given by
\begin{equation}
  \label{eq:space-X}
  X := \{ u \in \mathcal{C}(-\infty, V_F] \cap \mathcal{C}^1(-\infty, V_R]
  \cap \mathcal{C}^1[V_R, V_F] \mid u(V_F) = 0 \ \text{and} \ \|u\|_X < \infty\},
\end{equation}
where
\begin{equation}\label{eq:norm-X}
  \|u \|_X :=
  \|u \|_\infty + \|\p_vu \|_\infty
  + \|u \|_{L^2(\varphi)} + \|\p_vu \|_{L^2(\varphi)}
\end{equation}
and
\begin{equation*}
  \varphi(v)
  := \exp\left( \frac{(v - bN)^2}{2} \right),
  \qquad v \in (-\infty, V_F].
\end{equation*}
It is understood that the $L^2(\varphi)$ norms are taken on
$(-\infty, V_F]$. The linear space $X$ is a Banach space with the
above norm. We observe that $X \subseteq L^2((p_\infty^N)^{-1})$ since
for some constant $C > 0$
\begin{equation}
  \label{eq:4}
  \frac{1}{p_\infty^N(v)} = C \varphi(v)
  \qquad \text{for all $v \leq V_R$,}
\end{equation}
and for any $u \in X$ we have
\begin{multline}
  \label{eq:11}
  \int_{V_R}^{V_F} u(t,v)^2 p_\infty^N(v)^{-1} \d v
  \lesssim
  \int_{V_R}^{V_F} u(t,v)^2 \frac{1}{V_F - v} \d v
  \\
  \leq
  \| \p_v u(\cdot,t)\|_{\infty}^2 \int_{V_R}^{V_F} (V_F - v) \d v
  \lesssim \| u \|_X^2,
\end{multline}
where we have used the estimate
\begin{equation*}
  |u(t,v)| \leq \|\p_v u(t, \cdot)\|_\infty (V_F - v)
  \qquad \text{for $v \leq V_F$,}
\end{equation*}
easily obtained via the mean value theorem and the fact that
$u(V_F) = 0$. Inequalities \eqref{eq:4} and \eqref{eq:11} show that
$\| u \|_{L^2(p_\infty^{-1})} \lesssim \| u\|_X$, so
$X \subseteq L^2(p_\infty^{-1})$.

The property that the linear PDE decays exponentially to equilibrium
in a certain space is often stated by saying that it has a
\emph{spectral gap} in that space. The technique we use to ``shrink''
the space where a given linear PDE has a spectral gap was used in
\cite{Canizo2010a} in relation to a coagulation-fragmentation PDE, and
is linked to operator techniques described in \cite{Gualdani2018} and
\cite{Mischler2016}. A spectral gap in the space $X$ allows us to
carry out stability arguments in a more natural way, since now the
firing rate $N_p$, considered as an operator $-\p_v p(V_F)$ acting on
$p$, is a \emph{bounded} linear operator on $X$. This allows us to
relate the linearized and the nonlinear equations, and to give more
precise estimates on the range of parameters where stability or
instability take place.

Given $b \in \R$ and a probability equilibrium $p_\infty$ of the PDE
\eqref{eq:NNLIF}, we call $N_\infty := -\p_v p_\infty(V_F)$ the
equilibrium firing rate, and we define the \emph{linearized} equation
at $p_\infty$ by
\begin{align}
  \label{eq:linearized-pde}
  \p_t u
  = &\p^2_v u + \p_v((v- b N_\infty) u)
  + \delta_{V_R} N_u - b N_u(t-d) \p_v p_\infty
  \\
  \nonumber
  = &L_\infty u - b N_u(t-d) \p_v p_\infty,
\end{align}
where $L_\infty \equiv L_{N_\infty}$
is the linear operator from \eqref{eq:linear-operator}, with
$N = N_\infty$. We notice that this equation has a delay $d$ in the
last term. We add the same boundary condition as before, that is,
\begin{equation}
  \label{eq:linearized-pde-bc}
  u(V_F, t) = 0, \qquad \text{for $t > 0$.}
\end{equation}
\begin{dfn}
  Take $b \in \R$ and $d \geq 0$. We say that a probability
  equilibrium $p_\infty$ of system \eqref{eq:NNLIF} is \emph{linearly
    stable} if there exist $C \geq 1$ and $\lambda > 0$ such that all
  solutions $u$ to the linearized equation
  \eqref{eq:linearized-pde}--\eqref{eq:linearized-pde-bc} with an
  initial condition $u_0 \in \mathcal{C}([-d,0]; X)$ such that
  \begin{equation}
    \label{eq:u0-mass-0}
    \int_{-\infty}^0 u_0(v, t) \d v = 0
    \qquad
    \text{for all $t \in [-d, 0]$}
  \end{equation}
  satisfy
  \begin{equation}
    \label{eq:linearized-stability-def}
    \| u(\cdot,t) \|_X \leq C e^{-\lambda t}
    \sup_{\tau \in [-d,0]} \| u_0(\cdot, \tau) \|_X
    \qquad
    \text{for all $t \geq 0$.}
  \end{equation}
  We say that $p_\infty$ is \emph{linearly unstable} if this is not
  true for $\lambda = 0$ and any $C \geq 1$; that is, if for any
  $C \geq 1$ there exists $u_0$ satisfying \eqref{eq:u0-mass-0} and
  $t > 0$ such that
  \begin{equation*}
    \| u(\cdot,t) \|_X \geq C
    \sup_{\tau \in [-d,0]} \| u_0(\cdot, \tau) \|_X.
  \end{equation*}
\end{dfn}

For the rest of the paper, whenever $p_\infty$ is an equilibrium of
the nonlinear problem, we define
\begin{equation}
  \label{eq:q}
  q(v,t) := e^{t L_\infty} \p_v p_\infty,
\end{equation}
where $e^{t L_\infty}$ denotes the flow of the linear problem
\eqref{eq:linear-operator} (that is: $e^{t L_\infty} u_0$ is the
solution to problem \eqref{eq:linear-operator} with initial condition
$u_0$). In agreement with our previous notation we set
\begin{equation}
  \label{eq:Nq}
  N_q(t) := - \p_v q(V_F, t),
  \qquad \text{for $t > 0$,}
\end{equation}
and call $\hat{N}_q$ the Laplace transform of $N_q$,
\begin{equation}
  \label{eq:hatNq}
  \hat{N}_q(\xi) := \int_{0}^{\infty} e^{-\xi t} N_q(t) \d t.
\end{equation}
We prove in Section \ref{sec:linear} that
$|N_q(t)| \leq C e^{-\lambda t}$ for some $C, \lambda > 0$, and hence
$\hat{N}_q(\xi)$ is well defined for $\Re(\xi) > -\lambda$. Here is
the final theorem that we are able to prove in Section
\ref{sec:linearized-sg}:
\begin{thm}
  \label{thm:N-asymptotic-behavior-main}
  Take $b \in \R$ and $d \geq 0$, and let $p_\infty$ be a
  (probability) stationary state of system \eqref{eq:NNLIF}, and
  define $\hat{N}_q$ by \eqref{eq:q}--\eqref{eq:hatNq}. The steady
  state $p_\infty$ is linearly stable if and only if all zeros of the
  analytic function
  \begin{equation*}
    \Phi_d(\xi) := 1+b\hat{N_q}(\xi) \exp(-\xi d)
  \end{equation*}
  (defined for $\Re(\xi) > -\lambda$) are located on the real negative
  half-plane $\{ \xi \in \C \mid \Re(\xi) < 0\}$.
\end{thm}
This result simplifies the numerical study of the stability of
equilibria, and gives some theoretical consequences. For example, we
have:
\begin{cor}
  \label{cor:stability-d-small}
  Take $b \in \R$. If a probability stationary state $p_\infty$ 
  of system \eqref{eq:NNLIF} is linearly stable for $d=0$, then it is
  linearly stable also for small enough $d > 0$.
\end{cor}

\begin{proof}
  If $p_\infty$ is linearly stable for $d=0$, it means that for some
  $\mu > 0$, the function $\Phi_0$ does not have any zeros on
  $\C_{\mu} := \{ \xi \in \C \mid \Re(\xi) \geq -\mu \}$.
  By properties of the Laplace transform, $\hat{N_q}(\xi)$ tends to zero
  both when $\Im(\xi) \to \pm \infty$ and when $\Re(\xi) \to
  +\infty$. Hence all possible zeros of $\Phi_d$ in $\C_\mu$, for any
  $d \geq 0$, must be in some fixed compact subset
  $K \subseteq C_\mu$. Since the functions $\Phi_d$ are continuous and
  $\Phi_d(\xi) \to \Phi_0(\xi)$ uniformly as $d \to 0$, we see that for
  $d$ small, the function $\Phi_d$ cannot have any zeros on $K$, and
  hence cannot have zeros on $\C_\mu$.
\end{proof}
There is a specific value of $\hat{N}_q$ which is linked to the
function $I(N)$: in Lemma \ref{lem:ddtN} we show that
\begin{equation*}
  b \hat{N_q}(0) = \frac{\mathrm{d}}{\mathrm{d}N} \Big\vert_{N = N_\infty}
  \left( \frac{1}{I(N)} \right).
\end{equation*}
As a consequence, analyzing the zeros of $\Phi_d(\xi)$ gives us the
following criterion of sta\-bi\-li\-ty/ins\-ta\-bi\-li\-ty, often
easier to check:
\begin{thm}
  \label{thm:stability-main}
  Let us set $V_R < V_F \in \R$, $b \in \R$, and let $p_\infty$ be a
  probability equilibrium of the nonlinear equation
  \eqref{eq:NNLIF}. Call $N_\infty := -\p_v p_\infty(V_F)$ its
  associated firing rate.
  \begin{enumerate}
  \item If
    \begin{equation*}
      \frac{\mathrm{d}}{\mathrm{d}N} \Big\vert_{N = N_\infty}
      \left( \frac{1}{I(N)} \right)
      > 1
    \end{equation*}
    then $p_\infty$ is a linearly unstable equilibrium, for any delay
    $d \geq 0$.
    
  \item If
    \begin{equation*}
      |b|\int_0^\infty  \vert N_q (t)\vert \d t <1
    \end{equation*}
    then the equilibrium $p_\infty$ is linearly stable, for any delay
    $d \geq 0$.

  \item If
    \begin{equation}
      \label{eq:conditional-instability-condition-intro}
      \frac{\mathrm{d}}{\mathrm{d}N} \Big\vert_{N = N_\infty}
      \left( \frac{1}{I(N)} \right)
      < -1
    \end{equation}
    then the equilibrium $p_\infty$ is linearly unstable when the
    delay $d$ is large enough.
  \end{enumerate}
\end{thm}
This criterion is also in agreement with the prediction of the pseudo
equilibria sequence given in \cite{pseudo-equilibria} for networks
with large transmission delay.  It strikes us that the criterion is
given in terms of the slope of $\frac{1}{I(N)}$ (see \eqref{eq:I} for
an explicit expression of $I$), which determines the behaviour of the
discrete sequence of firing rates studied in \cite{pseudo-equilibria}.

\medskip The function $1/I(N)$ is explicit (see \eqref{eq:I}) and the
only obstacle to checking analytically whether the conditions in
Theorem \ref{thm:stability-main} hold is that the expression of $I(N)$
is cumbersome to work with. Some rigorous properties of $I$ are given
in \cite{caceres2011analysis, pseudo-equilibria}, but in any case the
following is clearly seen numerically (Figure \ref{fig:IN}):
\begin{itemize}
\item Case 1 in Theorem \ref{thm:stability-main} holds in the
  excitatory case $b > 0$, assuming there are two equilibria. The
  higher equilibrium satisfies
  $ \frac{\mathrm{d}}{\mathrm{d}N} \Big\vert_{N = N_\infty} \left(
    \frac{1}{I(N)} \right)> 1 $ (with $N_\infty$ the highest
  stationary firing rate). Hence assuming there are two equilibria,
  the higher equilibrium is unstable.

\item Case 2 in Theorem \ref{thm:stability-main} seems hard to check
  analytically, though numerical checks are straightforward. This case
  takes place for the lower equilibrium if there are two of them, and
  for the single equilibrium whenever there is a unique one.
  
\item Case 3 in Theorem \ref{thm:stability-main} can be rigorously
  proved to hold when $b < 0$ is sufficiently negative. Hence there is
  some $b^* < 0$ such that for $b < b^*$ and large enough $d > 0$, the
  unique equilibrium $p_\infty$ is linearly unstable.
\end{itemize}

\begin{figure}[h]
  \begin{center}
    \includegraphics[width=0.49\linewidth]{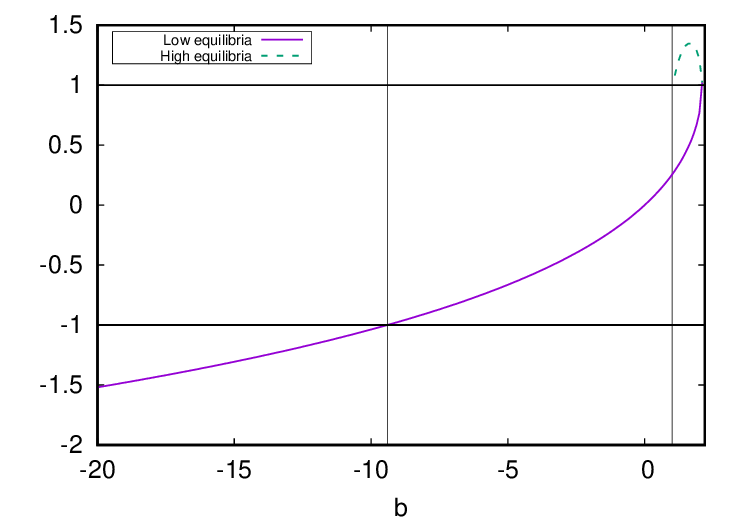}
  \end{center}
  \caption{{\bf Value of
      $S(b) := - b \int_{0}^{\infty} N_q(t) \d t =
      \frac{\mathrm{d}}{\mathrm{d}N} \Big\vert_{N = N_\infty} \left(
        \frac{1}{I(N)} \right)$ for different values of $b$.} Vertical
    lines at $ b=-9.4 $ and $ b=1 $ separate regions with different
    behaviours in terms of stability for equilibria of equation
    \eqref{eq:linearized}: for $ b<-9.4 $ (one equilibrium) we find
    that $S(b) < -1$, which coincides with point 3 of Theorem
    \ref{thm:stability-main}, i.e. the region of $ b $ where
    equilibria stability depends on the delay value; in the case
    $ -9.4<b<1 $ (one equilibrium), we get
    $|S(b)| = |b| \int_0^\infty |N_q(t)| \d t < 1$ (since the function
    $N_q(t)$ is numerically seen to be always monotone), leading us to
    the scenario in point 2 of Theorem \ref{thm:stability-main}, where
    regardless of the delay value, the equilibrium is stable. For any
    $ b>1 $ there are two equilibria: for the lower equilibrium (solid
    line) the situation is the same as in the previous case, while for
    the higher equilibrium (dashed line), $S(b) > 1 $ and point 1 of
    Theorem \ref{thm:stability-main} is fulfilled. Reset and firing potential are set to $ V_R=1 $ and $ V_F=2 $ respectively.}
  \label{fig:integral-h}
\end{figure}

We note that Theorem \ref{thm:stability-main} may not be fully
exhaustive. First, because it does not speak about the critical cases
where
$\frac{\mathrm{d}}{\mathrm{d}N} \Big\vert_{N = N_\infty} (1/I(N)) =
\pm 1$. And second, even ignoring these cases, it may still happen
that none of its three possibilities are satisfied. However, from our
simulations we suspect that the criterion is in fact exhaustive except
at the critical cases, since we expect $N_q(t)$ to have a fixed sign
for all $t > 0$ (which seems clear numerically), and then
\begin{equation*}
|b|\int_0^\infty \vert N_q (t)\vert \d t =\left\vert
  \frac{\mathrm{d}}{\mathrm{d}N} \Big\vert_{N = N_\infty} \left(
    \frac{1}{I(N)} \right) \right\vert.
\end{equation*}
On the other hand, point 3 of Theorem \ref{thm:stability-main} does
not give information on the case of strongly inhibitory systems with
small delay. For those cases we must study the zeros of $\Phi_d$, as
indicated by Theorem \ref{thm:N-asymptotic-behavior-main}, and the
best we can say is Corollary \ref{cor:stability-d-small}.

The stability criteria in Theorem \ref{thm:N-asymptotic-behavior-main}
and \ref{thm:stability-main} can be checked numerically and allow us
to give a rather complete picture of the stability of equilibria,
including the threshold delay for which the equilibrium becomes
unstable, as shown in Figure \ref{fig:map}.
\begin{figure}[h]
  \begin{center}
    \includegraphics[width=0.49\linewidth]{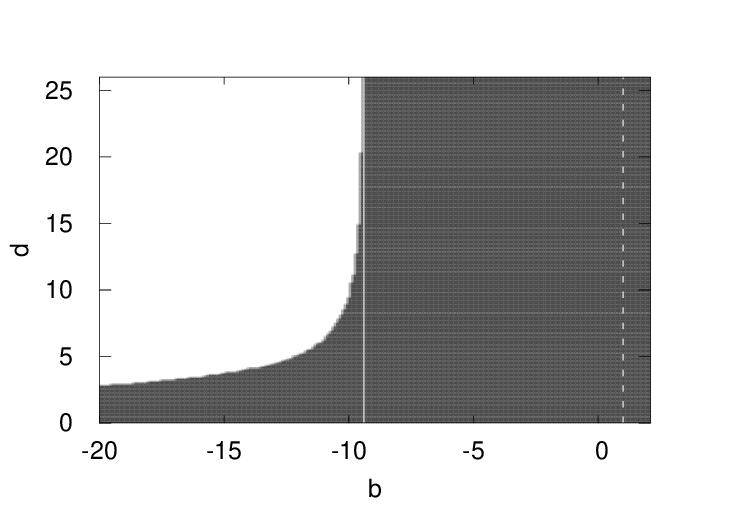}
  \end{center}
  \caption{{\bf Stability map of the linearized equation.} For each
    pair of values of the connectivity parameter and transmission
    delay $ (b,d) $, this color map shows the stability (dark) or
    instability (white) for the solutions of the linearized equation
    \eqref{eq:linearized-pde}. The vertical black solid line placed at
    $ b_p \approx-9.4 $ separates the region of unconditional
    stability (for any delay $d$), from that where stability depends
    on the value of $d$. The vertical dashed line at $ b=1 $ separates
    the region with a unique steady state (left) from the one with two
    steady states (right). For cases with two steady states (that is,
    to the right of the vertical dashed line) we consider the lower
    one. The stability condition considered to make this map is based
    on Theorems \ref{thm:N-asymptotic-behavior-main} and
    \ref{thm:stability-main}: for each $ b $ we find its related
    equilibrium $ p_\infty $, choosing the lower one if there are two
    equilibria; we differentiate $ p_\infty $ and use it as initial
    condition for the linear system given by equation
    \eqref{eq:NNLIF-PDE}, thus obtaining
    $q(v,t) = e^{tL_\infty}\p_v p_\infty$ and $N_q(t)$. Cases 1 and 2
    of Theorem \ref{thm:stability-main} account for all cases with
    $b > b_p$. In order to tell apart the cases with $b < b_p$, for
    each $ d>0 $ we compute
    $ Q(k):=-b e^{-ikd}\int_{0}^{\infty}N_q(t)e^{ikt} \d t$, with
    $ k\in\R $; finally we count the number of times that $ Q(k) $
    crosses the line $ (1,\infty) $ taking into account its
    orientation. If the net balance of crossings is different from
    $ 0 $, the equilibrium $ p_\infty $ related to $ (b,d) $ is
    unstable (see the proof of Theorem \ref{thm:stability-main} in
    Section \ref{sec:linearized-sg} for a better understanding of
    this). Otherwise we need to compute the zeros of function
    $ \Phi_d(\xi) := 1+b\hat{N_q}(\xi) \exp(-\xi d) $, given in
    Theorem \ref{thm:N-asymptotic-behavior-main} in order to check
    whether the equilibrium is stable or not. Reset and firing potential are set to $ V_R=1 $ and $ V_F=2 $ respectively.}
  \label{fig:map}
\end{figure}
Though it is not contained in Theorem \ref{thm:stability-main}, we
point out that the expected behavior of the nonlinear system
\eqref{eq:NNLIF} when $(b,d)$ is in the unstable region with $b < 0$
should be periodic, as shown consistently in simulations
\cite{ikeda2022theoretical,pseudo-equilibria}. That is, we expect
solutions to converge to a unique (up to time translations) periodic
solution with period approximately equal to $2d$. Since the analysis
leading to Theorem \ref{thm:stability-main} is based on the linearized
equation, we are not able to show a global behavior such as periodic
solutions, but the mechanism for their appearance seems to be given by
the presence of delayed negative feedback: in the literature on delay
equations there are many instances of periodic behavior, arising due
to a delayed system in which the driving force works against the
displacement that the system has seen in the previous delay
period. See for example the books \cite{Diekmann1995} (especially
chapter XV) or \cite{Erneux2009} for an exposition of this topic.

\

Before concluding this introduction it is worth mentioning that other
PDE families model the activity of a large number of neurons based on
the integrate and fire approach: population density models of
integrate and fire neurons with jumps \cite{omurtag, henry13, henry12,
  dumont2020mean}; Fokker-Planck equations for uncoupled neurons
\cite{NKKRC10,NKKZRC10}; and Fokker-Planck equations including
conductance variables \cite{Caikinetic,CCTa, perthame2013voltage,
  perthame2019derivation} There are also closely related models, not
strictly of NNLIF type, such as time elapsed models
\cite{PPD,PPD2,pakdaman2014adaptation} which are derived as mean-field
limits of Hawkes processes
\cite{chevallier2015microscopic,chevallier2015mean}; McKean-Vlasov
equations \cite{acebron2004noisy, mischler2016kinetic}, which are the
mean-field equations associated with a large number of Fitzhugh-Nagumo
neurons \cite{fitzhugh1961impulses}; Nonlinear Fokker-Planck systems
which describe networks of grid cells
\cite{carrillo2023noise,carrillo2023well} and are rigorously achieved
by mean-field limits \cite{carrillo2023mean}.

\

The paper is structured as follows: in Section \ref{sec:linear} we
carry out our study of the linear equation, giving well-posedness and
regularization estimates in the space $X$, and showing that the linear
PDE has a spectral gap in the space $X$. In Section \ref{sec:volterra}
we lay out some groundwork on classical delay equations, needed for
the arguments presented in later sections. In Section \ref{sec:nonlinear} we
give a result on the nonlinear stability of the NNLIF model for weak
interconnectivity, and we show that linearized stability implies
nonlinear stability (that is, that linearized behavior dominates close
to an equilibrium). Finally, in Section \ref{sec:linearized-sg} we give a
new criterion for linear stability or instability of equilibria with a
given connectivity $b$ and delay $d$, which implies corresponding
results for the nonlinear equation, due to the results in Section
\ref{sec:nonlinear}.

Throughout the paper we often use $C$ to denote an arbitrary
nonnegative constant which may change from line to line. We also use
the symbol $\lesssim$ to denote inequality up to a positive
multiplicative constant.

\section{Study of the linear equation}
\label{sec:linear}

In this section we study several aspects of the following linear equation, which
results when we fix the firing rate $N$, independent on time, in the nonlinear term of
equation \eqref{eq:NNLIF}:
\begin{subequations}
  \label{eq:linear-PDE-all}
  \begin{equation}
    \label{eq:linear-PDE}
    \p_t u(v,t) =  \p_v^2 u(v,t) + \p_v [(v - bN) u(v,t)] + \delta_{V_R} N_u(t),
    \
    v\in\left(-\infty,V_F\right], \ t\ge0,
  \end{equation}
  with the boundary condition
  \begin{equation}
    \label{eq:linear-PDE-bc}
    u(V_F, t) = 0,
    t \geq 0
  \end{equation}
  and defining
  \begin{equation}
    \label{eq:linear-PDE-N}
    N_u(t) := -\p_vu(V_F,t).
  \end{equation}
  As usual, we also impose an initial condition at time $t=0$:
  \begin{equation}
    \label{eq:linear-PDE-initial}
    u(v,0) = u_0(v),
    \qquad v \in (-\infty, V_F].
  \end{equation}
\end{subequations}
We will first study well-posedness bounds in several spaces, and then
study its spectral gap properties in suitable ``strong'' spaces which
will later be useful for the nonlinear theory.

\subsection{Summary of results: well-posedness, regularity and
  spectral gap}
\label{sec:wellposed}

Let us state our main results on the linear equation, leaving the
proof of all technical estimates for Section \ref{sec:regularity}. We
recall that an existence theory for classical solutions to the
nonlinear system \eqref{eq:NNLIF} was given in
\cite{carrillo2013classical}, which covers in particular the linear
case \eqref{eq:linear-PDE} (as the case with $b=0$ in their paper).
Throughout this section we assume this existence theory, and we
develop some more precise well-posedness estimates on these classical
solutions. Actually, from our estimates below one could construct a
linear existence theory for equation \eqref{eq:linear-PDE} by using
known tools from semigroup theory (see \cite{Engel1999}, in particular
Section III.3 and the Desch-Schappacher theorem). We do not follow
this path here since we think it is more direct to work with the
already studied classical solutions, developing new estimates for
them.

We also know from previous results \cite{caceres2011analysis} that
equation \eqref{eq:linear-PDE-all}, in the linear case, has a unique
probability equilibrium $p_\infty^N$ (given by
\eqref{eq:stationary-state-linear}), which in this section we denote
simply by $p_\infty$, since we work only with the linear equation in
all of Section \ref{sec:linear}. This equilibrium is positive, and the
linear equation has a spectral gap in the space $L^2(p_\infty^{-1})$:
there exists $\lambda > 0$ such that for all probability distributions
$p_0 \in L^2(p_\infty^{-1})$,
\begin{equation}
  \label{eq:gap-L2}
  \| e^{t L} p_0 - p_\infty \|_{L^2(p_\infty^{-1})}
  \leq e^{-\lambda t} \| p_0 - p_\infty \|_{L^2(p_\infty^{-1})}
  \qquad \text{for all $t \geq 0$.}
\end{equation}
In fact, the same reasoning from \cite{caceres2011analysis}, based on
a Poincaré's inequality, actually shows something slightly stronger: for
any initial condition $u_0 \in L^2(p_\infty^{-1})$ such that
$\int_{-\infty}^{V_F} u_0(v) \d v = 0$ we have
\begin{equation}
  \label{eq:gap-L2-int0}
  \| e^{t L} u_0 \|_{L^2(p_\infty^{-1})}
  \leq e^{-\lambda t} \| u_0 \|_{L^2(p_\infty^{-1})}
  \qquad \text{for all $t \geq 0$.}
\end{equation}
Another straightforward consequence is that there exists $C \geq 1$
such that for any $u_0 \in L^2(p_\infty^{-1})$ (not necessarily with
integral $0$) we have
\begin{equation}
  \label{eq:wellposed-L2-entropy}
  \| e^{t L} u_0 \|_{L^2(p_\infty^{-1})}
  \leq C \| u_0 \|_{L^2(p_\infty^{-1})}
  \qquad \text{for all $t \geq 0$.}
\end{equation}
Due to \eqref{eq:gap-L2-int0}, this is clearly true (with $C = 1$) if
$u_0$ has integral $0$. For a general $u_0 \in L^2(p_\infty^{-1})$,
\eqref{eq:wellposed-L2-entropy} may be obtained by writing
$u_0 = (u_0 - m p_\infty) + m p_\infty$, with
$m := \int_{-\infty}^{V_F} u_0(v) \d v$, applying
\eqref{eq:gap-L2-int0}, and noticing that
$|m| \leq C \|u_0\|_{L^2(p_\infty^{-1})}$ for some $C > 0$:
\begin{multline*}
  \| e^{t L} u_0 \|_{L^2(p_\infty^{-1})}
  \leq
  \| e^{t L} (u_0 - m p_\infty) \|_{L^2(p_\infty^{-1})}
  +
  |m| \| e^{t L} p_\infty \|_{L^2(p_\infty^{-1})}
  \\
  \leq
  e^{-\lambda t} \| u_0 - m p_\infty \|_{L^2(p_\infty^{-1})}
  + |m| \| p_\infty \|_{L^2(p_\infty^{-1})}
  \\
  \leq
  \| u_0 \|_{L^2(p_\infty^{-1})} + |m| \| p_\infty \|_{L^2(p_\infty^{-1})}
  + |m| \| p_\infty \|_{L^2(p_\infty^{-1})}
  \\
  \leq (1 + 2 C \| p_\infty \|_{L^2(p_\infty^{-1})}) \| u_0 \|_{L^2(p_\infty^{-1})}.
\end{multline*}
Our main aim is to have good estimates, similar to
\eqref{eq:gap-L2-int0} and \eqref{eq:wellposed-L2-entropy}, in a
linear space where the derivative of the solution is bounded (unlike
$L^2(p_\infty^{-1})$). We consider the space $X$ defined in
\eqref{eq:space-X} (associated to the unique probability equilibrium
$p_\infty$ of \eqref{eq:linear-PDE-all}). Here are the main estimates
we wish to prove on this:

\begin{lem}[Well-posedness estimate]
  \label{lem:wp}
  Let $u_0 \in X$ be any initial condition. Then the solution $u$ to
  the linear equation \eqref{eq:linear-PDE-all} satisfies the
  following: there are constants $M, C > 0$ such that
  \begin{gather}
    \label{eq:wp}
    \| e^{t L} u_0\|_X \leq M e^{C t} \|u_0\|_X
    \qquad \text{for all $t\ge0$.}
  \end{gather}
\end{lem}

And here is our main regularization estimate:

\begin{lem}[Regularization estimate]
  \label{lem:regularity}
  Let $u_0 \in X$ be any initial condition. Then the solution $u$ to
  the linear equation \eqref{eq:linear-PDE-all} satisfies the
  following: there exists a constant $C > 0$ and $t_0> 0$ (both
  independent of the initial condition $u_0$), such that
  \begin{equation}
    \label{eq:regularity1}
    \| e^{t L} u_0 \|_{X} \leq C t^{-3/4} \|u_0 \|_{L^2(\varphi)}
    \qquad \text{for all $0 < t \leq t_0$.}
  \end{equation}
  and in particular
  \begin{equation}
    \label{eq:regularity}
    \| e^{t L} u_0 \|_{X} \leq C t^{-3/4} \|u_0 \|_{L^2(p_\infty^{-1})}
    \qquad \text{for all $0 < t \leq t_0$.}
  \end{equation}
\end{lem}
Both of the lemmas above are a direct consequence of the bounds in
Lemmas \ref{lem:bounds-complete} and
\ref{lem:LIF-Gaussian-L2-estimates}, proved at the end of Section
\ref{sec:regularity}. Notice that
\begin{equation*}
  \varphi(v) \lesssim \frac{1}{p_\infty(v)}
  \qquad \text{for all $v < V_F$},
\end{equation*}
so \eqref{eq:regularity} is clearly a consequence of
\eqref{eq:regularity1}.

Regarding spectral gap estimates, we show that the semigroup
associated to the linear equation has a spectral gap in the space
$X$. This is essentially deduced from the known spectral gap estimate
\eqref{eq:gap-L2-int0} and Lemmas \ref{lem:wp} and
\ref{lem:regularity}:

\begin{prp}[Spectral gap estimate]
  \label{prp:gap-X}
  There exist constants $\lambda > 0$, $C \geq 1$ such that for all
  initial conditions $u_0 \in X$ such that
  $\int_{-\infty}^{V_F} u_0(v) \d v = 0$ it holds that
  \begin{equation}
    \label{eq:gap-X}
    \| e^{t L} u_0 \|_{X}
    \leq Ce^{-\lambda t} \| u_0 \|_{X}
    \qquad \text{for all $t \geq 0$}
  \end{equation}
  and also
  \begin{equation}
    \label{eq:gap-X-regularized}
    \| e^{t L} u_0 \|_{X}
    \leq C t^{-\frac34} e^{-\lambda t} \| u_0 \|_{L^2(\varphi)}
    \qquad \text{for all $t > 0$.}
  \end{equation}
\end{prp}

\begin{rem}
  From the proof below one can see that the constant $\lambda$ in
  \eqref{eq:gap-X} can be taken to be the same one as in
  \eqref{eq:gap-L2} (usually obtained via a Poincaré's inequality
  \cite{caceres2011analysis}). The constant $\lambda$ in
  \eqref{eq:gap-X-regularized} can be taken to be any constant
  strictly smaller than the $\lambda$ in \eqref{eq:gap-L2} (or
  alternatively, one could keep the same $\lambda$ and write $\max\{1,
  t^{-3/4}\}$ instead of $t^{-3/4}$ in \eqref{eq:gap-X-regularized}).
\end{rem}

\begin{proof}
  Call $Y \equiv L^2(p_\infty^{-1})$. Our regularization result
  in equation \eqref{eq:regularity} gives
  \begin{equation*}
    \| e^{tL} u_0 \|_{X} \lesssim t^{-\frac34} \|u_0\|_Y
    \qquad \text{for all $0 < t \leq t_0$.}
  \end{equation*}
  Now we can prove \eqref{eq:gap-X} for $t \geq t_0$ by using
  \eqref{eq:gap-L2-int0}: for $t \geq t_0$ we have
  \begin{multline*}
    \| e^{tL} u_0 \|_X = \| e^{t_0 L} e^{(t-t_0) L} u_0 \|_X
    \lesssim
    t_0^{-\frac34} \| e^{(t-t_0) L} u_0 \|_Y
    \\
    \lesssim
    t_0^{-\frac34} e^{-\lambda(t-t_0)} \| u_0 \|_Y
    \lesssim
    t_0^{-\frac34} e^{-\lambda(t-t_0)} \| u_0 \|_X
    \lesssim
    e^{-\lambda t} \| u_0 \|_X.
  \end{multline*}
  Notice that we also used $\|u_0\|_Y \lesssim \|u_0\|_X$, since
  $X \subseteq Y$ (see the observation after \eqref{eq:norm-X} in the
  introduction). For $t \leq t_0$ we can use the well-posedness
  estimate of Lemma \ref{lem:wp}, finally getting \eqref{eq:gap-X} for
  all $t \geq 0$.

  In order to show \eqref{eq:gap-X-regularized} we take any $t > 0$
  and any $0 < s < \min\{t_0, t\}$, and we use \eqref{eq:gap-X} and
  \eqref{eq:regularity1} to write
  \begin{equation*}
    \| e^{tL} u_0 \|_X = \| e^{(t-s) L} e^{s L} u_0 \|_X
    \lesssim
    e^{-\lambda(t-s)} \| e^{s L} u_0 \|_X
    \lesssim
    s^{-\frac34} e^{-\lambda(t-s)} \| u_0 \|_{L^2(\varphi)}.
  \end{equation*}
  For $t \geq t_0$ we take $s := t_0/2$ and this shows
  \eqref{eq:gap-X-regularized} (with a slightly smaller
  $\lambda$). For $t < t_0$ we take $s = t/2$, and this also gives the
  result on that range.  
\end{proof}

\subsection{Bounds for classical solutions}
\label{sec:regularity}

In order to study the linear equation \eqref{eq:linear-PDE} we
separate its terms as follows
\begin{equation}
  \label{eq:linear-PDE-2}
  \p_t u
  = \p_v^2 u + \p_v [(v - bN) u] + \delta_{V_R} N_u(t)
  = Fu  + L_1 u + L_2 u,
\end{equation}
with
\begin{equation*}
  F u := \p_v^2 u + \p_v [(v-V_F) u],
  \qquad
  L_1 u := \delta_{V_R} N_u(t),
  \qquad
  L_2 u := \p_v [(V_F- bN) u].
\end{equation*}
We still consider the same boundary conditions
\eqref{eq:linear-PDE-bc}.  In order, we prove estimates for the
equation $\p_t u = F u$ in Section \ref{sec:bounds-F-P}, then for
$\p_t u = F u + L_1 u$ in Section \ref{sec:plus_delta}, and finally
for the full linear equation in Section \ref{sec:complete}.

\subsubsection{Bounds for the Fokker-Planck equation}
\label{sec:bounds-F-P}

The Cauchy problem for the Dirichlet Fokker-Planck equation on the
half-line, posed for $t \geq 0$ and $v \in (-\infty,0]$,
\begin{equation}
  \label{eq:FP-0}
  \p_t u = \p_v^2 u + \p_v (v u)
  \qquad
  \text{with}
  \qquad
  u(0, t) = 0
  \quad \text{for $t \geq 0$},
\end{equation}
$$
 u(v,0)=u_0    \qquad
  \text{with}
  \qquad
u_0 \: (-\infty, 0] \to \R,
$$
has the explicit solution
\begin{equation}
  \label{eq:FP-sol}
  u(v, t)
  =
  \int_{-\infty}^{\infty} \widetilde{u}_0(w) \Phi_t(v, w) \d w
  =
  \int_{-\infty}^{0} u_0(w) \Psi_t(v,w) \d w,
\end{equation}
defined for $t > 0$ and $v \in (-\infty, 0]$. Here
$\widetilde{u}_0 \: \R \to \R$ is the anti-symmetrization of the
initial condition
\begin{equation*}
  \widetilde{u}_0(x) :=
  \begin{cases}
    u_0(x), \qquad & x\le0
    \\
    -u_0(-x), \qquad & x>0,
  \end{cases}
\end{equation*}
$\Phi_t = \Phi_t(v,w)$ is the fundamental solution of the standard
Fokker-Planck equation:
\begin{equation}\label{eq:fp-kernel}
  \Phi_{t}(v,w) =
  \frac{1}{\sqrt{2\pi (1-e^{-2t})}}
  \exp\left(
    {-\frac{( v-e^{-t}w )^2}{2\left(1-e^{-2t}\right)}}
  \right),
  \qquad t > 0, \quad v,w \in \R,
\end{equation}
and
\begin{equation}
  \label{eq:Psi-kernel}
  \Psi_t(v,w) := \Phi_t(v,w) - \Phi_t(v,-w),
  \qquad
  t > 0, \ v, w \in (-\infty, 0].
\end{equation}
For later reference, we notice that $\Psi_t(v,w)\ge 0$, since 
$t>0$ and $v, w \in (-\infty, 0]$, and 
\begin{equation}
  \label{eq:fp-kernel-derivative}
  \p_v \Phi_t(v,w) =
  (2\pi)^{-\frac12} (1- e^{-2t})^{-\frac32} (v-e^{-t} w)
  \exp\left(
    {-\frac{( v-e^{-t}w )^2}{2\left(1-e^{-2t}\right)}}
  \right).
\end{equation}
For an explicit solution to be available it is essential that the
boundary condition is set at $v=0$; otherwise we do not know of an
explicit expression. From the expression \eqref{eq:FP-sol} of the
solution, it is easy to see that many estimates for the standard
Fokker-Planck equation on $\R$ can be carried over to the
Fokker-Planck equation \eqref{eq:FP-0}, posed on $(-\infty,0]$ and
with a Dirichlet boundary condition.

It is then straightforward to translate equation \eqref{eq:FP-sol} and
obtain a solution to the translated Fokker-Planck equation
\begin{equation}
  \label{eq:FP-1}
  \p_t u
  = F u = \p_v^2 u + \p_v [(v - V_F) u]
\end{equation}
posed for $t \geq 0$ and $v \in (-\infty, V_F]$ with the boundary condition
\begin{equation}
  \label{eq:FP-1-bc}
  u(V_F, t) = 0,
  \qquad t \geq 0.
\end{equation}
A solution to \eqref{eq:FP-1}--\eqref{eq:FP-1-bc}, with
initial condition $u_0$, is then given by
\begin{equation}
  \label{eq:FP-sol-2}
  u(v, t)
  =
  \int_{-\infty}^{\infty} \widetilde{u}_0(w) \Phi_t(v-V_F, w-V_F) \d w
  =
  \int_{-\infty}^{V_F} u_0(w) \Psi_t(v-V_F,w-V_F) \d w,
\end{equation}
for $t > 0$ and $v \in (-\infty, V_F]$, where $\widetilde{u}_0$ is now
the function
\begin{equation*}
  \widetilde{u}_0(x) :=
  \begin{cases}
    u_0(x), \qquad & x \leq V_F
    \\
    -u_0(2V_F -x), \qquad & x > V_F.
  \end{cases}
\end{equation*}
With this explicit expression we can obtain the following estimates:

\begin{lem}
  \label{lem:FP-bounds}
  Denote by $e^{tF} u_0$ the solution \eqref{eq:FP-sol-2} to the
  Fokker-Planck equation \eqref{eq:FP-1}--\eqref{eq:FP-1-bc} in the
  domain $v\in (-\infty,V_F]$, $t \ge 0$, with an initial condition
  $u_0$. Then, the following well-posedness inequalities hold for all
  $u_0 \in X$ and all $t \geq 0$:
  \begin{align}
    \label{eq:bound-fp-infty-infty}
    \| e^{tF} u_0 \|_\infty &\leq e^t \|u_0\|_\infty.
    \\
    \label{eq:bound-fp-2-2}
    \| e^{tF} u_0 \|_2 &\leq \sqrt{2} e^{\frac{t}{2}} \|u_0\|_2.
    \\                     
    \label{eq:bound-fp-derivative-infty-infty}
    \| \p_ve^{tF} u_0 \|_\infty &\leq e^{2t}  \| \p_vu_0\|_\infty.
    \\
    \label{eq:bound-fp-derivative-2-derivative-2}
    \| \p_ve^{tF} u_0 \|_2 &\leq \sqrt{2} e^{\frac{3t}{2}} \|\p_vu_0\|_2.
  \end{align}
  Also, the following ``regularization'' inequalities hold for some
  constants $C > 0$, all $u_0 \in X$ and $t$ small enough:
  \begin{align}
     \label{eq:bound-fp-2-infty}
    \| e^{tF} u_0 \|_\infty &\leq C t^{-1/4} \|u_0\|_2.
    \\
    \label{eq:bound-fp-infty-infty-derivative}
    \| \p_ve^{tF} u_0 \|_\infty &\leq Ct^{-1/2} \|u_0\|_\infty.
    \\                               
    \label{eq:bound-fp-derivative-2-infty}
    \| \p_ve^{tF} u_0 \|_\infty &\leq C t^{-3/4} \|u_0\|_2.
    \\ 
    \label{eq:bound-fp-derivative-2-2}
    \| \p_ve^{tF} u_0 \|_2 &\leq C t^{-1/2} \|u_0\|_2.
    \\ 
    \label{eq:bound-fp-derivative-infty-2}
    \| \p_ve^{tF} u_0 \|_\infty &\leq Ct^{-1/4} \|\p_v u_0\|_2.
  \end{align}
\end{lem}

\begin{proof}
  These estimates are obtained in a straightforward way from the
  explicit solution \eqref{eq:FP-sol-2}, and are a consequence of the
  corresponding estimates for the standard Fokker-Planck equation on
  $\R$. We use the standard properties
  \begin{equation}
    \label{eq:Phi1}
    \int_{-\infty}^{+\infty} \Phi_t(v,w) \d v = 1,
    \qquad
    \int_{-\infty}^{+\infty} \Phi_t(v,w) \d w = e^t
  \end{equation}
  and
  \begin{equation}
    \label{eq:Phi2}
    \int_{-\infty}^{+\infty} (\Phi_t(v,w))^2 \d w
    = \frac{1}{2 \sqrt{\pi}} (1 - e^{-2t})^{-\frac12} e^t.
  \end{equation}
  In order to obtain \eqref{eq:bound-fp-infty-infty}, we directly apply
  the second part of \eqref{eq:Phi1} to the expression
  \eqref{eq:FP-sol-2}. To obtain \eqref{eq:bound-fp-2-2} from
  \eqref{eq:FP-sol-2} we use the Cauchy-Schwarz inequality and the
  second part of \eqref{eq:Phi1} to get
  \begin{multline}
    \label{eq:bound22}
    \left(
      \int_{-\infty}^{\infty} \widetilde{u}_0(w) \Phi_t(v-V_F, w-V_F) \d w
    \right)^2
    \\
    \leq
    \left(
      \int_{-\infty}^{\infty} \widetilde{u}_0(w)^2 \Phi_t(v-V_F,
      w-V_F) \d w
    \right)
    \left(
      \int_{-\infty}^{\infty} \Phi_t(v-V_F, w-V_F) \d w
    \right)
    \\
    =
    e^t \left(
      \int_{-\infty}^{\infty} \widetilde{u}_0(w)^2 \Phi_t(v-V_F,
      w-V_F) \d w
    \right).
  \end{multline}
  Integrating now in $v$ and taking into account that
  $\| \widetilde{u_0} \|_2^2 = 2 \| u_0 \|_2^2$ gives
  \eqref{eq:bound-fp-2-2}.

  For the estimates involving the derivative $\p_v$ we note that
  \begin{equation}
    \label{eq:Phi-2}
    \p_v \Phi_t(v,w)
    =
    - e^t \p_w \Phi_t(v,w).
  \end{equation}
  Using this, to obtain \eqref{eq:bound-fp-derivative-infty-infty} we
  write
  \begin{multline}
    \label{eq:pvu}
    \p_v u(v, t)
    =
    \int_{-\infty}^{\infty} \widetilde{u}_0(w)
    \p_v \Phi_t(v-V_F, w-V_F) \d w
    \\
    =
    e^t \int_{-\infty}^{\infty} \p_w \widetilde{u}_0(w)
    \Phi_t(v-V_F, w-V_F) \d w,
  \end{multline}
  and now take the supremum of both sides and use the second part of
  \eqref{eq:Phi1}. We can also obtain
  \eqref{eq:bound-fp-derivative-2-derivative-2} from this same
  expression by using a very similar argument as in
  \eqref{eq:bound22}.
  
  Regarding the regularization estimates, in order to show
  $\| e^{tF} u_0 \|_\infty \leq C t^{-1/4} \|u_0\|_2$, i.e. 
  \eqref{eq:bound-fp-2-infty}, from \eqref{eq:FP-sol-2} we easily get
  \begin{equation*}
    \|e^{tF} u_0\|_\infty \leq \sqrt{2} \|u_0\|_2 \|\Phi_t(v,\cdot)\|_2,
  \end{equation*}
  and using \eqref{eq:Phi2} gives \eqref{eq:bound-fp-2-infty}. To show
  \eqref{eq:bound-fp-infty-infty-derivative}, we take the supremum in the
  first expression of equation \eqref{eq:pvu} and use the estimate
  \begin{equation}
    \label{eq:pvPhi_1}
    \int_{-\infty}^{+\infty}
    \left| \p_v \Phi_t(v,w) \right| \d w
    = \frac{\sqrt{2} e^t}{\sqrt{\pi (1-e^{-2t})}}
    \int_{-\infty}^{+\infty} |z| e^{-z^2} \d z
    \lesssim
    C t^{-\frac12}
    \qquad
    \text{for $0 < t \leq 1$.}
  \end{equation}
  To prove $  \| \p_ve^{tF} u_0 \|_\infty \leq C t^{-3/4} \|u_0\|_2$, i.e. 
  \eqref{eq:bound-fp-derivative-2-infty}, we use \eqref{eq:pvu} to
  write
  \begin{equation*}
    \| \p_v e^{tF} u_0 \|_\infty
    \leq
    e^t \sqrt{2} \| u_0 \|_2 \|\p_v \Phi_t(v, \cdot)\|_2.
  \end{equation*}
  The fact that
  \begin{equation}
    \label{eq:pvPhi_2}
    \int_{-\infty}^{+\infty} (\p_v \Phi_t(v,w))^2 \d w
    =
    \frac{\sqrt{2}}{\pi} e^t (1 - e^{-t})^{-\frac32}
    \int_{-\infty}^{+\infty} z^2 e^{-2z^2} \d z
  \end{equation}
  now shows \eqref{eq:bound-fp-derivative-2-infty}. To prove
  $  \| \p_ve^{tF} u_0 \|_2 \leq C t^{-1/2} \|u_0\|_2$, i.e. 
    \eqref{eq:bound-fp-derivative-2-2}, we  use the Cauchy-Schwarz
    inequality in \eqref{eq:pvu} in a similar way as in equation
    \eqref{eq:bound22} to get
    \begin{equation*}
      \| \p_v e^{tF} u_0 \|_2
      \leq
      C e^{\frac{t}{2}} t^{-\frac14} \left(
        \int_{-\infty}^{\infty} \int_{-\infty}^{\infty}
        \widetilde{u}_0(w)^2 |\p_v \Phi_t(v-V_F,
        w-V_F)| \d v \d w
      \right)^{\frac12},
    \end{equation*}
    where we also used \eqref{eq:pvPhi_1}. Now the bound
    \begin{equation*}
      \int_{-\infty}^{\infty} |\p_v \Phi_t(v, w)| \d v
      = \frac{\sqrt{2}}{\sqrt{\pi (1 - e^{-2t})}}
      \int_{-\infty}^{+\infty} |z| e^{-z^2} \d z
      \lesssim
      C t^{-\frac12}
      \qquad
      \text{for $0 < t \leq 1$}
    \end{equation*}
    completes the proof of
    \eqref{eq:bound-fp-derivative-2-2}. Finally, to prove
    \eqref{eq:bound-fp-derivative-infty-2},
    i.e. $   \| \p_ve^{tF} u_0 \|_\infty \leq Ct^{-1/4} \|\p_v u_0\|_2$,
    we take the supremum in the
    last expression of \eqref{eq:pvu}, apply the Cauchy-Schwarz
    inequality, and then use \eqref{eq:Phi2}.
\end{proof}

We also need some additional weighted $L^2$ well-posedness and
regularization estimates with Gaussian weight, which we state
separately here:

\begin{lem}[$L^2$ Gaussian estimates for the Fokker-Planck equation]
  \label{lem:FP-Gaussian-L2-estimates}
  Take $v_0 \in \R$ and let $\varphi(v) := \exp((v-v_0)^2 / 2)$.
  There exist $t_0 > 0$ and $C > 0$ depending only on $v_0$ and $V_F$
  such that
  \begin{equation}
    \label{eq:FP-Gaussian-estimate-1}
    \| \p_v e^{tF} u_0 \|_{L^2(\varphi)}
    \leq C \| \p_v u_0 \|_{L^2(\varphi)}
    \qquad
    \text{for all $0 < t \leq t_0$,}
  \end{equation}
  and 
  \begin{equation}
    \label{eq:FP-Gaussian-estimate-2}
    \| \p_v e^{tF} u_0 \|_{L^2(\varphi)}
    \leq C t^{-1/2} \| u_0 \|_{L^2(\varphi)}
    \qquad
    \text{for all $0 < t \leq t_0$,}
  \end{equation}
  both for any $u_0 \in X$.
\end{lem}

\begin{rem}
  \label{rem:iterate}
  Of course, estimate \eqref{eq:FP-Gaussian-estimate-1} (which has the
  same norm on both sides) can be iterated for later times to obtain
  \begin{equation*}
    \| \p_v e^{tF} u_0 \|_{L^2(\varphi)}
    \leq C e^{\sigma t} \| \p_v u_0 \|_{L^2(\varphi)}
    \qquad \text{for all $t \geq 0$}
  \end{equation*}
  for some (other) $C > 0$ and some $\sigma > 0$. We have chosen to
  state this kind of estimates only for a short time when we have no
  straightforward estimate on $\sigma$.
\end{rem}

\begin{proof}[Proof of Lemma \ref{lem:FP-Gaussian-L2-estimates}]
  \textbf{Proof of the first estimate when $v_0=0$.} By possibly
  changing $v_0$ it is enough to write the proof for $V_F = 0$. We
  will first prove the result also assuming that $v_0 = 0$, and then
  point out the modifications needed to obtain the general case. Call
  $u(\cdot,t) = e^{tF} u_0$. From \eqref{eq:FP-sol-2} we have
  \begin{multline*}
    \p_v u(v,t) =
    \int_{-\infty}^0 u_0(w) \p_v \Psi_t(v,w) \d w
    \\
    =
    -e^t \int_{-\infty}^0
    u_0(w) \p_w (\Phi_t(v,w) + \Phi_t(v,-w)) \d w
    \\
    =
    e^t \int_{-\infty}^0
    \p_w u_0(w) (\Phi_t(v,w) + \Phi_t(v,-w)) \d w,
  \end{multline*}
  where we have used the property \eqref{eq:Phi-2} to change $\p_v$ by
  $\p_w$. Hence,
  \begin{equation}
    \label{eq:G0}
    | \p_v u(v,t) |
    \leq
    2 e^t \int_{-\infty}^0
    |\p_w u_0(w)|\, \Phi_t(v,w) \d w.
  \end{equation}
  We use this expression to write
  \begin{multline}
    \label{eq:G1}
    \int_{-\infty}^{0} \varphi(v) (\p_v u(v, t))^2 \d v
    \leq
    4 e^{2t} \int_{-\infty}^{0} \varphi(v) \left(
      \int_{-\infty}^{0} \p_w u_0(w)
      \Phi_t(v, w) \d w
    \right)^2
     \d v
    \\
    \leq
    4 e^{3t} \int_{-\infty}^{0} \varphi(v)
    \int_{-\infty}^{0} (\p_w {u}_0(w))^2
    \Phi_t(v, w) \d w  \d v
    \\
    =
    4 e^{3t}  \int_{-\infty}^{0} (\p_w {u}_0(w))^2
    \int_{-\infty}^{0} \varphi(v) \Phi_t(v, w)  \d v \d w,
  \end{multline}
  where we have used the Cauchy-Schwarz inequality and also
  \eqref{eq:Phi1}. In the case $v_0=0$ we can estimate the inner
  integral by (using the shorthand $r(t) := 1- e^{-2t}$)
  \begin{multline}
    \label{eq:G2}
    \int_{-\infty}^{+\infty} \varphi(v) \Phi_t(v, w)  \d v
    =
    \frac{1}{\sqrt{2\pi r(t)}}
    \int_{-\infty}^{+\infty} e^{\frac{v^2}{2}}
    \exp\left(
      {-\frac{( v-e^{-t}w )^2}{2r(t)}}
    \right)
    \d v
    \\
    =
    \frac{1}{\sqrt{1 - r(t)}}
    \exp\left(
      {\frac{ w^2}{2e^{2t}(1 - r(t))}}
    \right)
     =
    {e^{2t}}
    \exp\left(
      \frac{ w^2}{2}
    \right)
    =
    {e^{2t}}
    \varphi(w).
  \end{multline}
  In order to calculate the integral above we have used the following
  well-known formula for a Gaussian integral with $\alpha > 0$:
  \begin{equation}
    \label{eq:Gaussian-integral}
    \int_{-\infty}^{+\infty} e^{-\alpha v^2 + \beta v + \gamma} \d v
    = \frac{\sqrt{\pi}}{\sqrt{\alpha}}
    e^{\frac{\beta^2}{4\alpha} + \gamma}.
  \end{equation}
  In the previous calculation we used formula
  \eqref{eq:Gaussian-integral} with $\alpha = (1-r(t))/(2 r(t))$,
  $\beta = e^{-t} w / r(t)$ and $\gamma = -e^{-2t}w^2 / 2r(t)$. From
  \eqref{eq:G1} and \eqref{eq:G2} we immediately obtain
  \eqref{eq:FP-Gaussian-estimate-1}.

  \medskip
  \noindent
  \textbf{Proof of the first estimate for a general $v_0$.} Now, for
  any $v_0 \in \R$ we proceed in a similar way as before, but now
  choosing $p,q > 1$ such that $\frac{1}{p} + \frac{1}{q} = 1$ when
  applying the Cauchy-Schwarz inequality:
  \begin{multline}
    \label{eq:Gcauchy}
    \left( \int_{-\infty}^{0} |\p_w u_0(w)| \Phi_t(v, w) \d w \right)^2
    \\
    \leq
    \left( \int_{-\infty}^{0} (\p_w u_0(w))^2 \,
      \Phi_t(v, w)^{\frac{2}{p}} \d w \right)
    \left( \int_{-\infty}^{0} \Phi_t(v, w)^{\frac{2}{q}} \d w \right)
    \\
    =
    \frac{1}{2 \pi r(t)}
    \left( \int_{-\infty}^{0} (\p_w u_0(w))^2 \,
      e^{ -\frac{(v - e^{-t} w)^2}{p r(t)} }
      \d w \right)
    \left( \int_{-\infty}^{0}
      e^{ -\frac{(v - e^{-t} w)^2}{q r(t)} }
      \d w \right)
    \\
    \leq
    \frac{C e^t}{\sqrt{r(t)}}
    \int_{-\infty}^{0} (\p_w u_0(w))^2 \,
    e^{ -\frac{(v - e^{-t} w)^2}{p r(t)} }
    \d w.
  \end{multline}
  Using this in \eqref{eq:G0} we get
  \begin{multline}
    \label{eq:G3}
    \int_{-\infty}^{0} \varphi(v) (\p_v u(v, t))^2 \d v
    \leq
    4 e^{2t} \int_{-\infty}^{0} \varphi(v) \left(
      \int_{-\infty}^{0} |\p_w u_0(w)|
      \Phi_t(v, w) \d w
    \right)^2
     \d v
    \\
    \leq
    \frac{C e^{3t}}{\sqrt{ r(t)}}
    \int_{-\infty}^{0} \varphi(v)
    \int_{-\infty}^{0} (\p_w {u}_0(w))^2
    e^{ -\frac{(v - e^{-t} w)^2}{p r(t)} }
     \d w  \d v
    \\
    =
    \frac{C e^{3t}}{\sqrt{ r(t)}}
    \int_{-\infty}^{0}
     (\p_w {u}_0(w))^2
    \int_{-\infty}^{0} \varphi(v)
    e^{ -\frac{(v - e^{-t} w)^2}{p r(t)} }
     \d v  \d w.
  \end{multline}
  An application of the integral formula \eqref{eq:Gaussian-integral}
  gives
  \begin{multline*}
    \frac{1}{\sqrt{r(t)}}
    \int_{-\infty}^{0} \varphi(v)
    e^{ -\frac{(v - e^{-t} w)^2}{p r(t)} }
    \d v
    \leq
    \int_{-\infty}^{+\infty} e^{\frac{(v-v_0)^2}{2}}
    e^{ -\frac{(v - e^{-t} w)^2}{p r(t)} }
    \d v
    \\
    \leq
    \frac{C}{\sqrt{1 - \frac{p}{2}r(t)}}
    \exp\left(
      {\frac{ (e^{-t}w - v_0)^2}{2 - p r(t)}}
    \right).
  \end{multline*}
  In order to conclude as before it is enough to show that there is
  some constant $C > 0$ such that
  \begin{equation*}
    \exp\left(
      {\frac{ (e^{-t}w - v_0)^2}{2 - p r(t)}}
    \right)
    \leq C 
    \exp\left(
      {\frac{ (w - v_0)^2}{2}}
    \right)
    \qquad
    \text{for all $0 \leq t \leq 1$ and $w \leq 0$.}
  \end{equation*}
  Gathering terms one sees it is enough to show that there exists
  some (other) constant $C > 0$ such that
  \begin{equation*}
    w^2 \left( \frac{1}{e^{2t} (2 - p r(t)) } - \frac{1}{2} \right)
    + |w| |v_0| \left( 1 - \frac{2 e^{-t}}{2 - pr(t)} \right) \leq C
  \end{equation*}
  for all $w \leq 0$ and all $0 \leq t \leq 1$. Choosing $p = 3/2$
  (in fact, any $p \in (1,2)$ works) we may find $a, b > 0$ such that
  \begin{equation*}
    w^2 \left( \frac{1}{e^{2t} (2 - p r(t)) } - \frac{1}{2} \right)
    + |w| |v_0| \left( 1 - \frac{2 e^{-t}}{2 - pr(t)} \right)
    \leq
    - a t w^2 + b t |w| |v_0|
  \end{equation*}
  for all $w \leq 0$ and all $0 \leq t \leq 1$.  These $a$,$b$ are
  obtained easily by noticing that both of the terms in brackets are
  $0$ at $t=0$, and have simple linear bounds. We also notice that the
  factor of $w^2$ is negative whenever $t > 0$ and $p < 2$. Now by
  finding the maximum of this parabola in $w$ it
  is easy to find $C > 0$ such that
  \begin{equation*}
    - a t w^2 + b t |w| |v_0|
    \leq C
  \end{equation*}
  for all $w \leq 0$ and all $0 \leq t \leq 1$. This shows the result.

  \medskip
  \noindent
  \textbf{Proof of the second estimate.} The second estimate in the
  statement can be obtained in a similar way. Choose $1 < p < 2$. By
  using the explicit expression of $\p_v \Phi_t$ from
  \eqref{eq:fp-kernel-derivative} we first notice that
  \begin{gather}
    |\p_v \Phi_t(v,w)| \lesssim \frac{1}{r(t)} e^{-\frac{(v - e^{-t}
        w)^2}{2 p r(t)}},
    \\
    |\p_v \Phi_t(v, -w)|
    \lesssim
    \frac{1}{r(t)} e^{-\frac{(v + e^{-t}
        w)^2}{2 p r(t)}}
    \lesssim
    \frac{1}{r(t)} e^{-\frac{(v - e^{-t}
        w)^2}{2 p r(t)}},
  \end{gather}
  which implies
  \begin{equation*}
    |\p_v \Psi_t(v,w)|
    \lesssim
    \frac{1}{r(t)} e^{-\frac{(v - e^{-t} w)^2}{2 p r(t)}}.
  \end{equation*}
  Now we carry out a similar estimate as in \eqref{eq:Gcauchy} to get
    \begin{multline*}
    \left( \int_{-\infty}^{0} u_0(w) \p_v \Psi_t(v, w) \d w \right)^2
    \\
    \leq
    \left( \int_{-\infty}^{0} u_0(w)^2 \,
      (\p_v \Psi_t(v, w))^{\frac{2}{p}} \d w \right)
    \left( \int_{-\infty}^{0} (\p_v \Psi_t(v, w))^{\frac{2}{q}} \d w \right)
    \\
    \lesssim
    \frac{1}{r(t)^2}
    \left( \int_{-\infty}^{0} u_0(w)^2 \,
      e^{ -\frac{(v - e^{-t} w)^2}{p^2 r(t)} }
      \d w \right)
    \left( \int_{-\infty}^{0}
      e^{ -\frac{(v - e^{-t} w)^2}{p q r(t)} }
      \d w \right)
    \\
    \lesssim
    \frac{1}{(r(t))^{\frac32}}
    \int_{-\infty}^{0} u_0(w)^2 \,
    e^{ -\frac{(v - e^{-t} w)^2}{p^2 r(t)} }
    \d w.
  \end{multline*}
  Proceeding as in \eqref{eq:G3},
  \begin{multline*}
    \int_{-\infty}^{0} \varphi(v) (\p_v u(v, t))^2 \d v
    \leq
    \int_{-\infty}^{0} \varphi(v) \left(
      \int_{-\infty}^{0} u_0(w)
      \p_v \Psi_t(v, w) \d w
    \right)^2
     \d v
    \\
    \lesssim
    \frac{1}{(r(t))^{\frac32}}
    \int_{-\infty}^{0}
     {u}_0(w)^2
    \int_{-\infty}^{0} \varphi(v)
    e^{ -\frac{(v - e^{-t} w)^2}{p^2 r(t)} }
     \d v  \d w.
  \end{multline*}
  Choosing any $1 < p < 2$ such that also $1 < p^2 < 2$ we may carry
  out the same reasoning we did after \eqref{eq:G3} to get
  \begin{multline*}
    \int_{-\infty}^{0} \varphi(v) (\p_v u(v, t))^2 \d v
    \leq
    \int_{-\infty}^{0} \varphi(v) \left(
      \int_{-\infty}^{0} u_0(w)
      \p_v \Psi_t(v, w) \d w
    \right)^2
     \d v
    \\
    \lesssim
    \frac{1}{r(t)}
    \int_{-\infty}^{0}
     {u}_0(w)^2
    \varphi(w)
    \d w.
  \end{multline*}
  This finishes the result.
\end{proof}

\subsubsection{Bounds for the Fokker-Planck equation plus the source term}
\label{sec:plus_delta}
In the following lemma, we show  the bounds for the semigroup that adds
the source term $\delta_{V_R}N_u(t)$ to the previous one.
\begin{lem}
  \label{lem:F-L1}
  Denote by $e^{tA} u_0$ the classical solution to
  \begin{equation}
    \label{eq:FP-plus-delta}
    \p_t u = F u + L_1 u = \p_v^2 u + \p_v [(v-V_F) u] +
    \delta_{V_R} N_u(t)
  \end{equation}
  in the domain $v\in (-\infty,V_F]$, $t \ge 0$, with the Dirichlet
  boundary condition \eqref{eq:linear-PDE-bc} and with an initial
  condition $u_0 \in X$. Then there exist $C > 0$ and $t_0 > 0$
  depending only on the system parameters such that the following
  well-posedness inequalities hold for all $0 \leq t \leq t_0$:
  \begin{align}\label{eq:bound-linear-incomplete-derivative-infty-infty}
   \| \p_ve^{tA} u_0 \|_\infty &\leq C \| \p_vu_0\|_\infty,
     \\
    \label{eq:bound-linear-incomplete-derivative-2-derivative-2}
    \| \p_ve^{tA} u_0 \|_2 &\leq C \|\p_vu_0\|_2,
     \\
    \label{eq:bound-linear-incomplete-infty-infty}
    \| e^{tA} u_0 \|_\infty &\leq C \|u_0\|_\infty,
     \\
    \label{eq:bound-linear-incomplete-2-2}
    \| e^{tA} u_0 \|_2 &\leq C \|u_0\|_2,
  \end{align}
  and the following ``regularization'' inequalities hold for all times
  $0 < t \leq t_0$:
  \begin{align}\label{eq:bound-linear-incomplete-derivative-2-infinity}
    \| \p_ve^{tA} u_0 \|_{\infty} &\leq C t^{-3/4} \|u_0\|_2.
    \\
    \label{eq:bound-linear-incomplete-infty-infty-derivative}
    \| \p_ve^{tA} u_0 \|_\infty &\leq Ct^{-1/2} \|u_0\|_\infty.
    \\
    \label{eq:bound-linear-incomplete-derivative-2-2}
    \| \p_ve^{tA} u_0 \|_2 &\leq C t^{-1/2} \|u_0\|_2.
    \\
    \label{eq:bound-linear-incomplete-2-infinity}
    \| e^{tA} u_0 \|_\infty &\leq C t^{-1/4} \|u_0\|_2.
  \end{align}
\end{lem}

As we mentioned in Remark \ref{rem:iterate}, inequalities
\eqref{eq:bound-linear-incomplete-derivative-infty-infty}--\eqref{eq:bound-linear-incomplete-2-2}
can be iterated to obtain a bound for all times.

\begin{proof}
  The solution to this equation can be written through Duhamel's
  formula by using the Fokker-Planck semi-group $e^{tF}$ from Lemma
  \ref{lem:FP-bounds} as follows:
  \begin{multline}
    \label{eq:linear-incomplete}
    u(v,t) = e^{tA}u_0(v)
    =
    e^{tF}u_0(v)
    +
    \int_{0}^{t}e^{\left(t-s\right)F} \left[ N_{u}(s)\delta_{V_R}
    \right] \d s
   \\
   =e^{tF}u_0(v) + \int_{0}^{t}\Psi_{t-s}(v - V_F, V_R-V_F)N_{u}(s)\d s,
  \end{multline}
  where we have explicitly written $e^{(t-s) F}\left[ \delta_{V_R} \right]$ by
  using equation \eqref{eq:FP-sol-2}. We note that this Duhamel's
  formula is a possible definition of classical solutions to
  \eqref{eq:FP-plus-delta}, similarly to the approach in
  \cite{carrillo2013classical}.
  
  Taking the derivative in terms of $v$ of the previous equation, we
  can write the derivative of the solution to the incomplete linear
  equation as
  \begin{equation}
    \label{eq:linear-incomplete-derivative}
    \p_v u(v,t)=\p_v e^{tA}u_0(v)
    =
    \p_v e^{tF}u_0(v)
    +
    \int_{0}^{t} \p_v\Psi_{t-s}(v - V_F, V_R-V_F) N_{u}(s) \d s.
  \end{equation}
  We are first going to show the bounds on the derivative of the
  solution $u$, which will be useful later to prove the bounds on $u$
  itself.
  
  \medskip  
  \noindent
  \textbf{Step 1. Bounds on $\boldsymbol{N_u(t)}$.}  We need to bound $N_u(t)$ by
  the three different norms contained in the norm of $X$. With this
  aim we compute $ N_u(t) $ by evaluating equation
  \eqref{eq:linear-incomplete-derivative} at $ V_F $:
  \begin{equation}\label{eq:linear-incomplete-derivative-1}
    -N_{u}(t) =
    \p_v e^{tF}u_0(V_F)+\int_{0}^{t} N_{u}(s) \p_v\Psi_{t-s}(0,V_R-V_F) \d s,
  \end{equation}
  where we recall that the kernel $\Psi$ is given in 
  \eqref{eq:Psi-kernel}--\eqref{eq:fp-kernel}. This kernel is easily seen to satisfy
  \begin{equation*}
    |\p_v\Psi_{t}(0,V_R-V_F)| < C
    \qquad \text{for all $t>0$,}
  \end{equation*}
  for some constant $C > 0$. Using also the bound
  \eqref{eq:bound-fp-derivative-2-infty}, from
  \eqref{eq:linear-incomplete-derivative-1} we find
  \begin{equation*}
    |N_{u}(t)| \le Ct^{-3/4}\|u_0\|_2 + C\int_{0}^{t}|N_{u}(s)| \d s,
  \end{equation*}
  which can be also written as
  \begin{equation*}
    t^{3/4} |N_{u}(t)| \le C\|u_0\|_2 + Ct^{3/4}\int_{0}^{t}s^{-3/4}s^{3/4}|N_{u}(s)| \d s.
  \end{equation*}
  Now we write the previous equation as
  \begin{equation*}
    z(t) \le C\|u_0\|_2 + Ct^{3/4}\int_{0}^{t}s^{-3/4}z(s)\d s,
  \end{equation*}
  with $ z(t)=t^{3/4}|N_{u}(t)| $, and we apply Lemma \ref{lem:gronwall-continuous} with $ \alpha=3/4 $, $ \beta =0 $ and $ \gamma=3/4 $, which leads to
  \begin{equation}
    \label{eq:N-bound}
    |N_{u}(t)| \le Ct^{-3/4}\|u_0\|_2
  \end{equation}
  for all $t$ small enough. Proceeding as we have done so far,
  but using the bounds
  \eqref{eq:bound-fp-infty-infty-derivative},
  \eqref{eq:bound-fp-derivative-infty-infty} and
  \eqref{eq:bound-fp-derivative-infty-2} instead of
  \eqref{eq:bound-fp-derivative-2-infty}, we end up with:
  \begin{equation}\label{eq:N-bound-infty}
    |N_{u}(t)| \le Ct^{-1/2}\|u_0\|_\infty,
  \end{equation}
  \begin{equation}\label{eq:N-bound-derivative-infty}
    |N_{u}(t)| \le C\|\p_vu_0\|_\infty,
  \end{equation}
  and
  \begin{equation}
    \label{eq:N-bound-derivative-2}
    |N_{u}(t)| \le Ct^{-1/4}\|\p_vu_0\|_2,
  \end{equation}
  all valid for all $t > 0$ small enough.

  \medskip  
  \noindent
  \textbf{Step 2. Bounds on the second term in
    \eqref{eq:linear-incomplete}.} We need to give some bounds for the
  term
  \begin{multline*}
    \int_{0}^{t}\p_v\Psi_{t-s}(v-V_F, V_R-V_F)N_{u}(s) \d s
    \\
    =
    \int_{0}^{t}\p_v\Phi_{t-s}(v-V_F, V_R-V_F)N_{u}(s) \d s
    +
    \int_{0}^{t}\p_v\Phi_{t-s}(v-V_F, V_F-V_R)N_{u}(s) \d s.
  \end{multline*}
  The main difficulty when bounding this lies on the first integral on
  the right-hand side. For this purpose we use the explicit expression
  of the derivative of the kernel, given by equation
  \eqref{eq:fp-kernel-derivative}. We use inequality
  \eqref{eq:N-bound} for $N_u(s)$ as follows, with the notation
  $\tilde{v} \equiv v - V_F$, $\tilde{V}_R \equiv V_R - V_F$ and
  $r(t) := \left(1-e^{-2t}\right)$, (notice that $\tilde{v},\tilde{V_R}
  \in (-\infty,0]$),
  \begin{multline*}
    \left|\int_{0}^{t}\p_v\Phi_{t-s}(v-V_F,V_R-V_F)N_{u}(s) \d
      s\right|
    \\
    \leq (2\pi)^{-\frac12} \int_{0}^{t}r(t-s)^{-3/2}
    |N_{u}(s)|\left|\tilde{v}-e^{-(t-s)}\tilde{V}_R\right|
    e^{-\frac{|\tilde{v}-e^{-(t-s)}\tilde{V}_R|^2} {2r(t-s)} } \d s
    \\
    \le
    (2\pi)^{-\frac12}\|u_0\|_2
    \int_{0}^{t}r(t-s)^{-3/2}s^{-3/4}
    \left|\tilde{v}-e^{-(t-s)}\tilde{V}_R\right|
    e^{-\frac{|\tilde{v}-e^{-(t-s)}\tilde{V}_R|^2}{2r(t-s)}} \d s
    \\
    \le
    Ct^{-3/4}\|u_0\|_2,
  \end{multline*}
  where in the last inequality we have applied the Lemma
  \ref{lem:integrals} with $ \alpha=3/4 $.
  On the other hand, since $v
  - V_F$ and $V_R - V_F$ are both negative,
  \begin{multline*}
    \left|\int_{0}^{t}\p_v\Phi_{t-s}(v-V_F, V_F-V_R)N_{u}(s) \d
      s\right|
    \\
    \leq
    (2\pi)^{-\frac12} \int_{0}^{t}r(t-s)^{-3/2}
    |N_{u}(s)|\left|\tilde{v} + e^{-(t-s)}\tilde{V}_R\right|
    e^{-\frac{|\tilde{v} + e^{-(t-s)}\tilde{V}_R|^2}{2r(t-s)}} \d s
    \\
    \lesssim
    \int_{0}^{t}r(t-s)^{-1}
    |N_{u}(s)| 
    e^{-\frac{|\tilde{v} + e^{-(t-s)}\tilde{V}_R|^2}{4r(t-s)}} \d s
    \\
    \lesssim
    \int_{0}^{t}r(t-s)^{-1}
    |N_{u}(s)| 
    e^{-\frac{\tilde{V}_R^2}{4r(t-s)}} \d s
    \lesssim
    \int_{0}^{t}
    |N_{u}(s)|
    \d s
    \\
    \lesssim
    \|u_0\|_2
    \int_{0}^{t}
    s^{-3/4}
    \d s
    = 4 \|u_0\|_2 t^{1/4}.
  \end{multline*}
  From the last two equations we see that
  \begin{equation*}
    \left|\int_{0}^{t}\p_v\Psi_{t-s}(v-V_F, V_F-V_R)N_{u}(s) \d
      s\right|
    \lesssim
    t^{-3/4} \|u_0\|_2
  \end{equation*}
  for all $t > 0$ small enough. We can repeat this procedure using for
  $N_u(s)$ the bounds \eqref{eq:N-bound-infty},
  \eqref{eq:N-bound-derivative-infty} or
  \eqref{eq:N-bound-derivative-2} instead of \eqref{eq:N-bound},
  obtaining the following four bounds for the integral:
  \begin{equation}
    \label{eq:second-term-bound}
    \left|\int_{0}^{t}\p_v\Psi_{t-s}(v-V_F,V_R-V_F)N_{u}(s) \d s\right|
    \le
    \left\{\begin{split}
	Ct^{-3/4}\|u_0 \|_2 \\
	Ct^{-1/2}\|u_0\|_\infty \\
	C\|\p_vu_0\|_\infty \\
	Ct^{-1/4}\|\p_vu_0\|_2
      \end{split}\right.
  \end{equation}

  \medskip
  \noindent
  \textbf{Step 3. Bounds on $\boldsymbol{\p_v u}$.} Now we can show the bounds for
  the derivative of the solution. With this purpose, in equation
  \eqref{eq:linear-incomplete-derivative} we take the absolute value,
  which leads to
  \begin{equation}\label{eq:linear-incomplete-derivative-2}
    \left|\p_v u(v,t)\right|
    \le
    \left|\p_v e^{tF}u_0(v)\right|
    +
    \left|\int_{0}^{t}\p_v\Psi_{t-s}(v-V_F, V_R-V_F)N_{u}(s) \d s\right|
  \end{equation}
  First, by taking the supremum norm in the inequality
  \eqref{eq:linear-incomplete-derivative-2} together with the first
  bound in \eqref{eq:second-term-bound} we obtain
  \begin{equation*}
    \|\p_v u(.,t)\|_\infty=\|\p_v e^{tA}u_0\|_\infty
    \le
    \|\p_v e^{tF}u_0\|_\infty+Ct^{-3/4}\|u_0\|_2.
  \end{equation*}
  A direct application of the bound
  \eqref{eq:bound-fp-derivative-2-infty} for
  $ \|\p_v e^{tF}u_0\|_\infty $ gives inequality
  \eqref{eq:bound-linear-incomplete-derivative-2-infinity}. Alternatively,
  using the second bound in \eqref{eq:second-term-bound} and
  \eqref{eq:bound-fp-infty-infty-derivative} gives
  \eqref{eq:bound-linear-incomplete-infty-infty-derivative}; and using
  the third bound in \eqref{eq:second-term-bound} and
  \eqref{eq:bound-fp-derivative-infty-infty}, we get to
  \eqref{eq:bound-linear-incomplete-derivative-infty-infty}.
	
  In order to bound the $L^2$ norms of $\p_v u$ we take the $ L^2 $ norm
  in equation \eqref{eq:linear-incomplete-derivative}:
  \begin{equation}
    \label{eq:linear-incomplete-norm-2-derivative}
    \|\p_v e^{tA}u_0\|_2 \le
    \|\p_v e^{tF}u_0\|_2+\int_{0}^{t}\|\p_v\Psi_{t-s}(v-V_F,V_R-V_F)\|_2|N_{u}(s)|\d s.
  \end{equation}
  First we bound $ \|\p_v\Phi_{t}(v-V_F,V_R-V_F)\|_2 $ as follows:
  \begin{multline}
    \label{eq:3}
    \|\p_v\Phi_{t}(v-V_F,V_R-V_F)\|_2^2
    \\
    =
    \frac{1}{2\pi(1-e^{-2t})^{3}}
    \int_{-\infty}^{V_F}|v-V_F - e^{-t}(V_R-V_F)|^2
    e^{-\frac{|v-V_F-e^{-t}(V_R-V_F)|^2}{1-e^{-2t}}}
    \d v
    \\
    \leq
    \frac{1}{2\pi(1-e^{-2t})^{3/2}}
    \int_{\R}w^2e^{-w^2}\d w
    =
    \frac{1}{4\sqrt{\pi}(1-e^{-2t})^{3/2}} \le Ct^{-3/2},
  \end{multline}
  where we used the change of variables
  $ w=\frac{v-V_F - e^{-t}(V_R-V_F)}{(1-e^{-2t})^{1/2}}$. The
  reflected term $ \|\p_v\Phi_{t}(v-V_F,V_F-V_R)\|_2 $ is in turn
  bounded in the same way by using
  $ w=\frac{v-V_F - e^{-t}(V_F-V_R)}{(1-e^{-2t})^{1/2}}$ as change of
  variables in the previous process.  Then we use equations
  \eqref{eq:bound-fp-derivative-2-2} and \eqref{eq:N-bound} so that
  the previous inequality implies
  \begin{equation*}
    \begin{split}
      \|\p_v e^{tA}u_0\|_2 \le Ct^{-1/2}\|u_0\|_2 + C\|u_0\|_2\int_{0}^{t}\left(t-s\right)^{-3/4}s^{-3/4}\d s \\
      = Ct^{-1/2}\|u_0\|_2 + Ct^{-1/2}\|u_0\|_2=Ct^{-1/2}\|u_0\|_2.
    \end{split}
  \end{equation*}
  which is exactly the inequality
  \eqref{eq:bound-linear-incomplete-derivative-2-2}.  If, otherwise,
  in expression \eqref{eq:linear-incomplete-norm-2-derivative} we use
  the inequalities \eqref{eq:bound-fp-derivative-2-derivative-2} and
  \eqref{eq:N-bound-derivative-2} we obtain
  \begin{equation*}
    \|\p_v e^{tA}u_0\|_2 \le C\|\p_vu_0\|_2 + C\|\p_vu_0\|_2\int_{0}^{t}\left(t-s\right)^{-3/4}s^{-1/4}\d s \le C\|\p_vu_0\|_2.
  \end{equation*}
  By directly solving the integral, which is constant in time, we get
  to inequality
  \eqref{eq:bound-linear-incomplete-derivative-2-derivative-2}.
	
  \medskip	
  \noindent
  \textbf{Step 4. Bounds on $\boldsymbol{u}$.} We now carry out similar estimates
  on the solution given in \eqref{eq:linear-incomplete}. First we take
  the supremum norm to get:
  \begin{equation*}
    \|u(.,t)\|_\infty
    =
    \|e^{tA}u_0\|_\infty
    \le
    \|e^{tF}u_0\|_\infty+\int_{0}^{t}\|\Psi_{t-s}(v-V_F,V_R-V_F)\|_\infty|N_{u}(s)| \d s,
  \end{equation*}
  now, using the definition of $\Psi$ (see \eqref{eq:Psi-kernel})
  $$
  \Phi_{t-s}(v-V_F,V_R-V_F)=\Psi_{t-s}(v-V_F,V_R-V_F)+\Phi_{t-s}(v-V_F,V_F-V_R) $$
  and the fact that three terms are nonnegative, we obtain
  $$
  \|\Psi_{t-s}(v-V_F,V_R-V_F)\|_\infty<\|\Phi_{t-s}(v-V_F,V_R-V_F)\|_\infty,
  $$ 
  and therefore
  \begin{equation*}
    \|e^{tA}u_0\|_\infty
    \le
    \|e^{tF}u_0\|_\infty+\int_{0}^{t}
    \|\Phi_{t-s}(v-V_F,V_R-V_F)\|_\infty|N_{u}(s)| \d s.
  \end{equation*}
  By using the explicit expression of the Fokker-Planck kernel
  $\Phi_{t-s}(v-V_F,V_R-V_F)$, together with the equations
  \eqref{eq:bound-fp-2-infty} and \eqref{eq:N-bound}, we find the
  following inequality:
  \begin{equation*}
    \|e^{tA}u_0\|_\infty \le Ct^{-1/4}\|u_0\|_2+C\|u_0\|_2\int_{0}^{t}s^{-3/4}(1-e^{-2(t-s)})^{-1/2}e^{-\frac{|v-V_F-e^{-t}(V_R-V_F)|^2}{2(1-e^{2(t-s)})}}\d s,
  \end{equation*}
  where we can apply Lemma
  \ref{lem:integrals} (considering $\tilde{v} \equiv v - V_F$,
  $\tilde{V}_R \equiv V_R - V_F$)  to bound the integral and obtain inequality
  \eqref{eq:bound-linear-incomplete-2-infinity}.
	
  Additionally, if we want to bound the norm infinity of the solution
  by the norm infinity of the initial condition, we can do as before,
  taking norm infinity in \eqref{eq:linear-incomplete}, but using
  equations \eqref{eq:bound-fp-infty-infty} and
  \eqref{eq:N-bound-infty}, which gives the inequality
  \eqref{eq:bound-linear-incomplete-infty-infty}.
  
  \
	
  \noindent Now we take norm $ L_2 $ of equation
  \eqref{eq:linear-incomplete}, resulting in
  \begin{equation*}
    \| e^{tA}u_0\|_2 \le \| e^{tF}u_0\|_2+\int_{0}^{t}\|\Psi_{t-s}(v-V_F,V_R-V_F)\|_2|N_{u}(s)|\d s,
  \end{equation*}
  where $ |N_{u}(s)| $ can be bounded by using again the equation
  \eqref{eq:N-bound} and $ \|\Psi_{t-s}(v-V_F,V_R-V_F)\|_2 $ can be directly
  bounded using:
  \begin{multline*}
    \|\Psi_{t-s}(v-V_F,V_R-V_F)\|_2^2\le
    \|\Phi_{t-s}(v-V_F,V_R-V_F)\|_2^2 =
    \\
    \frac{1}{2\pi(1-e^{-2(t-s)})}
    \int_{\R}e^{-\frac{|v-V_F-e^{-(t-s)}(V_R-V_F)|^2}{(1-e^{-2(t-s)})}}
    \d v
    \\
    = \frac{1}{2\pi\sqrt{1-e^{-2(t-s)}}}\int_{\R}e^{-w^2} \d w
    = \frac{1}{\sqrt{4\pi(1-e^{-2(t-s)})}},
  \end{multline*}
  having used the change of variables
  $ w=\frac{v-V_F-e^{-(t-s)}(V_R-V_F)}{(1-e^{-(t-s)})^{1/2}} $. Using this
  bound, for small $ t$, such that $t<1-e^{-2t}$,
  and \eqref{eq:bound-fp-2-2} for estimate $\|e^{tF}u_0\|_2$ and
  (see \eqref{eq:N-bound})
    $|N_{u}(t)| \le Ct^{-3/4}\|u_0\|_2$, we get:
  \begin{equation*}
    \| e^{tA}u_0\|_2 \le Ce^{t/2}\|u_0\|_2+C\|u_0\|_2\int_{0}^{t}(t-s)^{-1/4}s^{-3/4}\d s \le C\|u_0\|_2.
  \end{equation*}
  where we have directly solved the integral, using the change of
  variable $\tau=s/t$ and the beta function
  $\beta(x,y):=\int_0^1\tau^{x-1}(1-\tau)^{y-1} d\tau$, for
  $x,y\in \C$ with real part positive (we recall the relation with the
  gamma function,
  $\beta(x,y)=\frac{\Gamma(x)\Gamma(y)}{\Gamma(x+y)}$), to obtain
  inequality \eqref{eq:bound-linear-incomplete-2-2}. This completes
  the proof of all inequalities.
\end{proof}
We finish this subsection with the proof of the lemma used to prove
the previous result.
\begin{lem}
  \label{lem:integrals}
  For $t > 0$, $v \in (-\infty, 0]$, $\tilde{V}_R<0$ and $0 \leq \alpha < 1$, let us
  consider the following integrals:
  \begin{gather*}
    I_1 (v,t) := \int_{0}^{t}s^{-\alpha}
    (1-e^{-2(t-s)})^{-1/2}e^{-\frac{|v-e^{-(t-s)}\tilde{V}_R|^2}{2(1-e^{-2(t-s)})}}\d s,
    \\
    I_2 (v,t) := \int_{0}^{t}s^{-\alpha}
    (1-e^{-2(t-s)})^{-3/2} \, \left|v-e^{-(t-s)}\tilde{V}_R\right| \,
    e^{-\frac{|v-e^{-(t-s)}\tilde{V}_R|^2}{2(1-e^{-2(t-s)})}}\d s.
\end{gather*}
Then, there exists
$C = C(\alpha)$ such that
\begin{equation*}
  I_1(v,t) \le Ct^{\frac12-\alpha}
  \qquad \text{and} \qquad
  I_2(v,t) \le Ct^{-\alpha}
\end{equation*}
for $t$ small enough and $v \in (-\infty, 0]$.

\end{lem}

\begin{proof}
  For both cases we will consider a sufficiently small $t$ so that
  $t \leq (1-e^{-2t})$. In the case of $I_1$ we use this bound and
  estimate the exponential by $1$ to get
  \begin{equation*}
    I_1 \le \int_{0}^{t}s^{-\alpha}(t-s)^{-1/2}\d s
    = t^{1/2-\alpha}\frac{\Gamma(1-\alpha) \sqrt{\pi}}{\Gamma(3/2-\alpha)},
  \end{equation*}
  where the last equality is seen through the change of variables
  $\tau=s/t$, as explained above.

  Now we consider the second integral and we split it by using
  $v-e^{-t}\tilde{V}_R=(v-\tilde{V}_R)+\tilde{V}_R(1-e^{-t})$
  and therefore,
    $|v-e^{-t}\tilde{V}_R|
    \leq |\tilde{V}_R|t + |v-\tilde{V}_R|$
    (valid for all $t \geq 0$, since $1-e^{-t}\le t$):
  \begin{multline*}
    I_2\le\int_{0}^{t}s^{-\alpha}(t-s)^{-3/2}|v-e^{-(t-s)}\tilde{V}_R|
    e^{-\frac{|v-e^{-(t-s)}\tilde{V}_R|^2}{2(t-s)}}\d s
    \\
    \le
    |\tilde{V}_R| \int_{0}^{t}s^{-\alpha}(t-s)^{-1/2}
    e^{-\frac{|v-e^{-(t-s)}\tilde{V}_R|^2}{2(t-s)}}\d s
    \\
    + \int_{0}^{t}s^{-\alpha}(t-s)^{-3/2}|v-\tilde{V}_R|
    e^{-\frac{|v-e^{-(t-s)}\tilde{V}_R|^2}{2(t-s)}} \d s
    =: I_{2.1} + I_{2.2}.
  \end{multline*}
  The first integral, $I_{2,1}$, can be bounded as was done for $I_1$
  by $C t^{1/2-\alpha}$.

  For the second one we use the
  inequality $(a+b)^2 \geq \frac12 a^2 - b^2$ (valid for any real
  $a,b$) to get
  \begin{multline*}
    (v-e^{-(t-s)}\tilde{V}_R)^2 = ((v - \tilde{V}_R) + \tilde{V}_R
    (1 - e^{-(t-s)}))^2
    \\
    \geq \frac12 (v-\tilde{V}_R)^2 - \tilde{V}_R^2 (1 - e^{-(t-s)})^2
    \geq \frac12 (v-\tilde{V}_R)^2 - \tilde{V}_R^2 (t-s)^2.
  \end{multline*}
  As a consequence, for $t$ small, we have
  \begin{equation*}
    e^{ {-\frac{(v-e^{-(t-s)}\tilde{V}_R)^2}{2(t-s)}} }
    \leq
    e^{ {-\frac{(v-\tilde{V}_R)^2}{4(t-s)}} }
    e^{ \frac12 \tilde{V}_R^2(t-s) }
    \lesssim
    e^{ {-\frac{(v-\tilde{V}_R)^2}{4(t-s)}} }.
  \end{equation*}
  We can use this bound on the term $I_{2,2}$ and obtain
  \begin{equation*}
    I_{2.2} \lesssim \int_{0}^{t}s^{-\alpha}(t-s)^{-3/2}|v-\tilde{V}_R| e^{
      {-\frac{(v-\tilde{V}_R)^2}{4(t-s)}} } \d s.
  \end{equation*}
  Now we use the change of variables $w = (v-\tilde{V}_R)^2 / (t-s)$ and
  call $a := (v-\tilde{V}_R)^2$. After standard computations we have
  \begin{equation*}
    I_{2.2} \lesssim
    \int_{\frac{a}{t}}^{\infty} \left(t-\frac{a}{w}\right)^{-\alpha}
    w^{-\frac{1}{2}}
    e^{-\frac{w}{4}} \d w
    =
    t^{-\alpha}
    \int_{\frac{a}{t}}^{\infty} \left(w-\frac{a}{t}\right)^{-\alpha}
    w^{\alpha-\frac{1}{2}}
    e^{-\frac{w}{4}} \d w.
  \end{equation*}
  Now, if $\alpha \geq \frac12$ we write
  \begin{multline*}
    t^{-\alpha}
    \int_{\frac{a}{t}}^{\infty} \left(w-\frac{a}{t}\right)^{-\alpha}
    w^{\alpha-\frac{1}{2}}
    e^{-\frac{w}{4}} \d w
    \lesssim
    t^{-\alpha}
    \int_{\frac{a}{t}}^{\infty} \left(w-\frac{a}{t}\right)^{-\alpha}
    e^{-\frac{w}{8}} \d w
    \\
    \leq
    t^{-\alpha}
    \int_{\frac{a}{t}}^{\infty} \left(w-\frac{a}{t}\right)^{-\alpha}
    e^{-\frac{1}{8}\left( w-\frac{a}{t} \right)} \d w
    =
    t^{-\alpha}
    \int_{0}^{\infty} w^{-\alpha}
    e^{-\frac{w}{8}} \d w
    = C t^{-\alpha} \Gamma(1-\alpha).
  \end{multline*}
  Otherwise, if $0 \leq \alpha < 1/2$, we use that
  $w^{\alpha-\frac{1}{2}} \leq
  \left(w-\frac{a}{t}\right)^{\alpha-\frac{1}{2}}$ to write
  \begin{multline*}
    t^{-\alpha}
    \int_{\frac{a}{t}}^{\infty} \left(w-\frac{a}{t}\right)^{-\alpha}
    w^{\alpha-\frac{1}{2}}
    e^{-\frac{w}{4}} \d w
    \leq
    t^{-\alpha}
    \int_{\frac{a}{t}}^{\infty}
    \left(w-\frac{a}{t}\right)^{-\frac{1}{2}}
    e^{-\frac{1}{4}\left(w-\frac{a}{t}\right)} \d w
    \\
    =
    t^{-\alpha}
    \int_{0}^{\infty}
    w^{-\frac{1}{2}}
    e^{-\frac{w}{4}} \d w
    = C t^{-\alpha} \Gamma(1/2).
  \end{multline*}
  In any of the two cases we obtain
  \begin{equation*}
    I_{2,2} \lesssim C t^{-\alpha}
    \qquad \text{for $t$ small enough.}
  \end{equation*}
  Together with the bound for $I_{2,1}$ we obtain the stated result.
\end{proof}

As in the previous section, we also need Gaussian estimates for the
semigroup associated to the operator $A$:

\begin{lem}[$L^2$ Gaussian estimates including the source term]
  \label{lem:FP-delta-Gaussian-L2-estimates}
  Take $v_0 \in \R$ and let $\varphi(v) := \exp((v-v_0)^2 / 2)$.
  There exists $C > 0$ and $t_0 > 0$ depending only on $v_0$ and $V_F$
  such that
  \begin{equation}
    \label{eq:FP-delta-Gaussian-estimate-1}
    \| \p_v e^{tA} u_0 \|_{L^2(\varphi)}
    \leq C \| \p_v u_0 \|_{L^2(\varphi)}
    \qquad
    \text{for all $0 \leq t \leq t_0$}
  \end{equation}
  and
  \begin{equation}
    \label{eq:FP-delta-Gaussian-estimate-2}
    \| \p_v e^{tA} u_0 \|_{L^2(\varphi)}
    \leq C t^{-1/2} \| u_0 \|_{L^2(\varphi)}
    \qquad
    \text{for all $0 < t \leq t_0$,}
  \end{equation}
  both for any $u_0 \in X$.
\end{lem}

\begin{proof}
  For simplicity we show the estimates for $V_F = 0$, since the
  general case follows also from very similar calculations. Let us
  first show \eqref{eq:FP-delta-Gaussian-estimate-1}. We take the
  $L^2(\varphi)$ norm in \eqref{eq:linear-incomplete-derivative} and
  obtain similar bound as in
  \eqref{eq:linear-incomplete-norm-2-derivative}:
  \begin{equation*}
    \|\p_v e^{tA}u_0\|_{L^2(\varphi)} \le
    \|\p_v e^{tF}u_0\|_{L^2(\varphi)}
    +
    \int_{0}^{t} \|\p_v\Psi_{t-s}(v,V_R)\|_{L^2(\varphi)}
    |N_{u}(s)|\d s.
  \end{equation*}
  An explicit calculation along the lines of \eqref{eq:3} shows that
  \begin{equation*}
    \|\p_v\Psi_{t}(v, V_R)\|_{L^2(\varphi)}^2
    \lesssim
    t^{-3/2}.
  \end{equation*}
  Using this bound, \eqref{eq:N-bound-derivative-2} and
  \eqref{eq:FP-Gaussian-estimate-1} from Lemma
  \ref{lem:FP-Gaussian-L2-estimates} we have
  \begin{equation*}
    \|\p_v e^{tA}u_0\|_{L^2(\varphi)}
    \lesssim
    \|\p_v u_0\|_{L^2(\varphi)}
    +
    \| \p_v u_0 \|_{L^2(\varphi)} \int_{0}^{t} (t-s)^{-\frac34}
    s^{-\frac14} \d s
    \lesssim
    \|\p_v u_0\|_{L^2(\varphi)},
  \end{equation*}
  valid for all times $t \geq 0$ small enough. This shows
  \eqref{eq:FP-delta-Gaussian-estimate-1}.

  In order to prove \eqref{eq:FP-delta-Gaussian-estimate-2} we follow
  the same path but use \eqref{eq:N-bound} and
  \eqref{eq:FP-Gaussian-estimate-2} instead of
  \eqref{eq:N-bound-derivative-2} and
  \eqref{eq:FP-Gaussian-estimate-1}, respectively, to get
  \begin{equation*}
    \|\p_v e^{tA}u_0\|_{L^2(\varphi)}
    \lesssim
    t^{-\frac12} \|u_0\|_{L^2(\varphi)}
    +
    \| u_0 \|_{L^2(\varphi)} \int_{0}^{t} (t-s)^{-\frac34}
    s^{-\frac34} \d s
    \lesssim t^{-\frac12} \|u_0\|_{L^2(\varphi)}.
  \end{equation*}
  This completes the proof.  
\end{proof}

\subsubsection{Bounds for the complete linear equation}
\label{sec:complete}
Finally, using the bounds proved in Lemma \ref{lem:F-L1} we can obtain
the bounds for the complete linear equation.
\begin{lem}
  \label{lem:bounds-complete}
  Let us consider $u_0 \in X$ any initial condition to the equation:
  \begin{equation*}
    \p_tu(v,t)=\p_v^2u(v,t)+\p_v\left(vu(v,t)\right) + N_{u}(t)\delta_{V_R} -
    bN\p_v u(v,t)=:Lu(v,t)
  \end{equation*}
  in the domain $ v\in\left(-\infty,V_F\right) $, $ t\ge0 $, with the
  boundary condition $ u(V_F,t)=0 $ and whose associated semi-group is
  $ e^{tL} $. We recall that $b$ is the connectivity parameter and $N$
  is considered constant here.  Then there exist $C \geq 1$ and
  $t_0 > 0$ depending only on $N$, $V_R$ and $V_F$ such that for all
  $u_0 \in X$ and all $0 \leq t \leq t_0$ the following well-posedness
  inequalities hold:
  \begin{align}
	\label{eq:bound-linear-derivative-2-derivative-2}
	\| \p_ve^{tL} u_0 \|_2 &\leq C \|\p_vu_0\|_2,
    \\
	\label{eq:bound-linear-derivative-infty-infty}
	\| \p_ve^{tL} u_0 \|_\infty &\leq C \| \p_vu_0\|_\infty,
	\\
	\label{eq:bound-linear-infty-infty}
	\| e^{tL} u_0 \|_\infty &\leq C \|u_0\|_\infty,
	\\
	\label{eq:bound-linear-2-2}
	\| e^{tL} u_0 \|_2 &\leq C \|u_0\|_2,
  \end{align}
  and the following ``regularization'' inequalities hold for all $0 <
  t \leq t_0$:
  \begin{align}
	\label{eq:bound-linear-complete-derivative-2-2}
	\| \p_ve^{tL} u_0 \|_2 &\leq C t^{-1/2} \|u_0\|_2,
	\\
	\label{eq:bound-linear-complete-derivative-2-infinity}
	\| \p_ve^{tL} u_0 \|_\infty &\leq C t^{-3/4} \|u_0\|_2,
	\\
	\label{eq:bound-linear-infty-infty-derivative}
	\| \p_ve^{tL} u_0 \|_\infty &\leq Ct^{-1/2} \|u_0\|_\infty,
	\\
	\label{eq:bound-linear-complete-2-infinity}
	\| e^{tL} u_0 \|_\infty &\leq C t^{-1/4} \|u_0\|_2.
  \end{align}
\end{lem}

\begin{proof}
  We use the splitting \eqref{eq:linear-PDE-2} for the PDE,
  $Lu=Fu+L_1u+L_2u$, which is rewritten as $Lu=Au+L_2u$ and we recall
  that $L_2u=\p_v[(V_F-bN)u]$ and $V_F-bN$ is a constant. Again, to
  prove the lemma, we use Duhamel's formula for the solution and its
  derivative:
   \begin{gather}
    \label{eq:linear-complete}
    u(v,t) = e^{tL}u_0(v) =
    e^{tA}u_0(v) +
    (V_F - bN)\int_0^te^{(t-s)A}\left( \p_v u(v,s)\right)\d s,
    \\
    \label{eq:linear-complete-derivative}
    \p_v u(v,t) =
    \p_ve^{tL}u_0(v)=\p_ve^{tA}u_0(v)+
    (V_F - bN) \int_0^t\p_ve^{(t-s)A}\left(\p_v u(v,s)\right)\d s,
  \end{gather}
   and  bounds obtained Lemma \ref{lem:F-L1} for the semigroup $e^{At}$.
   
  \medskip
  \noindent
  \textbf{Step 1. Bounds on $\boldsymbol{\p_v u}$.} If we take the
  $ L^2 $ norm in equation \eqref{eq:linear-complete-derivative} we
  obtain the following inequality:
  \begin{equation}\label{eq:complete-derivative-norm-2}
    \|\p_ve^{tL}u_0\|_2 \le \|\p_ve^{tA}u_0\|_2+C\int_0^t\|\p_v e^{(t-s)A}\left(\p_v u(.,s)\right)\|_2\d s,
  \end{equation}
  which, directly using inequality \eqref{eq:bound-linear-incomplete-derivative-2-2}, $ \|\p_ve^{tA} u_0 \|_2 \leq C t^{-1/2} \|u_0\|_2 $, implies
  \begin{equation*}
    \|\p_ve^{tL}u_0\|_2 \le Ct^{-1/2}\|u_0\|_2+C\int_{0}^{t}\left(t-s\right)^{-1/2}\|\p_v u(.,s)\|_2\d s.
  \end{equation*}
  After that, taking into account that $ \|\p_v u(.,s)\|_2 $ can be also written as $ \|\p_v e^{sL}u_0\|_2 $, we use the change of variables $ z(t):=t^{1/2}\|\p_ve^{tL}u_0\|_2 $ to write the previous expression as follows:
  \begin{equation*}
    z(t) \le C\|u_0\|_2+Ct^{1/2}\int_{0}^{t}\left(t-s\right)^{-1/2}s^{-1/2}z(s)\d s,
  \end{equation*}
  where the application of the Lemma \ref{lem:gronwall-continuous}
  with $ \alpha=\beta=\gamma=1/2 $ gives us the proof for the
  inequality \eqref{eq:bound-linear-complete-derivative-2-2}. Another
  way of bounding the l.h.s of \eqref{eq:complete-derivative-norm-2}
  consists of using inequality
  \eqref{eq:bound-linear-incomplete-derivative-2-derivative-2},
  $ \| \p_ve^{tA} u_0 \|_2 \leq C \|\p_vu_0\|_2 $, in the first term
  of the r.h.s and \eqref{eq:bound-linear-incomplete-derivative-2-2},
  $ \| \p_ve^{tA} u_0 \|_2 \leq C t^{-1/2} \|u_0\|_2 $, in the second
  one, obtaining
  \begin{equation*}
    \|\p_ve^{tL}u_0\|_2 \le C\|\p_vu_0\|_2+C\int_{0}^{t}\left(t-s\right)^{-1/2}\|\p_v e^{sL}u_0\|_2\d s,
  \end{equation*}
  which proves inequality
  \eqref{eq:bound-linear-derivative-2-derivative-2} after the
  application of Lemma \ref{lem:gronwall-continuous} with
  $ \alpha=0,\beta=1/2 $ and $ \gamma=0 $.
  
  \
	
  Now we are going to bound equation
  \eqref{eq:linear-complete-derivative} by taking norm $ L^\infty $ on
  it, which leads to
  \begin{equation}\label{eq:linear-complete-derivative-norm-infinity}
    \|\p_ve^{tL}u_0\|_\infty \le \|\p_ve^{tA}u_0\|_\infty+C\int_0^t\|\p_ve^{(t-s)A}\left(\p_v u(.,s)\right)\|_\infty \d s,
  \end{equation}
  where, using equation \eqref{eq:bound-linear-incomplete-derivative-2-infinity}, $ \|\p_ve^{tA} u_0 \|_{\infty} \leq C t^{-3/4} \|u_0\|_2 $, we obtain
  \begin{equation*}
    \|\p_ve^{tL}u_0\|_\infty \le Ct^{-3/4}\|u_0\|_2+C\int_{0}^{t}\left(t-s\right)^{-3/4}\|\p_v e^{sL}u_0\|_2\d s.
  \end{equation*}
  Finally, by using the inequality \eqref{eq:bound-linear-complete-derivative-2-2}, $ \| \p_ve^{tL} u_0 \|_2 \leq C t^{-1/2} \|u_0\|_2 $, and by integrating the second term in the r.h.s, we prove inequality \eqref{eq:bound-linear-complete-derivative-2-infinity} as follows:
  \begin{equation*}
    \begin{split}
      \|\p_ve^{tL}u_0\|_\infty \le Ct^{-3/4}\|u_0\|_2+C\int_{0}^{t}\left(t-s\right)^{-3/4}s^{-1/2}\| u_0\|_2\d s \\
      =Ct^{-3/4}\|u_0\|_2+Ct^{-1/4}\|u_0\|_2 \le Ct^{-3/4}\|u_0\|_2.
    \end{split}
  \end{equation*}
  If, instead of performing the latter procedure, we apply bound
  \eqref{eq:bound-linear-incomplete-infty-infty-derivative},
  $ \| \p_ve^{tA} u_0 \|_\infty \leq Ct^{-1/2} \|u_0\|_\infty $, to
  inequality \eqref{eq:linear-complete-derivative-norm-infinity}, we
  find
  \begin{equation*}
    \|\p_ve^{tL}u_0\|_\infty \le Ct^{-1/2}\|u_0\|_\infty+C\int_{0}^{t}(t-s)^{-1/2}\|\p_{v} e^{sL}u_0\|_\infty \d s,
  \end{equation*}
  which we can equivalently write, through the use of the change of variables $ z(t):=t^{1/2}\|\p_ve^{tL}u_0\|_\infty $, as
  \begin{equation*}
    z(t)\le C\|u_0\|_\infty + Ct^{1/2}\int_{0}^{t}(t-s)^{-1/2}s^{-1/2}z(s)\d s
  \end{equation*}
  where, the application of Lemma \ref{lem:gronwall-continuous} with
  $ \alpha=\beta=\gamma = 1/2 $, implies inequality
  \eqref{eq:bound-linear-infty-infty-derivative}.
  
  \
	
  As a last estimate for the inequality
  \eqref{eq:linear-complete-derivative-norm-infinity}, we apply the
  bounds \eqref{eq:bound-linear-incomplete-derivative-infty-infty},
  $ \| \p_ve^{tA} u_0 \|_\infty \leq C \| \p_vu_0\|_\infty $, in the
  first term on the r.h.s and
  \eqref{eq:bound-linear-incomplete-infty-infty-derivative},
  $ \| \p_ve^{tA} u_0 \|_\infty \leq Ct^{-1/2} \|u_0\|_\infty $, in
  the second one, so that we obtain the following expression:
  \begin{equation*}
    \|\p_ve^{tL}u_0\|_\infty \le C\|\p_vu_0\|_\infty+C\int_{0}^{t}(t-s)^{-1/2}\|\p_{v} e^{sL}u_0\|_\infty \d s,
  \end{equation*}
  which, proceeding as in the last case, through the same change of
  variable, and applying Lemma \ref{lem:gronwall-continuous} as well,
  implies inequality \eqref{eq:bound-linear-derivative-infty-infty}.

  \medskip
  \noindent
  \textbf{Step 2. Bounds on $\boldsymbol{u}$.} Now that we have computed all the
  estimates for the derivative of the solution to the linear problem,
  we are able to control the second term of the Duhamel's formula
  given by equation \eqref{eq:linear-complete}. First we apply the
  norm $ L^\infty $ on it, getting to the following inequality:
  \begin{equation}\label{eq:complete-norm-infinity}
    \|e^{tL}u_0\|_\infty \le \|e^{tA}u_0\|_\infty+C\int_0^t\|e^{(t-s)A}\left(\p_v e^{sL}u_0\right)\|_\infty \d s,
  \end{equation}
  where both terms in the r.h.s can be bounded by using the inequality
  \eqref{eq:bound-linear-incomplete-2-infinity},
  $ \| e^{tA} u_0 \|_\infty \leq C t^{-1/4} \|u_0\|_2 $, leading to
  the following expression:
  \begin{equation*}
    \|e^{tL}u_0\|_\infty \le Ct^{-1/4}\|u_0\|_2+C\int_{0}^{t}(t-s)^{-1/4}\|\p_v e^{sL}u_0\|_2\d s.
  \end{equation*}
  and then, given the impossibility of completing the estimate, by
  having inside the integral a different norm that the one in the
  l.h.s of the inequality, we need to control the integral term by
  applying the bound \eqref{eq:bound-linear-complete-derivative-2-2},
  $ \| \p_ve^{tL} u_0 \|_2 \leq C t^{-1/2} \|u_0\|_2 $, so that we
  get:
  \begin{equation*}
    \|e^{tL}u_0\|_\infty \le Ct^{-1/4}\|u_0\|_2+C\int_{0}^{t}(t-s)^{-1/4}s^{-1/2}\|u_0\|_2\d s=Ct^{-1/4}\|u_0\|_2,
  \end{equation*}
  where we solved the integral and we put the constants together, in
  order to proof the inequality
  \eqref{eq:bound-linear-complete-2-infinity}.
	
  \
	
  Going backwards and starting again from inequality
  \eqref{eq:complete-norm-infinity} the application of the estimate
  \eqref{eq:bound-linear-incomplete-infty-infty},
  $ \| e^{tA} u_0 \|_\infty \leq C \|u_0\|_\infty $, gives
  \begin{equation*}
    \|e^{tL}u_0\|_\infty \le C\|u_0\|_\infty+C\int_{0}^{t}\|\p_v e^{sL}u_0\|_\infty \d s,
  \end{equation*}
  where, using the inequality
  \eqref{eq:bound-linear-infty-infty-derivative},
  $\| \p_ve^{tL} u_0 \|_\infty \leq Ct^{-1/2} \|u_0\|_\infty$ and
  Lemma \ref{lem:gronwall-continuous} with $\alpha=\beta=0$ and
  $\gamma=1/2$ proves the inequality
  \eqref{eq:bound-linear-infty-infty}.

  \medskip
  
  To show the last estimate, let us take the $L^2$ norm in equation
  \eqref{eq:linear-complete}, getting to
  \begin{equation*}
    \|e^{tL}u_0\|_2 \le \|e^{tA}u_0\|_2+C\int_0^t\|e^{(t-s)A}\left(\p_v e^{sL}u_0\right)\|_2 \d s.
  \end{equation*}
  First we use the inequality \eqref{eq:bound-linear-incomplete-2-2},
  $ \| e^{tA} u_0 \|_2 \leq C \|u_0\|_2 $, in both terms of the r.h.s
  and then we apply the formula
  \eqref{eq:bound-linear-complete-derivative-2-2},
  $ \| \p_ve^{tL} u_0 \|_2 \leq C t^{-1/2} \|u_0\|_2 $, to the term
  inside the integral, as it is shown below:
  \begin{equation*}
    \begin{split}
      \|e^{tL}u_0\|_2 \le C\|u_0\|_2+C\int_0^t\|\p_v u(.,s)\|_2 \d s \\
	\le C\|u_0\|_2+C\| u_0\|_2\int_0^ts^{-1/2}\d s \le C\| u_0\|_2,
      \end{split}		
    \end{equation*}
    thus proving inequality \eqref{eq:bound-linear-2-2}.
\end{proof}

Finally we give some Gaussian bounds for the complete linear equation:

\begin{lem}[$L^2$ Gaussian estimates for the full linear equation]
  \label{lem:LIF-Gaussian-L2-estimates}
  Take $v_0 \in \R$ and let $\varphi(v) := \exp((v-v_0)^2 / 2)$.
  There exist $C > 0$ and $t_0 > 0$ depending only on $N$, $v_0$, $V_R$ and $V_F$
  such that
  \begin{equation}
    \label{eq:LIF-Gaussian-estimate-1}
    \| \p_v e^{tL} u_0 \|_{L^2(\varphi)}
    \leq C \| \p_v u_0 \|_{L^2(\varphi)}
    \qquad
    \text{for all $0 \leq t \leq t_0$}.
  \end{equation}
  and
  \begin{equation}
    \label{eq:LIF-Gaussian-estimate-2}
    \| \p_v e^{tL} u_0 \|_{L^2(\varphi)}
    \leq C t^{-1/2} \| u_0 \|_{L^2(\varphi)}
    \qquad
    \text{for all $0 < t \leq t_0$,}
  \end{equation}
  both for any $u_0 \in X$.
\end{lem}

\begin{proof}
  As usual, call $u(\cdot, t) = e^{tL} u_0$. We prove the first
  estimate by taking the $L^2(\varphi)$ norm in
  \eqref{eq:linear-complete-derivative}, getting
  \begin{equation}
    \label{eq:2}
    \|\p_v u(\cdot,t) \|_{L^2(\varphi)} \le
    \|\p_v u(\cdot,t) \|_{L^2(\varphi)}
    + C \int_0^t\|\p_v e^{(t-s)A}\left(\p_v u(\cdot,s)\right)\|_{L^2(\varphi)}\d s.
  \end{equation}
  We use now \eqref{eq:FP-delta-Gaussian-estimate-1} from Lemma
  \ref{lem:FP-delta-Gaussian-L2-estimates} on the first term, and
  \eqref{eq:FP-delta-Gaussian-estimate-2}  on the term inside the
  integral:
  \begin{equation*}
    \|\p_v u(\cdot,t)\|_{L^2(\varphi)} \lesssim
    \|\p_v u_0 \|_{L^2(\varphi)}
    + \int_0^t (t-s)^{-\frac12} \| \p_v u(\cdot,s) \|_{L^2(\varphi)} \d s,
  \end{equation*}
  valid for $t$ small enough. Lemma \ref{eq:gronwall-before} applied
  to $z(t) = \| \p_v u(\cdot,t) \|_{L^2(\varphi)}$, with
  $\beta = \frac12$ and $\alpha = \gamma = 0$ immediately gives
  \eqref{eq:LIF-Gaussian-estimate-1}.

  In order to get \eqref{eq:LIF-Gaussian-estimate-2} we continue from
  \eqref{eq:2} and use now \eqref{eq:FP-delta-Gaussian-estimate-2} on
  both terms on the right hand side:
  \begin{equation*}
    \|\p_v u(\cdot,t)\|_{L^2(\varphi)} \lesssim
    t^{-\frac12} \| u_0 \|_{L^2(\varphi)}
    + \int_0^t (t-s)^{-\frac12} \| \p_v u(\cdot,s) \|_{L^2(\varphi)} \d s,
  \end{equation*}
  or equivalently, setting $z(t) := t^{\frac12}\|\p_v u(\cdot,t)\|_{L^2(\varphi)}$,
  \begin{equation*}
     z(t) \lesssim
    \| u_0 \|_{L^2(\varphi)}
    + t^{\frac12} \int_0^t (t-s)^{-\frac12} s^{-\frac12} z(s) \d s,
  \end{equation*}
  again valid for small enough $t$. Lemma \ref{eq:gronwall-before}
  with $\alpha = \beta = \gamma = \frac12$ now proves the result.
\end{proof}


In the following lemma we prove the Gronwall-type inequality used in
the previous results. The case with $\beta=0$ can be deduced from the
standard Gronwall's inequality, and the case $\alpha = \gamma$ can also
be obtained as a consequence of the techniques in Section
\ref{sec:volterra}. However, since the focus of this result is the
short-time behavior and the proof in that case is very
straightforward, we prefer to give an independent proof now:

\begin{lem}[Short-time Gronwall-type inequality]
  \label{lem:gronwall-continuous}
  Take $T > 0$ and $z \in L^\infty[0,T]$ which satisfies
  \begin{equation}
    \label{eq:gronwall-before}
    z(t)\le
    k +
    t^{\alpha}
    \int_{0}^{t}\left(t-s\right)^{-\beta}s^{-\gamma}z(s)\d s,
    \qquad \text{for all $t \in [0,T]$,}
  \end{equation}
  with $k, \alpha,\beta$ and $\gamma$ constants such that $k > 0$;
  $\beta, \gamma \in [0,1)$ and $0< \alpha-\beta-\gamma+1$.  Then
  there exist $t_0 \in (0,T]$ and $C>0$ depending only on $\alpha$,
  $\beta$ and $\gamma$ such that
  \begin{equation*}
    z(t) < Ck \qquad
    \text{for all $t \in [0,t_0]$.}
  \end{equation*}
\end{lem}

\begin{proof}
  Let us define the function
  $y(t):=\max_{\tau\in\left[0,t\right]}z(\tau)$, for which equation
  \eqref{eq:gronwall-before} implies
  \begin{equation*}
    y(t)
    \leq
    k+y(t)t^\alpha\int_{0}^t\left(t-s\right)^{-\beta}s^{-\gamma}\d s
    =
    k + y(t) t^{\alpha-\beta-\gamma+1} C_{\beta, \gamma}
  \end{equation*}
  where $C_{\beta,\gamma} := B(1-\beta, 1-\gamma)$, the Beta function
  evaluated at $(1-\beta, 1-\gamma)$.  Now, using that
  $\alpha - \beta - \gamma + 1 > 0$, we may take $t_0 \in (0,T]$ small
  enough so that $t_0^{\alpha-\beta-\gamma+1} C_{\beta, \gamma} <
  1$. Then for any $t \in [0,t_0]$ we obtain $ y(t) \le k C $ with
  $C := (1 - t_0^{\alpha-\beta-\gamma+1} C_{\beta, \gamma})^{-1}$.
\end{proof}

\section{Asymptotics of Volterra-type equations}
\label{sec:volterra}

We devote this section to the following integral equation with unknown
$f$:
\begin{equation}
  \label{eq:renewal}
  f(t) = g(t) + \int_0^t f(s) h(t-s) \d s
  \qquad \text{for $t \geq 0$}.
\end{equation}
Understanding this equation is essential to the proofs of some of the
main results of this paper, especially those in Sections
\ref{sec:nonlinear} and \ref{sec:linearized-sg}. Equation
\eqref{eq:renewal} is a Volterra's convolution integral equation of the
second kind, also known as the \emph{renewal equation} in
the context of delay differential equations \cite[Chapter
I]{Diekmann1995}.

\subsection{Exact asymptotic behavior}

The following result is a version of \cite[Theorem 5.4]{Diekmann1995},
which we need to adapt since the technical assumptions on the
functions $g$ and $h$ do not match our case. The strategy of the proof
is however essentially the same. Our main result is the following:

\begin{thm}
  \label{thm:asymptotic}
  Assume that $g, h \: [0,+\infty) \to \R$ are given functions, with
  $g$ being continuous and of bounded variation on compact subsets of
  $(0,+\infty)$, such that for some $\alpha \in \R$ we have
  \begin{align}
    \label{eq:c1}
    &h \in L^1([0,+\infty); e^{-\alpha t}) \cap L^p([0,+\infty);
      e^{-\alpha t})
      \qquad \text{for some $p > 1$},
    \\
    \label{eq:c2}
    &t \mapsto g(t) e^{-\alpha t}
    \qquad \text{is bounded on $[0,+\infty)$.}
  \end{align}
  Let $f$ be the solution to equation \eqref{eq:renewal}, and consider the
  function
  \begin{equation*}
    F(k) := \frac{\l{g}(k)}{1 - \l{h}(k)}
    \qquad \text{for $\Re(k) \geq \alpha$},
  \end{equation*}
  which is a meromorphic function on the set of $k$ with
  $\Re(k) > \alpha$. Then:
  \begin{enumerate}
  \item If $F$ has no poles on the set
    $\{k \in \C \mid \Re(k) > \alpha\}$, then for every
    $\beta > \alpha$ there is a constant $C_\beta > 0$ (depending on
    $g$, $h$ and $\beta$) such that
    \begin{equation*}
      |f(t)| \leq C_\beta e^{\beta t}
      \qquad \text{for all $t \geq 1$.}
    \end{equation*}

  \item If $F$ has some pole of order $N$ on the set
    $R_\alpha := \{k \in \C \mid \Re(k) > \alpha\}$, then there is a polynomial
    $p = p(t)$ of degree $N-1$ such that
    \begin{equation*}
      f(t) \sim p(t) e^{\beta t}
      \qquad \text{as $t \to +\infty$,}
    \end{equation*}
    where $\beta$ is the largest real part of the poles of $F$ on
    $R_\alpha$.
  \end{enumerate}
  In addition, the constant $C_\beta$ and the coefficients $a_1,
  \dots, a_N$ of $p$ can be chosen so that
  \begin{equation*}
    \max\{|a_1|, \dots, |a_N| \}
    \leq
    C_{h} \sup_{t \geq 0} |g(t)| e^{-\alpha t},
    \qquad
    C_\beta \leq C_{h,\beta} \sup_{t \geq 0} |g(t)| e^{-\alpha t},
  \end{equation*}
  where $C_h$ is a constant which depends only on $h$, and
  $C_{h,\beta}$ depends only on $h$ and $\beta$.
\end{thm}

\begin{proof}
  For $\Re(k)$ large enough (let us say for
  $\Re(k) \geq \gamma > \alpha$) we may take the Laplace transform of
  \eqref{eq:renewal} and obtain
  \begin{equation*}
    \l{f}(k) = \frac{\l{g}(k)}{1 - \l{h}(k)}
    \qquad \text{for $\Re(k) \geq \gamma$}.
  \end{equation*}
  Notice that we know $\l{f}(k) = F(k)$ for $\Re(k) \geq \gamma$, but
  we cannot say that $\l{f}(k)$ even makes sense for a general $k$
  with $\Re(k) \geq \alpha$.

  Take $R > 0$ large enough and let us define the positively oriented
  curve
  \begin{equation*}
    \Gamma := [\gamma -iR, \gamma + iR] \cup
    [\gamma + iR, \alpha + iR] \cup
    [\alpha + iR, \alpha - iR] \cup
    [\alpha - iR, \gamma - iR].
  \end{equation*}
  The function $F$ is analytic on its domain, so we may use Cauchy's
  integral theorem on the curve $\Gamma$ to obtain that, for any $t > 0$,
  \begin{multline}
    \label{eq:5}
    2 \pi i \sum{\Res(F(k) e^{kt}, z_n)} = \int_\Gamma e^{k t} F(k) \d k
    =
    \int_{\gamma -iR}^{\gamma + iR} e^{kt} F(k) \d k
    -   e^{i R t} \int_{\alpha}^{\gamma} e^{s t} F(s + i R) \d s
    \\
    - \int_{\alpha -iR}^{\alpha + iR} e^{kt} F(k) \d k
    +   e^{- i R t} \int_{\alpha}^{\gamma} e^{s t} F(s - i R) \d s
    =: I_1 + I_2 + I_3 + I_4,
  \end{multline}
  where the residue sum is over all zeros of $F$ inside the curve
  $\Gamma$. From equation \eqref{eq:renewal} one can see that $f$ must
  be continuous and of bounded variation on compact subsets of
  $[0,+\infty)$ (since $h$ is locally integrable and $g$ is continuous
  and locally of bounded variation). Hence we can apply the inversion
  theorem for the Laplace transform (Lemma \ref{lem:inverse_Laplace})
  to get
  \begin{equation}
    \label{eq:I1}
    \lim_{R \to +\infty} I_1
    =
    2 \pi i f(t).
  \end{equation}
  On the other hand, Riemann-Lebesgue's lemma shows that
  \begin{equation*}
    \lim_{R \to +\infty} \l{g}(s+iR) = \sqrt{2 \pi} \lim_{R \to
      +\infty} \F_t( e^{-st} g(t)) (R) = 0,
  \end{equation*}
  uniformly for $s \in [\alpha, \gamma]$ (one can see that the
  arguments in the standard proof of the Riemann-Lebesgue lemma can be
  made uniformly for $s \in [\alpha, \gamma]$). For a similar reason
  we have
  \begin{equation*}
    \lim_{R \to +\infty} \l{h}(s+iR) = 0
  \end{equation*}
  uniformly for $s \in [\alpha, \gamma]$. This shows that
  \begin{equation*}
    \lim_{R \to +\infty} F(s + iR) =
    \lim_{R \to +\infty} \frac{\l{g}(s+iR)}{1 - \l{h}(s+iR)}
    = 0
  \end{equation*}
  uniformly for $s \in [\alpha, \gamma]$, and hence that
  \begin{equation}
    \label{eq:I2}
    \lim_{R \to +\infty} I_2
    = 0.
  \end{equation}
  By symmetric arguments applied to $s-iR$,
  \begin{equation}
    \label{eq:I4}
    \lim_{R \to +\infty} I_4
    = 0.
  \end{equation}
  Finally, for the third integral we have (see Lemma
  \ref{lem:F_inverse_bounded})
  \begin{equation}
    \label{eq:I3}
    \left| \lim_{R \to +\infty} I_3 \right| =
    e^{\alpha t} \left|
      \lim_{R \to +\infty} \int_{-R}^R e^{i \xi t} F(\alpha + i \xi) \d \xi
    \right|
    \leq C e^{\alpha t}.
  \end{equation}
  Regarding the residue of $F$ around $z_n$, Lemma \ref{lem:residue}
  shows that for some nontrivial polynomial $p_n$,
  \begin{equation}
    \label{eq:Ires}
    \Res_{k=z_n} (F(k) e^{kt}) = p_n(t) e^{z_n t}.
  \end{equation}
  Using \eqref{eq:I1}, \eqref{eq:I2}, \eqref{eq:I4}, \eqref{eq:I3} and
  \eqref{eq:Ires} in \eqref{eq:5} and passing to the limit as
  $R \to +\infty$ shows the result.
\end{proof}

The following result is a version of the Fourier inversion theorem
taken from \cite[Theorem 24.4]{Doetsch1974}:
\begin{lem}[Inverse Laplace transform]
  \label{lem:inverse_Laplace}
  Let $f \: [0,+\infty) \to \R$ and $\alpha \in \R$ be such that
  $t \mapsto f(t) e^{-\alpha t}$ is absolutely integrable. Let $t > 0$
  be a point where $f$ has bounded variation in some neighbourhood of
  $t$. Then for every $\beta \geq \alpha$ we have
  \begin{equation*}
    \lim_{R \to +\infty} \frac{1}{2 \pi i}
    \int_{\beta-i R}^{\beta + i R} e^{kt} \l{f}(k) \d k
    = \frac12 (f(t_+) + f(t-)),
  \end{equation*}
  where the integral is understood as a complex integral along the
  straight line with real part $R$.
\end{lem}
Of course, if additionally $f$ is continuous at $t$, then
\begin{equation*}
  \lim_{R \to +\infty} \frac{1}{2 \pi i}
  \int_{\beta-i R}^{\beta + i R} e^{kt} \l{f}(k) \d k
  = f(t).
\end{equation*}

The following lemma gives the result of calculating the residue of the
complex function $F$ from the proof of Theorem \ref{thm:asymptotic} at
a given pole. It is essentially the same as Lemma 5.1 in \cite{Diekmann1995}:

\begin{lem}[Residue close to a pole]
  \label{lem:residue}
  Under the assumptions of Theorem \ref{thm:asymptotic}, define
  \begin{equation*}
    F(k) := \frac{\l{g}(k)}{1 - \l{h}(k)}
    \qquad 
    \qquad \text{for $\Re(k) \geq \alpha$}.
  \end{equation*}
  We notice $F$ is analytic on
  $R_\alpha := \{w \in \C \mid \Re(w) > \alpha\}$, except for possibly
  a set of isolated poles. Let $z \in \C$ be a pole of $F$ with
  $\Re(z) > \alpha$. There exists a nontrivial polynomial $p = p(t)$
  of degree $N-1$ such that
  \begin{equation*}
    \Res_{k = z} ( e^{k t} F(k) ) = p(t) e^{z t}
  \end{equation*}
  for any $t \in \R$. All coefficients $a_1, \dots, a_{N}$ of $p$ can
  be bounded by
  \begin{equation*}
    |a_n| \leq C \sup_{t \geq 0} |g(t)| e^{-\alpha t}
    \qquad \text{for $n = 1, \dots, N$,}
  \end{equation*}
  for some constant $C$ which depends only on $h$.
\end{lem}

\begin{proof}
  The residue at $k=z$ is equal to $2 \pi i$ times the coefficient of
  the power $-1$ in the Laurent expansion of $e^{k t} F(k)$ around
  $k = z$. Expanding both $F(k)$ and $e^{kt}$ close to $k = z$ we have
  \begin{equation*}
    F(k) = \sum_{n = -N}^\infty a_n (k - z)^n,
    \qquad
    e^{k t} = e^{zt} e^{(k-z)t}
    = e^{zt} \sum_{n=0}^\infty  \frac{t^n}{n!} (k-z)^n,
  \end{equation*}
  with $a_{-N} \neq 0$. Multiplying out these series we obtain that
  the coefficient of the power $-1$ in the Laurent expansion of
  $e^{kt} F(k)$ is
  \begin{equation*}
    \frac{1}{2 \pi i} p(t) := \sum_{j=0}^{N-1} a_{-j-1} \frac{t^{j}}{j!}.
  \end{equation*}
  We see that $p$ is of degree $N-1$, since the coefficient $a_{-N}$
  is nonzero. Its coefficients can easily be bounded in the following
  way: choose a closed circle $\gamma$ of radius $r > 0$, centered on
  $z$, and which does not enclose any other zero of $1-\l{h}$. We may
  also choose $\gamma$ such that it is contained in
  $\{ w \in \C \mid \Re(w) \geq \beta\}$ for some $\beta>
  \alpha$. Then for $n = 1, \dots, N$
  \begin{multline*}
    |a_n| = \frac{1}{2\pi} \left| \int_\gamma \frac{F(w)}{(w-z)^{n+1}}
      \d w \right| \leq \frac{1}{2\pi} r^{n+1} \int_\gamma |F(w)| \d w
    \\
    \leq \frac{1}{2\pi} r^{n+1} C_h \int_\gamma |\l{g}(w)| \d w \leq
    r^{n+2} C_h \int_0^\infty g(t) e^{-\beta t} \d t.
  \end{multline*}
  We notice that $C_h := \inf_\gamma |1 - \l{h}|$ only depends on $h$,
  and the choice of $r$ and $\beta$ also depends on $h$ only. This
  easily gives the bound in the statement.
\end{proof}

\begin{lem}
  \label{lem:F_inverse_bounded}
  Assume that $g, h \: [0,+\infty) \to \R$ are given functions, with
  $g$ being of bounded variation on compact subsets of $(0,+\infty)$,
  such that for some $\alpha_0 \in \R$ we have
  \begin{align}
    \label{eq:7}
    &h \in L^1([0,+\infty); e^{-\alpha_0 t}) \cap L^p([0,+\infty);
      e^{-\alpha_0 t})
      \qquad \text{for some $p > 1$},
    \\
    \label{eq:8}
    &t \mapsto |g(t)| e^{-\alpha_0 t}
    \qquad \text{is bounded on $[0,+\infty)$.}
  \end{align}
  Take any $\alpha > \alpha_0$, and assume there are no zeros of the
  function $1 - \l{h}(k)$ on the line with $\Re(k)=\alpha$. Then the
  function
  \begin{equation*}
    F(k) :=  \frac{\l{g}(k)}{1 - \l{h}(k)}
  \end{equation*}
  satisfies that there exists $C > 0$ such that
  \begin{equation}
    \label{eq:6}
    \left |
      \lim_{R \to +\infty}
      \int_{-R}^R e^{i \xi t} F(\alpha + i \xi) \d \xi
    \right|
    \leq C
    \qquad \text{for all $t > 0$.}
  \end{equation}
  In addition, the constant $C$ can be chosen as $C = C_h \sup_{t \geq
  0} |g(t)| e^{-\alpha_0 t}$, where $C_h$ depends only on $h$.
\end{lem}

The difficulty in proving the above statement is that in general we do
not know that $F(k)$ is absolutely integrable on the line with
$\Re(k) = \alpha$; otherwise, we would bound the above integral by the
integral of $|F(k)|$ over that line, ignoring $e^{i \xi
  t}$. Similarly, we do not know whether $F$ on the line with
$\Re(k) = \alpha$ is the Laplace transform of any function $f$;
otherwise, we would use the inversion lemma \ref{lem:inverse_Laplace}
to say that the limit must be $e^{-\alpha t} f(t)$, and we would just
need to check whether the latter is a bounded function. Instead, what
we can do is write $F$ as a sum of $F_1$ and $F_2$, where $F_1$ is
absolutely integrable on the line and $F_2$ is the Laplace transform
of a known function. That is the main idea of the following proof:

\begin{proof}[Proof of Lemma \ref{lem:F_inverse_bounded}]
  We first observe that for any $\xi \in \R$,
  \begin{equation*}
    \l{h}(\alpha + i \xi) = \sqrt{2 \pi} \F_t(e^{-\alpha t} h(t)) (\xi),
  \end{equation*}
  so by the Riemann-Lebesgue lemma we have $\l{h}(\alpha + i \xi) \to
  0$ as $|\xi| \to +\infty$. Since $\xi \mapsto \l{h}(\alpha + i \xi)$
  is continuous and $1 - \l{h}(\alpha + i \xi)$ has no zeros for $\xi
  \in \R$, it is clear that for some $C_2 > 0$,
  \begin{equation}
    \label{eq:denom-bound}
    \frac{1}{|1 - \l{h}(\alpha + i \xi)|} \leq C_2
    \qquad \text{for all $\xi \in \R$,}
  \end{equation}
  Choosing an integer $N \geq 1$ we can rewrite
  \begin{equation*}
    F = \frac{\l{h}^N \l{g}}{1 - \l{h}}
    + \l{g} \sum_{n=0}^{N-1} \l{h}^n = F_1 + F_2.
  \end{equation*}
  The second term $F_2$ is the Laplace transform of
  $f_2 := g + g * (\sum_{n=1}^{N-1} h^{*n})$. Lemma
  \ref{lem:inverse_Laplace} then shows that
  \begin{equation*}
    \lim_{R \to +\infty} \frac{1}{2 \pi i}
    \int_{\beta-i R}^{\beta + i R} e^{kt} F_2(k) \d k
    = \frac12 (f_2(t_+) + f_2(t_-))
    \qquad \text{for all $t \geq 0$.}
  \end{equation*}
  Notice that Lemma \ref{lem:inverse_Laplace} is applicable because
  $f_2$ is of bounded variation on compact sets (which can be seen
  from \eqref{eq:7} and \eqref{eq:8}). Hence for all $t > 0$ we have
  \begin{equation}
    \label{eq:9}
    \left|
      \lim_{R \to +\infty}
      \int_{-R}^R e^{i \xi t} F_2(\alpha + i \xi) \d \xi
    \right|
    = \pi e^{-\alpha t} (|f_2(t_+)| + |f_2(t_+)|) \leq C_3,
  \end{equation}
  for some $C > 0$. (We notice that $t \mapsto |f_2(t)| e^{-\alpha t}$
  is bounded for all $t > 0$ due to \eqref{eq:7} and \eqref{eq:8},
  so the same is true if we write $t_+$ or $t_-$ instead of $t$.) Now
  for the term $F_1$, we use \eqref{eq:denom-bound} to write
  \begin{multline*}
    \int_{-\infty}^{+\infty} |F_1(\alpha + i\xi)| \d \xi
    \leq
    C_2 \int_{-\infty}^{+\infty} |\l{h}(\alpha + i\xi)|^N
    |\l{g}(\alpha + i\xi)| \d \xi
    \\
    =
    2 \pi C_2 \int_{-\infty}^{+\infty} \F_t(e^{-\alpha t} h)(\xi)^N
    \F_t(e^{-\alpha t} g)(\xi) \d \xi
    \\
    \leq
    C_2 \|\F_t(e^{-\alpha t} h)\|_{2 N}^N \|
    \F_t(e^{-\alpha t} g) \|_2
    \leq C_2 \|e^{-\alpha t} h\|_{p}^N \|e^{-\alpha t} g\|_2,
  \end{multline*}
  where we have used the Hausdorff-Young inequality and $p$ is such
  that $\frac{1}{p} + \frac{1}{2N} = 1$. Now, $\|g\|_2$ is finite, and
  we may choose $N$ large enough to make $p$ small enough so that
  $\|e^{-\alpha t} h\|_{p}$ is finite (by hypothesis). So we see that
  \begin{equation*}
    \left|
      \int_{-R}^R e^{i \xi t} F_1(\alpha + i \xi) \d \xi
    \right|
    \leq C_4
  \end{equation*}
  for all $t \geq 0$, and also that this integral has a limit as
  $R \to +\infty$, for all $t \geq 0$. From this and \eqref{eq:9} we
  see that
  \begin{equation*}
    \lim_{R \to +\infty} \left|
      \int_{-R}^R e^{i \xi t} F(\alpha + i \xi) \d \xi
      \right|
    \leq C := \max\{C_3, C_4\},
  \end{equation*}
  for all $t > 0$. This shows the result. One may also check that all
  constants which appear above are consistent with the form
  $C = C_h \sup_{t \geq 0} |g(t)| e^{-\alpha_0 t}$ given in the
  statement.
\end{proof}

\subsection{Comparison results and special cases}

We also have a comparison result for solutions to inequalities of this
type. 
\begin{lem}[Comparison result]
  \label{lem:volterra-comparison}
  Take $T \in (0,+\infty]$ and let $g \: [0,T) \to \R$ and
  $h \: [0,T) \to [0,+\infty)$ be locally integrable functions. If
  $f_1, f_2 \: [0,T) \to \R$ are locally bounded functions (i.e.,
  $L^\infty$ on compact intervals) such that
  \begin{gather*}
    f_2(t) \geq g(t) + \int_0^t f_2(s) h(t-s) \d s
    \qquad \text{for $t \in [0,T)$},
    \\
    f_1(t) \leq g(t) + \int_0^t f_1(s) h(t-s) \d s
    \qquad \text{for $t \in [0,T)$},
  \end{gather*}
  then
  \begin{equation*}
    f_1(t) \leq f_2(t)
    \qquad \text{for $t \in [0,T)$}.
  \end{equation*}
\end{lem}

\begin{proof}
  Choose $t_0 < T$ such that
  \begin{equation*}
    \delta := \int_0^{t_0} h(s) \d s < 1.
  \end{equation*}
  The function $w(t) := f_1(t) - f_2(t)$ satisfies
  \begin{equation*}
    w(t) \leq \int_0^t w(s) h(t-s) \d s.
  \end{equation*}
  Our aim is to show that $w(t) \leq 0$ for all $t \in [0,T)$. Taking
  the essential supremum on $[0,t_0]$ and calling
  $M := \esssup_{t \in [0,t_0]} w(t)$ gives
  \begin{equation*}
    M \leq M \int_0^{t_0} h(t-s) \d s
    =
    M \delta,
  \end{equation*}
  which implies $M \leq 0$ and hence $w(t) \leq 0$ for almost
  all $t \in [0,t_0]$. But then the inequality is also true for all $t
  \in [0,t_0]$ since
  \begin{equation*}
    w(t) \leq \int_0^t w(s) h(t-s) \d s
    \leq
    0
    \qquad
    \text{for all $t \in [0,t_0]$.}
  \end{equation*}
  One can then repeat the same argument for the function
  $w_{t_0} \: [0, T-t_0) \to \R$ defined by $w_{t_0}(t) := w(t+t_0)$,
  which satisfies
  \begin{align*}
    w_{t_0}(t)
    &\leq
    \int_0^{t_0} w(s) h(t+t_0-s) \d s
    +
    \int_{t_0}^{t+t_0} w(s) h(t+t_0-s) \d s
    \\
    &\leq
    \int_{0}^{t} w_{t_0}(s) h(t-s) \d s.
  \end{align*}
  With this we obtain that $w(t) \leq 0$ for
  $t \in [0,2 t_0) \cap [0,T)$. Iterating this argument we obtain the
  result.
\end{proof}

\begin{lem}[Comparison result with delay]
  \label{lem:volterra-comparison-delay}
  Take $d > 0$, $T \in (0,+\infty]$ and let $g \: [-d,T) \to \R$ and
  $h \: [-d,T) \to [0,+\infty)$ be locally integrable functions. If
  $f_1, f_2 \: [-d,T) \to \R$ are locally bounded functions (i.e.,
  $L^\infty$ on compact intervals) such that
  \begin{equation*}
    f_1(t) \leq f_2(t) \qquad \text{for $t \in [-d,0]$}
  \end{equation*}
  and
  \begin{gather*}
    f_2(t) \geq g(t) + \int_0^t f_2(s-d) h(t-s) \d s
    \qquad \text{for $t \in [0,T)$},
    \\
    f_1(t) \leq g(t) + \int_0^t f_1(s-d) h(t-s) \d s
    \qquad \text{for $t \in [0,T)$},
  \end{gather*}
  then
  \begin{equation*}
    f_1(t) \leq f_2(t)
    \qquad \text{for $t \in [-d,T)$}.
  \end{equation*}
\end{lem}

\begin{proof}
  The proof is easier in this case and can be done by induction: we
  show that 
  $f_1(t) \leq f_2(t)$ for $t \in [-d, nd] \cap [0,T)$, for all integers $n \geq
  0$. The case $n=0$ holds by assumption, and if the case $n$ holds
  then for $t \in [nd, (n+1)d] \cap [0,T)$ we have
  \begin{equation*}
    f_1(t)
    \leq
    g(t) + \int_0^t f_1(s-d) h(t-s) \d s
    \leq
    g(t) + \int_0^t f_2(s-d) h(t-s) \d s
    \leq
    f_2(t),
  \end{equation*}
  where the intermediate inequality holds because $s-d$ is always in
  $[-d, nd]$, where the inequality is already assumed to hold.  
\end{proof}

\bigskip As a consequence we obtain a ``modified Gronwall's
inequality'' which is needed below. It can also be deduced as a
consequence of Theorem \ref{thm:asymptotic}, but we give a
direct proof since it is easier in this case:

\begin{lem}[Convolution Gronwall inequality]
  \label{lem:gronwall-modified}
  Let $0<\alpha<1 $ and $ 0 < T \le \infty $. If
  $f \: \left[0,T\right)\to \left[0,\infty\right) $ is a locally
  bounded function which satisfies
  \begin{equation}
    \label{eq:gronwall-modified}
    f(t)\le A+B\int_{0}^{t}\left(t-s\right)^{-\alpha}f(s) \d s
    \qquad
    \text{for all $ t\in \left[0,T\right) $,}
  \end{equation}
  for some constants $A, B > 0$, then there exists $\mu = \mu(\alpha) > 0$ such that
  \begin{equation*}
    f(t) \le 2 A \exp\left( B^{\frac{1}{1-\alpha}} \mu t \right)
    \qquad \text{for all $ t\in \left[0,T\right) $.}
  \end{equation*}
\end{lem}

\begin{proof}
  We first notice that it is enough to prove the result for
  $A=B=1$. Then the general case can be obtained by noticing that if
  $f$ satisfies \eqref{eq:gronwall-modified} then the function
  $g(t) := \frac{1}{A} f(\beta t)$ satisfies the same inequality with
  $A = B=1$ if we choose $\beta := B^{-\frac{1}{1-\alpha}}$.

  In order to prove the result when $A = B = 1$, let us find $C$ and
  $\mu$ such that $F(t) := C e^{\mu t}$ is a supersolution of the
  corresponding Volterra's equation; that is, such that
  \begin{equation}
    \label{eq:1}
    F(t) \geq
    1 + \int_0^t (t-s)^{-\alpha} F(s) \d s
    \qquad \text{for all $t\in [0,\infty)$.}
  \end{equation}
  This is not hard, since
  \begin{multline*}
    \int_0^t (t-s)^{-\alpha} F(s) \d s
    =
    C \int_0^t (t-s)^{-\alpha} e^{\mu s} \d s
    \\
    =
    C e^{\mu t} \int_0^t (t-s)^{-\alpha} e^{-\mu(t-s)} \d s
    \leq
    C e^{\mu t} \mu^{\alpha-1} \Gamma(1-\alpha) =: K_{\alpha,\mu} F(t),
  \end{multline*}
  where $K_{\alpha,\mu} := \mu^{\alpha-1} \Gamma(1-\alpha)$. Hence
  \begin{equation*}
    1 + \int_0^t (t-s)^{-\alpha} F(s) \d s
    \leq
    \frac1{C} F(t) + K_{\alpha,\mu} F(t) = (\frac{1}{C} + K_{\alpha,\mu}) F(t).
  \end{equation*}
  So equation \eqref{eq:1} holds if $\frac{1}{C} + K_{\alpha,\mu} \leq
  1$. It is enough to take $C = 2$ and
  \begin{equation*}
    K_{\alpha,\mu} = \mu^{\alpha-1} \Gamma(1-\alpha) = \frac12,
    \qquad \text{that is,} \qquad
    \mu := \left(2\Gamma(1-\alpha) \right)^{\frac{1}{1-\alpha}}.
  \end{equation*}
  For this choice of $C$ and $\mu$, the function $F(t)$ is a
  supersolution of the Volterra's equation in the statement, and then
  Lemma \ref{lem:volterra-comparison} shows that $f(t) \leq F(t)$ on
  $[0,T)$.
\end{proof}

\begin{lem}[Convolution Gronwall's inequality with delay]
  \label{lem:gronwall-modified-delay}
  Let $d \geq 0$, $0<\alpha<1 $ and $ 0 < T \le \infty $. If
  $f \: \left[-d,T\right)\to \left[0,\infty\right) $ is a locally
  bounded function which satisfies
  \begin{equation}
    \label{eq:gronwall-modified-delay}
    f(t)
    \le
    A+B\int_{0}^{t}\left(t-s\right)^{-\alpha}f(s-d) \d s
    \qquad
    \text{for all $ t\in \left[0,T\right) $,}
  \end{equation}
  for some constants $A, B > 0$, then
  \begin{equation*}
    f(t) \le (2 A + M_0)\, e^{\mu t }
    \qquad \text{for all $ t\in \left[0,T\right) $,}
  \end{equation*}
  where $\mu$ is any positive number satisfying
  \begin{equation*}
    B e^{-\mu d} \mu^{\alpha-1}  \leq \frac{1}{2 \Gamma(1-\alpha)}
  \end{equation*}
  and
  \begin{equation*}
    M_0 := \sup_{t \in [-d,0]} f(t).
  \end{equation*}
\end{lem}

\begin{proof}
  Let us find $C$ and $\mu$ such that $F(t) := C e^{\mu t}$ is a
  supersolution of the corresponding Volterra's equation; that is, such
  that $F(t) \geq f(t)$ for $t \in [-d,0]$ and
  \begin{equation}
    \label{eq:1-1}
    F(t) \geq
    A + B \int_0^t (t-s)^{-\alpha} F(s-d) \d s
    \qquad \text{for all $t\in [0,\infty)$.}
  \end{equation}
  The condition $F(t) \geq f(t)$ for $t \in [-d,0]$ is satisfied if we
  take $C \geq \sup_{t \in [-d,0]} f(t) =: M_0$. And the second condition can
  also be satisfied, since
  \begin{multline*}
    \int_0^t (t-s)^{-\alpha} F(s-d) \d s
    =
    C e^{-\mu d} \int_0^t (t-s)^{-\alpha} e^{\mu s} \d s
    \\
    =
    C e^{-\mu d} e^{\mu t} \int_0^t (t-s)^{-\alpha} e^{-\mu(t-s)} \d s
    \leq
    C e^{-\mu d} e^{\mu t} \mu^{\alpha-1} \Gamma(1-\alpha) =: K_{\alpha,\mu} F(t),
  \end{multline*}
  where $K_{\alpha,\mu} := e^{-\mu d} \mu^{\alpha-1}
  \Gamma(1-\alpha)$. Hence for $t \geq 0$,
  \begin{equation*}
    A + B \int_0^t (t-s)^{-\alpha} F(s-d) \d s
    \leq
    \frac{A}{C} F(t) + B K_{\alpha,\mu} F(t) = (\frac{A}{C} + B K_{\alpha,\mu}) F(t).
  \end{equation*}
  So equation \eqref{eq:1-1} holds if $\frac{A}{C} + B K_{\alpha,\mu} \leq
  1$. It is enough to take $C \geq 2A$ and $\mu > 0$ such that
  \begin{equation*}
    B K_{\alpha,\mu} = B e^{-\mu d} \mu^{\alpha-1} \Gamma(1-\alpha) \leq \frac12,
  \end{equation*}
  Hence for this choice of $\mu$ and with $C := \max\{2A, M_0 \}$, the
  function $F(t)$ is a supersolution of the Volterra's equation in the
  statement, and then Lemma \ref{lem:volterra-comparison} shows that
  $f(t) \leq F(t)$ on $[0,T)$.
\end{proof}

\section{Behavior of the nonlinear equation}
\label{sec:nonlinear}

In this section we consider the solution of the Cauchy problem
associated with \eqref{eq:NNLIF} with delay $d \ge 0$, whose related
stationary problem is given by \eqref{eq: large-delta-stationary},
with solution \eqref{eq:stationary-state}. We assume in this section
that $|b|$ is small enough for the solution to this stationary problem
to be unique.

\subsection{Weakly excitatory case}
\label{sec:weakly-connected}

\begin{thm}[Long-time behavior for weakly coupled
  systems]\label{thm:convergence-b-small}
  Consider $b\in\R$ with $|b|$ small enough for the equilibrium
  $p_\infty$ of \eqref{eq:NNLIF} to be unique, $V_R< V_F$, $d \geq 0$,
  a nonnegative initial data $p_0\in \mathcal{C}([-d,0], X)$, and $p$
  the corresponding solution to the nonlinear system
  \eqref{eq:NNLIF}. Let $K > 0$ and assume that the initial data $p_0$
  satisfies
  \begin{equation*}
    K_0 := \sup_{t \in [-d,0]} \|p_0(\cdot,t) - p_\infty \|_X \leq K,
  \end{equation*}
  where $p_\infty$ is the (unique) probability stationary solution of
  the nonlinear system \eqref{eq:NNLIF}.  There
  exist $C \geq 1$ depending only on $V_R$ and $V_F$ and
  $b_0>0$ depending on $V_R$, $V_F$ and $K$ such that if $|b|< b_0$
  then
  \begin{equation}
    \|p(.,t)-p_\infty\|_X\le C e^{-\frac{\lambda}{2+d}
      t}\|p_0(\cdot,0) - p_\infty\|_X
    \qquad \text{for all $t\ge 0$.}
  \end{equation}
  We emphasize that $b_0$ is independent of $d$.
\end{thm}

\begin{proof}
  We consider $ u(v,t):=p(v,t)-p_\infty(v) $ with
  $N_u(t) := N_p(t)-N_{p_\infty}$. From
  \eqref{eq:NNLIF} we see that $u$ satisfies
  \begin{equation*}
    \p_tu=\p^2_vu + \p_v(vu)+N_u(t)\delta_{V_R}
    -bN_{p_\infty}\p_vu-bN_u(t-d)\p_vp
    =:
    Lu-bN_u(t-d)\p_vp.
  \end{equation*}
  By Duhamel's formula we get the following expression for $ u(v,t) $:
  \begin{equation}
    \label{eq:Duhamel1}
    u(v,t) = e^{t L} u_0(v)
    -
    b \int_0^t N_u(s-d) e^{(t-s) L} \p_v p(v,s) \d s,
  \end{equation}
  where $ e^{tL} $ refers to the semigroup of the linear problem
  described in Section \ref{sec:linear}, and we denote
  $u_0(v) := u(v,0) = p_0(v,0) - p_\infty(v)$.
  We can apply the norm $ \|.\|_X $ to find
  \begin{equation*}
    \|u(.,t)\|_X \le \|e^{t L} u_0\|_X
    +
    |b| \int_0^t |N_u(s-d)| \|e^{(t-s) L} \p_v p(.,s)\|_X \d s.
  \end{equation*}
  We can apply \eqref{eq:gap-X} to the first term, and the
  spectral gap / regularization result \eqref{eq:gap-X-regularized} to
  the norm inside the integral, which leads to
  \begin{equation*}
    \|u(.,t)\|_X \le C e^{-\lambda t} \| u_0 \|_X
    +
    C |b|\int_0^t e^{-\lambda(t-s)} (t-s)^{-3/4}|N_u(s-d)|
    \| \p_v p(.,s) \|_2 \d s.
  \end{equation*}
  Considering that, for any $t$,
  $ |N_u(t)| \le \|\p_v u(.,t)\|_{\infty} \le \|u(.,t)\|_X $ and
  $ \| \p_v p(.,t) \|_2 \le \| p(.,t)\|_X \le \|u(.,t)\|_X +
  \|p_\infty\|_X $,  from our previous estimate for
  $ \|u(.,t)\|_X $ we obtain
  \begin{equation*}
    \|u(.,t)\|_X \le C e^{-\lambda t} \| u_0 \|_X
    +
    C|b|\int_0^t e^{-\lambda(t-s)} (t-s)^{-3/4}\|u(.,s-d)\|_X
    \left(\|u(.,s)\|_X + \bar{C}\right) \d s,
  \end{equation*}
  where $\bar{C} := \|p_\infty\|_X$. Let us now take
  $K_1 := (4 C + 1) K_0$ and define
  \begin{equation*}
    T := \sup \{t \geq 0 \mid \|u(\cdot, s)\|_X \leq K_1 \ \text{for $s \in
      [0,t]$} \}.
  \end{equation*}
  This $T$ must be strictly positive, and it may be $+\infty$; we will
  show later that with an appropriate choice of $b$ it must be
  $+\infty$. Calling $C_b:=C|b|\left(K_1+\bar{C}\right) $, we
  get to
  \begin{equation*}
    \|u(.,t)\|_X \le C e^{-\lambda t} \| u_0 \|_X
    +
    C_b\int_0^t e^{-\lambda(t-s)} (t-s)^{-3/4}\|u(.,s-d)\|_X \d s,
  \end{equation*}
  for all $0 \leq t < T$. We can now define
  $ f(t):=e^{\lambda t}\|u(.,t)\|_X $ in order to simplify previous
  expression, which becomes
  \begin{equation*}
    f(t) \le C f(0)
    +
    C_b e^{\lambda d} \int_0^t (t-s)^{-3/4}f(s-d) \d s.
  \end{equation*}
  By our modified Gronwall's Lemma \ref{lem:gronwall-modified-delay} with
  $A=C f(0)$, $B=C_b e^{\lambda d}$, $M_0=\sup_{t\in[-d,0]}f(t)$ and $\alpha = 3/4 $, we find the following bound for
  $f(t)$:
  \begin{equation*}
    f(t) \le (2 C f(0)+M_0) \exp( \mu t)
    \qquad \text{for all $0 \leq t < T$},
  \end{equation*}
  for any $\mu > 0$ satisfying
  \begin{equation}
    \label{eq:mu-condition}
    C_2 |b| e^{\lambda d} e^{-\mu d} \mu^{\alpha-1}
    \leq
    \frac{1}{2 \Gamma(1-\alpha)},
  \end{equation}
  with $C_2 := C(K_1 + \bar{C})$. Notice that
  $\sup_{t \in [-d, 0]} f(t) = \sup_{t \in [-d, 0]} e^{\lambda t}
  \|u(\cdot, t)\|_{X} \leq K_0$, and also $f(0) \leq K_0$. In terms of
  $ \|u(.,t)\|_X$ this implies
  \begin{equation*}
    \|u(.,t)\|_X \le
    (2C+1) K_0
    \exp( -(\lambda - \mu) t)
    \qquad \text{for all $0 \leq t < T$},
  \end{equation*}
  Whenever $\mu < \lambda$, this shows in particular that
  $\|u(.,t)\|_X \leq (2 C + 1) K_0$ on $[0,T)$, which in fact shows
  that $T$ cannot be finite (otherwise, there would be a slightly
  larger $T'$ for which $\|u(.,t)\|_X \leq (4 C+1) K_0$ on $[0,T']$,
  contradicting the definition of $T$).
  
  Now it only remains to show that for $|b|$ small enough
  one can choose $\mu$ so that $\mu < \lambda$ for all $d \geq
  0$. First, notice that for $\mu := \lambda - \frac{\lambda}{d}$
  equation \eqref{eq:mu-condition} reads
  \begin{equation*}
    C_2 |b| e^{\lambda} \left(
      \lambda - \frac{\lambda}{d}
    \right)^{\alpha-1}
    \leq
    \frac{1}{2 \Gamma(1-\alpha)},
  \end{equation*}
  which is valid for $|b|$ small enough (depending also on $K$) and
  all $d \geq 2$. For all $0 \leq d \leq 2$ it is clear that we can
  also choose $|b|$ small enough, independent of $d \in [0,2]$, for
  which \eqref{eq:mu-condition} holds with $\mu := \lambda/2$. Hence
  the statement is proved.
\end{proof}

\subsection{Linearized stability implies nonlinear stability}
\label{sec:linearized-to-nonlinear}

Given an equilibrium of the nonlinear equation, the linearized
  equation close to $p_\infty$ is given by
\begin{equation*}
  \p_t u
  =
  \p_v \left[(v - b N_\infty) u \right]
  + \p^2_v u  + \delta_{V_R} N_u(t)
  - b N_u(t-d) \p_v p_\infty(v)
  .
\end{equation*}
The link between this and the solution of the nonlinear system is that
we expect $p \approx u + p_\infty$ as long as the solution $p$ is
close to the equilibrium $p_\infty$. When there is no delay ($d=0$) we
denote by $T$ the operator on the right hand side of the previous
equation:
\begin{equation}
  \label{eq:linearized-op}
   T u := \p_v \left[(v - b N_\infty) u \right]
  + \p^2_v u  + \delta_{V_R} N_u(t)
  - b N_u(t) \p_v p_\infty(v).
\end{equation}
We will assume that the PDE $\p_t u = Tu$, with boundary condition
$u(V_F, t) = 0$ and an initial condition $u_0 \in X$ is well posed,
and we denote the associated solution by $e^{tT} u_0$. As far as we
know there are no previous results on this, and we realize that this
is a gap that needs to be filled in order to make the theory fully
rigorous. However, assuming a solution exists which satisfies the
usual Duhamel's formula, it is easy to give \emph{a priori} estimates
which could then be used to give a full proof of existence of
solutions. We will not develop this existence theory in this paper
since it would add a long technical part, but we will show two
straightforward a priori estimates. Call $u(t, \cdot) = e^{tT}
u_0$. Since $T u = L u - b N_u \p_v p_\infty$, and assuming we may
apply Duhamel's formula, we have
\begin{equation}
  \label{eq:Duhamel-linearized}
  u(\cdot,t) = e^{t L} u_0
  + \int_0^t N_u(s) e^{(t-s) L} (\p_v p_\infty) \d s.
\end{equation}
Taking the $\|\cdot\|_X$ norm and using equations \eqref{eq:gap-X} and
\eqref{eq:gap-X-regularized} from Proposition \ref{prp:gap-X} gives
\begin{multline*}
  \|u(\cdot,t)\|_X
  \leq
  \| e^{t L} u_0\|_X
  + |b| \int_0^t |N_u(s)| \| e^{(t-s) L} (\p_v p_\infty)\|_X \d s
  \\
  \lesssim
  e^{-\lambda t} \| u_0\|_X
  + |b| \| \p_v p_\infty \|_{L^2(\varphi)}
  \int_0^t \| u(\cdot,s) \|_X (t-s)^{-\frac34} e^{-\lambda(t-s)} \d s,
\end{multline*}
valid for all $t \geq 0$. Calling
$z(t) := e^{\lambda t} \|u(\cdot,t)\|_X$ we have
\begin{equation*}
  z(t)
  \lesssim
  \| u_0\|_X
    + \int_0^t z(s) (t-s)^{-\frac34} \d s.
\end{equation*}
Lemma \ref{lem:gronwall-continuous} now shows that there exists $C >
0$ and $t_0 > 0$ such that
\begin{equation*}
  \|u(\cdot,t)\|_X
  \leq C \| u_0\|_X
  \qquad \text{for all $0 \leq t \leq t_0$.}
\end{equation*}
This is an apriori estimate on the well-posedness of the linearized
equation. Similarly, we can obtain a regularization estimate: taking
again the $\|\cdot\|_X$ norm in \eqref{eq:Duhamel-linearized} and
using now Lemma \ref{lem:regularity} we have
\begin{multline*}
  \|u(\cdot,t)\|_X
  \lesssim
  t^{-\frac34} \| u_0\|_{L^2(\varphi)}
  + |b| \| \p_v p_\infty \|_{L^2(\varphi)}
  \int_0^t \| u(\cdot,s) \|_X (t-s)^{-\frac34} e^{-\lambda(t-s)} \d s
  \\
  \lesssim
  t^{-\frac34} \| u_0\|_{L^2(\varphi)}
  +
  \int_0^t \| u(\cdot,s) \|_X (t-s)^{-\frac34} \d s.
\end{multline*}
Lemma \ref{lem:gronwall-continuous} now shows that there exist $C > 0$
and $t_0 > 0$ such that
\begin{equation}
  \label{eq:linearized-regularization}
  \|u(\cdot,t)\|_X
  \leq C t^{-\frac34} \| u_0\|_{L^2(\varphi)}
  \qquad \text{for all $0 \leq t \leq t_0$.}
\end{equation}
This regularization estimate will be needed for the proof of the
next result.

\begin{thm}[Linearized stability implies nonlinear
  stability]
  \label{thm:convergence-linearized-to-nonlinear}
  Let us consider $ b\in\R $, $V_R< V_F \in \R$, a nonnegative initial
  condition $p_0\in X$ and $p$ the corresponding solution to the
  nonlinear system \eqref{eq:NNLIF}. Let $p_\infty$ be an equilibrium
  of the nonlinear system. Assume that the flow $(e^{tT})_{t \geq 0}$
  associated to the linearized operator $T$ given in
  \eqref{eq:linearized-op} has a spectral gap of size $\lambda > 0$ in
  the space $X$; that is, there exist $\lambda > 0$ and $C \geq 1$
  such that for all $u_0 \in X$ with
  $\int_{-\infty}^{V_F} u_0(v) \d v = 0$ we have
  \begin{equation}
    \label{eq:sg-hyp}
    \| e^{tT} u_0 \|_X \leq C e^{-\lambda t} \|u_0\|_X
    \qquad \text{for all $t \geq 0$.}
  \end{equation}
  Then there exists $\epsilon > 0$ (depending on $b$, $V_R$, $V_F$,
  $p_\infty$ and $\lambda$) such that if
  $\|p_0-p_\infty\|_X \leq \epsilon$ we have
  \begin{equation}
    \|p(.,t) - p_\infty(\cdot) \|_X
    \le 2 C e^{-\frac{\lambda t}{2}} \|p_0-p_\infty\|_X
    \qquad \text{for all $t\ge 0$.}
  \end{equation}
\end{thm}

The above result is given for systems without delay. Though we believe
the analogous result holds for systems with delay, in that case there
are technical difficulties in carrying out the same argument (in
particular in stating the Duhamel's formula). We also point out that
from the proof below one can easily obtain a decay rate with
$e^{-(\lambda-\delta) t}$ for any $\delta > 0$ instead of
$e^{-(\lambda t) / 2}$, but we have chosen the latter for readability.

\begin{proof}[Proof of Theorem \ref{thm:convergence-linearized-to-nonlinear}]
  The proof goes along the lines of the one for the Theorem
  \ref{thm:convergence-b-small}, using this time our assumption on the
  linearized operator. As a preliminary step, we notice that due to
  \eqref{eq:sg-hyp}, the semigroup $e^{tT}$ satisfies
  \begin{equation}
    \label{eq:gap-T-regularized}
    \| e^{tT} u_0 \|_X \leq C t^{-\frac34} e^{-\lambda t}
    \|u_0\|_{L^2(\varphi)}
  \end{equation}
  for all $u_0 \in X$, all $t > 0$ and some $C > 0$, $\lambda >
  0$. For $0 < t \leq t_0$ this is a consequence of
  \eqref{eq:linearized-regularization}; and for $t > t_0$ we can carry
  out the same argument as in the proof of
  \eqref{eq:gap-X-regularized}.

  We call $ u(v,t):=p(v,t)-p_\infty(v) $ with
  $N_u(t) := N_p(t)-N_{p_\infty}$, and we write the nonlinear equation
  \eqref{eq:NNLIF-PDE}
  as
  \begin{multline*}
    \p_tu
    =
    \p^2_vu + \p_v(v u) + N_u(t) \delta_{V_R}
    - b N_{p_\infty} \p_vu
    -bN_u(t) \p_vp_\infty - bN_u(t) \p_vu
    \\
    =: Tu(v,t)-bN_u(t)\p_vu,
  \end{multline*}
  where $e^{tT}$ refers to the semigroup of the linearized equation
  (with no delay). By Duhamel's formula we have
  \begin{equation*}
    u(\cdot,t) = e^{t T} u_0
    - b \int_0^t N_u(s) e^{(t-s) T} (\p_v u(\cdot,s)) \d s.
  \end{equation*}
  Taking the $\|\cdot\|_X$ norm, using \eqref{eq:sg-hyp},
  \eqref{eq:gap-T-regularized} and the fact that
  $\| \p_v u \|_{L^2(\varphi)} \leq \| u \|_X$ we have
  \begin{multline*}
    \|u(.,t)\|_X
    \leq
    C e^{-\lambda t} \| u_0 \|_X
    +
    C|b|\int_0^t \| u(.,s)\|_X e^{-\lambda(t-s)}
    (t-s)^{-3/4}\|\p_v u(.,s)\|_{L^2(\varphi)} \d s.
    \\
    \leq
    C e^{-\lambda t} \| u_0 \|_X
    +
    C|b|\int_0^t e^{-\lambda(t-s)} (t-s)^{-3/4}\|u(.,s)\|_X^2 \d s.
  \end{multline*}
  Let us now take
  $K_1 := 4 C \|u_0\|_X$ and define
  \begin{equation*}
    T_0 := \sup \{t \geq 0 \mid \|u(\cdot, s)\|_X \leq K_1 \ \text{for $s \in
      [0,t]$} \}.
  \end{equation*}
  This $T_0$ must be strictly positive, and it may be $+\infty$; we will
  show later that with an appropriate choice of $b$ it must be
  $+\infty$. Defining $ C_b := C|b|K_1 $ we get
  \begin{equation}\label{eq:duhamel-norm-linearized-to-nonlinear}
    \|u(.,t)\|_X \le C e^{-\lambda t} \| u_0 \|_X
    +
    C_b\int_0^t e^{-\lambda(t-s)} (t-s)^{-3/4}\|u(.,s)\|_X \d s.
  \end{equation}
  We define now $f(t):=e^{\lambda t}\|u(.,t)\|_X$ and write the
  previous equation as
  \begin{equation*}
    f(t) \le C f(0)
    +
    C_b\int_0^t (t-s)^{-3/4}f(s) \d s.
  \end{equation*}
  Using the modified Gronwall's Lemma
  \ref{lem:gronwall-modified} with $A=C, B=C_b$ and
  $\alpha = 3/4 $ we find the following bound for $f(t)$:
  \begin{equation*}
    f(t) \le
    2 C f(0) \exp( C_b^4 \mu t ),
  \end{equation*}
  which in terms of $ \|u(.,t)\|_X $ is written as
  \begin{equation}
    \label{eq:nsp1}
    \|u(.,t)\|_X \le 2 C
    \| u_0 \|_X \exp( -(\lambda - C_b^4\mu) t ).
  \end{equation}
  As in the proof of Theorem \ref{thm:convergence-b-small}, this shows
  that $T_0 = +\infty$ whenever $C_b^4\mu \leq \lambda$, that is, when
  \begin{equation*}
    4 C^2 |b| \|u_0\|_X  \leq \lambda^{\frac14} \mu^{-\frac14}.
  \end{equation*}
  The latter condition is satisfied if $\|u_0\|_X$ is small enough. If
  we additionally take $\|u_0\|_X$ so that
  \begin{equation*}
    4 C^2 |b| \|u_0\|_X  \leq 2^{-\frac14} \lambda^{\frac14} \mu^{-\frac14}
  \end{equation*}
  then $\lambda - C_b^4\mu \leq \lambda/2$ and eq. \eqref{eq:nsp1}
  shows that
  \begin{equation*}
    \|u(.,t)\|_X \le 2 C
    \| u_0 \|_X \exp( -\frac{\lambda}{2} t )
    \qquad \text{for all $t \geq 0$.}
    \qedhere
  \end{equation*}
\end{proof}

\section{Spectral gap of the linearized equation}
\label{sec:linearized-sg}

Consider equation \eqref{eq:NNLIF} with delay $d \geq 0$; that is,
\begin{equation}
  \label{eq:nlif}
  \p_t p
  = \p_v\left[(v - bN_p(t-d)) p\right]
  + \p^2_vp + \delta_{V_R} N_p(t),
\end{equation}
where $p=p(v, t)$ is the unknown and
$N_p(t) := -\p_v p(v,t) \big\vert_{v = v_F}$. We recall that the
corresponding linear equation studied in Section \ref{sec:linear}
corresponds to the case in which $N_p(t)$ is fixed to a given constant
in the nonlinear term:
\begin{equation}
	\label{eq:linear2}
	\p_t u =
	\p_v\left[(v - bN) u \right] + \p_v^2 u
        + \delta_{V_R} N_u(t) := Lu(v,t),
\end{equation}
where it is understood that
$N_u(t) := -\p_v u(v,t) \big\vert_{v = V_F}$. The \emph{linear}
equation is inherently an equation without any delay. On the other
hand, we may also consider the \emph{linearized} version of equation
\eqref{eq:nlif}: if $p_\infty$ is a stationary state of this equation
(nonnegative, with integral equal to $1$), we may search for solutions
``close to $p_\infty$'', of the form
$p(v,t) = p_\infty(v) + \epsilon u(v,t)$; for small $\epsilon$, $u$
must approximately solve the linearization of this partial
differential equation close to $p_\infty$, given by
\begin{equation*}
	\p_t u
	=
	\p_v \left[(v - b N_\infty) u \right]
	+ \p^2_v u  + \delta_{V_R} N_u(t)
	- b N_u(t-d) \p_v p_\infty
	,
\end{equation*}
that is,
\begin{equation}
	\label{eq:linearized}
	\p_t u
	=
	L_\infty u - b N_u(t-d) \p_v p_\infty
	,
\end{equation}
where $L_\infty$ is the linear operator \eqref{eq:linear2}
corresponding to $N = N_\infty := -\p_v p_\infty(V_F)$. The linearized
equation \eqref{eq:linearized} inherits the mass-preservation property
from the nonlinear equation \eqref{eq:nlif}, and it is natural to
always assume that $\int_{-\infty}^{V_F} u(t,v) \d v = 0$. Equation
\eqref{eq:linearized} is fundamental in the study of the behavior of
the nonlinear equation \eqref{eq:nlif} close to equilibrium. The
purpose of this section is to give conditions under which solutions to
\eqref{eq:linearized} converge to $0$ as $t \to +\infty$, always
assuming that $v \mapsto u(t,v)$ has integral zero for all $t$.

Our first observation is that assuming we know the solution to the
linear equation \eqref{eq:linear2}, equation \eqref{eq:linearized}
becomes a closed equation for $N_u(t) := -\p_v u(t, V_F)$. To see
this, consider first the case with delay $d=0$. By the Duhamel's
formula, the solution to \eqref{eq:linearized} is given by
\begin{equation*}
  u(v,t)= e^{t L_\infty} u_0
  - b \int_0^t N_{u}(s) e^{(t-s) L_\infty} (\p_v  p_\infty) \d s.
\end{equation*}
Taking the derivative at $v = V_F$ we obtain
\begin{equation}
  \label{eq:8-1}
  N_u(t)
  =
  g(t)
  + \int_0^t N_u(s) h(t-s) \d s,
\end{equation}
where we define
\begin{equation}
  \label{eq:defhg}
  h(t) := - b N ( e^{t L_\infty} \p_v p_\infty),
  \qquad
  g(t) := N(  e^{t L_\infty} u_0 ).
\end{equation}
Equation \eqref{eq:8-1} is a type of integral equation sometimes known
as the \emph{renewal equation} or \emph{Volterra's integral equation of
  the second kind}. It can be solved and its asymptotic behavior as
$t \to +\infty$ can be characterized in terms of the Laplace transform
of $h$; see Section \ref{sec:volterra}.

Now, in the case of a positive delay $d > 0$, we consider an initial
condition $u_0 \in \mathcal{C}([-d,0], X)$. Duhamel's formula applied
to equation \eqref{eq:linearized} gives
\begin{equation}
  \label{eq:9}
  u_t = e^{t L_\infty} u_0
  - b \int_0^t N_{u}(s-d) e^{(t-s) L_\infty} (\p_v  p_\infty) \d s,
\end{equation}
where it is understood that $e^{t L_\infty} u_0$ means
$e^{t L_\infty}$ applied to the initial condition $u_0(0)$. Taking the
derivative at $v = V_F$ we obtain
\begin{equation*}
  N_u(t)
  =
  g(t)
  + \int_0^t N_u(s-d) h(t-s) \d s,
\end{equation*}
where we define $h$ and $g$ just as before. We now rewrite this
equation in a form closer to the convolution equation
\eqref{eq:8-1}. For $t \geq d$ we have
\begin{equation*}   
  N_u(t)
  =
  g(t)
  + \int_0^{d} N_u(s-d) h(t-s) \d s
  + \int_0^{t-d} N_u(s) h(t-s-d) \d s.
\end{equation*}
If we define $h(t) := 0$ for $t < 0$, the above equation can be
written as
\begin{equation*}   
  N_u(t)
  =
  g(t)
  + \int_0^{d} N_u(s-d) h(t-s) \d s
  + \int_0^{t} N_u(s) h(t-s-d) \d s,
\end{equation*}
and this can actually be checked to hold for all $t \geq 0$. Hence we
have
\begin{equation}
  \label{eq:10}
  N_u(t)
  =
  g_1(t)
  + \int_0^{t} N_u(s) h(t-s-d) \d s,
\end{equation}
where
\begin{equation*}
  g_1(t)
  :=
  g(t)
  + \int_0^{d} N_u(s-d) h(t-s) \d s.
\end{equation*}
Notice that the values of $N_u$ which appear in the definition of
$g_1$ are only for negative times, and are hence given by the initial
condition of the delay equation for $u$.

\subsection{Proof of Theorem \ref{thm:N-asymptotic-behavior-main}}

A study of the integral equations \eqref{eq:8-1} and \eqref{eq:10} (see
Theorem \ref{thm:asymptotic} in Section \ref{sec:volterra}) leads to
Theorem \ref{thm:N-asymptotic-behavior-main}, which we prove now:

\begin{proof}[Proof of Theorem \ref{thm:N-asymptotic-behavior-main}]
  First, assume that the zeros of $\Phi_d$ (if any) are all on the
  real negative half-plane. Since $\Phi_d$ has a finite number of
  zeros on any strip with bounded real part, we know then that there
  exists $\lambda_0 > 0$ such that $F$, the function which appears in
  Theorem \ref{thm:asymptotic}, does not have any poles with real part
  $\geq -\lambda_0$. Let $u$ be any solution to the linearized
  equation \eqref{eq:linearized} with delay $d \geq 0$, with initial
  data $u_0 \in X$ (if $d=0$) or $u_0 \in \mathcal{C}([-d,0], X)$ (if
  $d > 0$). In both cases it is enough to show
  \eqref{eq:linearized-stability-def} in the definition of linearized
  stability when $u_0(\cdot, 0)$ is a regular function (say
  $\mathcal{C}^2$), since the general case can be obtained by
  density. For the same reason, we may assume $N[u_0(\cdot, t)]$ is a
  $\mathcal{C}^1$ function on $[-d, 0]$.

  The firing rate $N_u$ associated to the solution $u$
  satisfies the renewal-type equation \eqref{eq:10} (which is
  \eqref{eq:8-1} in the particular case $d=0$). The functions $g_1$ and
  $h(\cdot-d)$ satisfy the hypotheses of Theorem \ref{thm:asymptotic}
  (with $g_1$ playing the role of $g$ in its statement):
  \begin{enumerate}
  \item The function $g$ satisfies
    $|g(t)| = |N( e^{t L_\infty} u_0 )| \leq \| e^{t L_\infty} u_0\|_X
    \leq C e^{-\lambda t} \| u_0 \|_X$ due to Theorem
    \ref{prp:gap-X}. Hence the bound \eqref{eq:c2} is satisfied. Since
    we may assume $u_0(0)$ is regular enough, it is known from
    \cite{carrillo2013classical} that $g$ is $\mathcal{C}^1$ on
    $[0,+\infty)$, so it is continuous and of bounded variation on
    compact intervals of $[0,+\infty)$.
    
  \item Since we have assumed the function $N_u$ to be $\mathcal{C}^1$
    on $[-d,0]$, the integral term in the definition of $g_1$ is
    easily seen to be of bounded variation on compact sets, and
    satisfies \eqref{eq:c2}.
    
  \item Due to \eqref{eq:gap-X-regularized} in Proposition
    \ref{prp:gap-X}, the function $h$ can be bounded by
    $|h(t)| \leq C t^{-\frac34} e^{-\lambda t}$ for some
    $C, \lambda > 0$, so $h$ satisfies \eqref{eq:c1}.
  \end{enumerate}
  We note that $\Phi_d$ is precisely the denominator of the function
  $F$ which appears in Theorem \ref{thm:asymptotic}, so we have proved
  that $F$ does not have any poles with real part $\geq
  \max\{-\lambda, -\lambda_0\}$. Then the theorem shows
  that, for some $C \geq 1$ and $\lambda > 0$,
  \begin{equation*}
    |N_u(t)| \leq C e^{-\frac{\lambda}{2} t} \|u_0\|_X
    \qquad \text{for $t \geq 0$.}
  \end{equation*}
  (Or any positive constant strictly smaller than $\lambda$ instead of
  $\lambda/2$.) This shows the decay of $N_u$ as $t \to +\infty$. The
  decay of $\|u(\cdot, t)\|_X$ is obtained by taking the $\|\cdot\|_X$
  norm in eq. \eqref{eq:9} and using the decay of $N_u$. This gives
  the spectral gap result, since the constants $C$ and $\lambda$ can
  be taken to depend only on $h$.

  Reciprocally, assume that there is a zero $\xi_0$ of $\Phi_d$ with
  $\Re(\xi_0) \geq 0$. One can choose an initial condition $u_0$ on
  $[-d,0]$ such that $\xi_0$ is a pole of the function $F$ in Theorem
  \ref{thm:asymptotic}. For the solution associated to this initial
  condition, Theorem \ref{thm:asymptotic} shows that $N_u(t)$ does not
  decay exponentially as $t \to +\infty$, so the equation does not
  have a spectral gap in this case.
\end{proof}

\subsection{Proof of Theorem \ref{thm:stability-main}}

\begin{proof}[Proof of Theorem \ref{thm:stability-main}]
  \textbf{Proof of point 1.} Due to Lemma \ref{lem:ddtN}, the
  inequality in point 1 means that
  \begin{equation*}
    \hat{h}(0) = \int_0^\infty h(t) \d t > 1.
  \end{equation*}
  Hence the function $\Phi_d(\xi)$ satisfies $\Phi_d(0) = 1 -
  \hat{h}(0) < 0$, it is real and analytic for $\xi \in [0,+\infty)$, and satisfies
  $\lim_{\xi \to +\infty, \ \xi \in \R} \Phi_d(\xi) = 1$. Hence
  $\Phi_d$ must have a zero at some $\xi_0 \in \R$ with positive real
  part, and by Theorem \ref{thm:N-asymptotic-behavior-main} the linearized
  equation \eqref{eq:linearized} must have a solution whose associated
  $N$ diverges to $+\infty$. Hence $p_\infty$ is linearly unstable.
  
  \medskip
  \noindent
  \textbf{Proof of point 2.} In the second case it holds that, for
  purely imaginary 
  $\xi = ik$ (with $k \in \R$) we have
  \begin{equation*}
    |\hat{h}(\xi) \exp(-ikd)|
    = \left|
      \int_0^\infty h(t) e^{-i k t} \d t
    \right|
    \leq \int_0^\infty |h(t)| \d t < 1.
  \end{equation*}
  Hence, the function $\hat{h}$ is analytic on the real positive
  half-plane
  $\C_{\Re \geq 0} := \{ \xi \in \C \mid \Re(\xi) \geq 0 \}$, has
  modulus strictly less than one on the imaginary axis, and converges
  to $0$ as $|\xi| \to +\infty$ on $\C_{\Re \geq 0}$. Hence $\hat{h}$
  never has modulus $1$ on $\C_{\Re \geq 0}$, and as a consequence the
  function $\Phi_d$ does not have any zeros on $\C_{\Re \geq 0}$. By
  Theorem \ref{thm:N-asymptotic-behavior-main}, $p_\infty$ is linearly
  stable.

  \medskip
  \noindent
  \textbf{Proof of point 3.}  In order to use Theorem
  \ref{thm:N-asymptotic-behavior-main} to prove that $p_\infty$ is an
  unstable equilibrium for large $d$, we will find a root
  $\xi \in \C_{\Re \geq 0}$ of the function $\Phi_d$ (that is, a root
  with $\Re(\xi) > 0$). Equivalently, we need to find
  $\xi \in \C_{\Re \geq 0}$ with $\hat{h}(\xi) e^{-\xi d} = 1$. We
  will first show (in steps 1-4) that there exists $\xi \in \C$ with
  $\hat{h}(\xi) e^{-\xi d} = 1$ and $\Re(\xi) \geq 0$, and then we
  will slightly modify the argument (in step 5) to find a root $\xi$
  with strictly positive real part.
  
  \medskip
  \noindent
  \textbf{Step 1: Link to the winding number of a certain curve.}  Due
  to Lemma \ref{lem:ddtN}, the inequality
  \eqref{eq:conditional-instability-condition-intro} means that
  \begin{equation*}
    \hat{h}(0) = \int_0^\infty h(t) \d t < -1.
  \end{equation*}
  For any $y \geq 0$, consider the curve
  $\alpha_y \: [0, +\infty) \to \C$ defined by
  \begin{equation*}
    \alpha_y(k) := \hat{h}(y + i k) \exp(-(y +i k) d)
    \qquad \text{for $k \in \R$.}
  \end{equation*}
  This curve satisfies $\alpha_y(0) = \hat{h}(y) e^{-y d}$ and
  $\lim_{k \to +\infty} \alpha_y(k) = 0$. We may complete this curve to
  a closed, continuous curve $\beta_y \: [0,2] \to \C$ by joining
  these endpoints, for example by defining
  \begin{equation*}
    \beta_y(x) :=
    \begin{cases}
      \alpha_y(\tan(\frac{\pi x}{2})) \qquad & \text{if $x \in [0,1)$,}
      \\
      (x-1) \alpha_y(0) \qquad & \text{if $x \in [1,2]$.}
    \end{cases}
  \end{equation*}
  Assume the opposite of what we want to prove; that is, assume that
  \begin{equation}
    \label{eq:contradiction}
    \hat{h}(\xi) e^{-\xi d} \neq 1
    \qquad
    \text{for all $\xi \in \C_{\Re \geq 0}$.}  
  \end{equation}
  Equivalently, the curve $\alpha_y$ never touches the point $1$, for
  any $y \geq 0$. As a further consequence, notice that
  $\alpha_y(0) = \hat{h}(y) e^{-y d}$ must always be strictly less
  than $1$ for $y \geq 0$, since it is real and continuous, it is less
  than $-1$ for $y=0$, and it never touches $1$. Hence, the curves
  $\beta_y$ never touch the point $1$.

  Since the curves $\beta_y$ give a homotopy from the curve $\beta_0$
  to the single point $0$, all curves $\beta_y$ must have winding
  number $0$ about the point $1$. Our proof by contradiction is
  complete if we prove that $\beta_0$ has nonzero winding number about
  $1$ for $d$ large.
  
  \medskip
  \noindent
  \textbf{Step 2: For large $d$, all crossings of $\beta_0$ with
    $(1,+\infty)$ are negatively oriented.} That is, we will show that
  any time $\beta_0$ crosses the line $R := (1,\infty)$ in the complex
  plane, it does so by traversing it from above. Since
  $\alpha_0(0) < -1$, it is clear that $\beta_y(x)$ never touches $R$
  for $x \in [1,2]$, so we only need to study the way in which the
  curve $\alpha_0$ crosses $R$. We have, for $k > 0$,
  \begin{equation*}
    \ddk \alpha_0(k) = i \hat{h}'(ik) e^{-i k d} - i d \alpha_0(k).
  \end{equation*}
  Notice that for all $k \geq 0$,
  \begin{equation*}
    |\hat{h}'(ik)|
    =
    \left| \int_0^\infty i t\, h(t) e^{-ik t} \d t
    \right|
    \leq
    \int_0^\infty t\, |h(t)| \d t =: M.
  \end{equation*}
  Hence, if $\alpha_0(k) > 1$ at a certain point $k > 0$ then
  \begin{equation*}
    \Im\left( \ddk \alpha_0(k) \right)
    \leq
    |\hat{h}'(ik)| - d \alpha_0(k)
    \leq M - d \alpha_0(k),
  \end{equation*}
  where $\Im(\cdot)$ denotes the imaginary part. If $d \geq M$ then
  this is a negative number, and hence the crossing is negatively
  oriented. 

  \medskip
  \noindent
  \textbf{Step 3: For large $d$, the curve $\beta_0$ crosses the line
    $(1,+\infty)$ at least once.} This is easy to see, since the curve
  $\gamma(k) := \alpha_0(k/d)$ converges to the curve
  $\alpha_0(0) e^{-ik}$, uniformly in compact sets. Since
  $\alpha_0(0) < -1$ by assumption, the curve $\gamma$ must cross the
  interval $(1, +\infty)$ for large enough $d$, hence the same is true
  of $\alpha_0$.
 
  \medskip
  \noindent
  \textbf{Step 4: The winding number of $\beta_0$ is nonzero for $d$
    large.}  The crossing number can be equivalently defined as the
  total number of crossings of the curve $\beta_0$ with $(1,+\infty)$,
  counting each one with its orientation. Since $\beta_0$ crosses
  $(1,+\infty)$ at least once, and all crossings are negatively
  oriented, this means that its winding number around $1$ must be
  negative. This completes the contradiction, and shows that for $d$
  large, $\Phi_d$ must have at least one root on the positive complex
  half plane $\C_{\Re \geq 0}$.

  \medskip
  \noindent
  \textbf{Step 5: End of the proof.} It only remains to show that (for
  $d$ large) there must actually be a root of $\Phi_d$ with strictly
  positive real part. But all the previous steps can be carried out
  with the slightly perturbed curve $\alpha_y$, instead of $\alpha_0$,
  for small $y$, hence obtaining the conclusion.  
\end{proof}

In Figure \ref{fig:integral-h} we see a numerical description of the
different cases in Theorem \ref{thm:stability-main},
depending on the connectivity parameter $ b $. This gives a global
description of the long-time behavior of the solutions to the
linearized equation \eqref{eq:linearized}. However, in cases where the
value of the delay is key to the stability of the equilibria, no
quantitative results are shown with respect to the delay. For this
purpose we have made Figure \ref{fig:map}, where the stability or
instability of equilibria to the linearized equation is shown in a
two-dimension map $(b,d)$.

We end the section by proving the relationship between $\hat{h}(0)$
and $\frac{d}{dN}\left(\frac{1}{I(N)}\right)_{|N=N_\infty}$ used in
proof of Theorem \ref{thm:stability-main}.
\begin{lem}
  \label{lem:ddtN}
  Take $b \in \R$, and let
  $p_\infty \: (-\infty, V_F] \to [0,+\infty)$ be a nonnegative
  equilibrium with unit mass of the nonlinear integrate-and-fire
  equation \eqref{eq:nlif}. The following
  equality holds:
  \begin{equation}
    \hat{h}(0)=bN(L^{-1}_\infty(\partial_v p_\infty))=
    \frac{d}{dN}\left(\frac{1}{I(N)}\right)_{|N=N_\infty}.
  \end{equation}
  (We recall that $I \: [0, \infty) \to \R^+$ was defined in
  \eqref{eq:I} and $h(t) = -b N_q(t)$ was defined in
  \eqref{eq:defhg}.)
\end{lem}
\begin{proof}
  \noindent
  \textbf{Proof of the first equality.} Due to our well-posedness
  inequalities in Section \ref{sec:linear} we have
  \begin{equation*}
    \| e^{t L_\infty} u_0 \|_X \lesssim C e^{\mu t} \|u_0\|_X
  \end{equation*}
  for some $C, \mu > 0$, where $e^{t L_\infty} u_0$ denotes the
  classical solution to the linear problem with a sufficiently regular
  initial condition $u_0$. By density we may then extend the map
  $e^{t L_\infty}$ to all of $X$, defining a strongly continuous
  semigroup in this space. Due to mass conservation, this semigroup
  may be restricted to the space $X_0$ of functions in $X$ with
  integral $0$, which we do in the following. By a general result on
  semigroups \cite[Theorem 1.10]{Engel1999} we have
  \begin{equation*}
    L_\infty^{-1} (u_0) = - \int_0^\infty e^{t L_\infty} u_0 \d t
  \end{equation*}
  for any $u_0 \in X_0$. Applying the operator $N$ (which is
  continuous in $X$) to both sides,
  \begin{equation*}
    N \left[
      L_\infty^{-1} (u_0)
    \right]
    = - \int_0^\infty N\left[ e^{t L_\infty} u_0  \right] \d t
  \end{equation*}
  for all $u_0 \in X_0$. Approximating $\p_v p_\infty$ by functions in
  $X_0$ and multiplying by $b$ we obtain
  \begin{equation*}
    b N \left[
      L_\infty^{-1} ( \p_v p_\infty )
    \right]
    = - b \int_0^\infty N\left[ e^{t L_\infty} \p_v p_\infty  \right] \d t
    = \int_0^\infty h(t) \d t = \hat{h}(0).
\end{equation*}

  \medskip
  \noindent
  \textbf{Proof of the second equality.}
  We denote $\phi \equiv L^{-1}_\infty(\partial_v p_\infty)$, so
  $\phi\: (-\infty, V_F] \to [0,+\infty)$ is a function which satisfies
  $$
  \partial_v p_\infty=L_\infty(\phi),
  \ \phi(V_F)=0, \ N_\phi:=-\phi'(V_F),
  \ \mbox{and} \int_{-\infty}^{V_F} \phi(v) dv=0,
  $$
  or equivalently, $\phi(V_F)=0$, $N_\phi:=-\phi'(V_F)$,
  $\int_{-\infty}^{V_F} \phi(v) dv=0$ and
  $$
  \p_v p_\infty(v)=\phi''(v)+(v\phi)'(v)-bN_\infty\phi'(v)+\delta_{V_R}.
  $$
  Therefore,
  $$
  \left( \phi'+(v-bN_\infty)\phi+H(v-V_R)N_\phi-p_\infty
  \right)'(v)=0
  $$
  and  using the boundary condition $\phi(V_F)=0$, we obtain
  $$
  \phi'(v)=-(v-bN_\infty)\phi(v)-H(v-V_R)N_\phi+p_\infty(v),
  $$
  which can be solved and we get, in terms of $N_\phi$,
  $$
  \phi(v)=\frac{N_\phi}{N_\infty}p_\infty-N_\infty e^{-\frac{(v-bN_\infty)^2}{2}}
  \int_v^{V_F}\int_{w}^{V_F}e^{\frac{(y-bN_\infty)^2}{2}}H(y-V_R) \d y \d w.
  $$
  To determine $N_\phi$ we integrate in $(-\infty, V_F)$ and we impose
  that the integral is zero
  \begin{eqnarray}
    N_\phi & = & N_\infty^2\int_{-\infty}^{V_F}
                    e^{-\frac{(v-bN_\infty)^2}{2}}
                    \int_v^{V_F}\int_{w}^{V_F}e^{\frac{(y-bN_\infty)^2}{2}}H(y-V_R)\d y \d w \dv
                    \nonumber
    \\ 
              & = &
                    N_\infty^2\int_{-\infty}^{V_F}
                    e^{-\frac{(v-bN_\infty)^2}{2}}
                    \int_v^{V_F}
                    e^{\frac{(y-bN_\infty)^2}{2}}H(y-V_R)
                    \int_{v}^{y}\d w\d y \dv
                    \nonumber
    \\
              & = &
                    N_\infty^2\int_{-\infty}^{V_F}
                    e^{-\frac{(v-bN_\infty)^2}{2}}
                    \int_v^{V_F}
                    e^{\frac{(y-bN_\infty)^2}{2}}H(y-V_R) (y-v)\d y \dv.
                    \nonumber
  \end{eqnarray}
  So we have proved
  \begin{equation}
    N(L^{-1}_\infty(\partial_v p_\infty))=N_\phi =N_\infty^2\int_{-\infty}^{V_F}
    e^{-\frac{(v-bN_\infty)^2}{2}}
    \int_v^{V_F}
    e^{\frac{(y-bN_\infty)^2}{2}}H(y-V_R) (y-v)\d y \dv.
    \label{Nvarphi}
  \end{equation}
  To conclude the proof we calculate
  $ \frac{d}{dN}\left(\frac{1}{I(N)}\right)_{|N=N_\infty}$. To do that we recall
  the definition of $I$ (see \eqref{eq:I})
  \begin{equation}
    I(N)= \int_{-\infty}^{V_F}
    e^{-\frac{\left(v-bN\right)^2}{2}}
    \int_v^{V_F}e^{\frac{\left(y-bN\right)^2}{2}} H(y-V_R)\d y  \dv.
    \label{IN}
  \end{equation}
  and write $g(N):=\frac{1}{I(N)}$, and we want to prove that
  \begin{equation}
    g'(N_\infty)=b N_\phi.
    \label{g'}
  \end{equation}
  We know that $N_\infty I(N_\infty)=1$, because
  $N_\infty$ is the firing rate of the equilibrium
  (see \eqref{eq:stationary-state}), thus
  \begin{equation}
    g'(N_\infty)=\frac{-I'(N_\infty)}{I(N_\infty)^2}=
    -I'(N_\infty)N_\infty^2.
    \label{g'bis}
  \end{equation}
  We differentiate $I(N)$
  \begin{eqnarray}
    I'(N)  & = & 
                 b\int_{-\infty}^{V_F}
                 (v-bN)	e^{-\frac{\left(v-bN\right)^2}{2}}
                 \int_v^{V_F}e^{\frac{\left(y-bN\right)^2}{2}} H(y-V_R) \d y  \dv
                 \nonumber
    \\
           & &- b \int_{-\infty}^{V_F}
               e^{-\frac{\left(v-bN\right)^2}{2}}
               \int_v^{V_F}(y-bN)e^{\frac{\left(y-bN\right)^2}{2}} H(y-V_R) \d y  \dv      
               \nonumber
    \\
           &  =  & 
                   b
                   \int_{-\infty}^{V_F}
                   v\, e^{-\frac{\left(v-bN\right)^2}{2}}
                   \int_v^{V_F}e^{\frac{\left(y-bN\right)^2}{2}} H(y-V_R) \d y  \dv
                   \nonumber
    \\
           & & - b \int_{-\infty}^{V_F}
               e^{-\frac{\left(v-bN\right)^2}{2}}
               \int_v^{V_F}y\,e^{\frac{\left(y-bN\right)^2}{2}} H(y-V_R) \d y  \dv      
               \nonumber
    \\
           &  =  & 
                   b
                   \int_{-\infty}^{V_F}
                   e^{-\frac{\left(v-bN\right)^2}{2}}
                   \int_v^{V_F}e^{\frac{\left(y-bN\right)^2}{2}} H(y-V_R)(v-y) \d y  \dv,
                   \nonumber
  \end{eqnarray}
  and evaluating in $N_\infty$  and using \eqref{g'bis} we obtain \eqref{g'},
  which concludes the proof.  
\end{proof}

\section*{Acknowledgements}

\thanks{\em The authors acknowledge support from grant
  PID2020-117846GB-I00, the research network RED2022-134784-T, and the
  María de Maeztu grant CEX2020-001105-M from the Spanish government.}

\bibliographystyle{acm}
\bibliography{neurons}

\end{document}